\documentclass[12pt]{amsart}

\usepackage[utf8]{inputenc}
\usepackage{graphicx}
\usepackage{graphicx,color}
\usepackage{latexsym}
\usepackage[all]{xy}
\usepackage{mathrsfs}
\usepackage{enumerate,enumitem}
\usepackage{multicol}
\usepackage{lipsum}
\usepackage{amssymb}
\usepackage{tikz}
\usepackage{float}
\usepackage{bm}
\usepackage{mathptmx}
\usetikzlibrary{shapes.geometric}
\usepackage{dsfont}
\usepackage{cjhebrew}
\usepackage[alpine]{ifsym}
\usepackage{xpicture}
\usepackage{calculator}
\usepackage{graphicx}
\usepackage{pgf,tikz,pgfplots}
\pgfplotsset{compat=1.15}
\usetikzlibrary{arrows}

\usepackage{amsmath}
\usepackage{amsfonts}
\usepackage{amssymb,enumerate}
\usepackage{amsthm}
\usepackage{fancyhdr}
\usepackage{hyperref}
\usepackage{cleveref}
\usepackage{caption}
\usepackage{subcaption}
\usepackage[backend=biber,style=alphabetic, maxbibnames=99,doi=false,isbn=false,url=false]{biblatex}
\renewbibmacro{in:}{%
 \ifentrytype{article}{}{\printtext{\bibstring{in}\intitlepunct}}}

\theoremstyle{plain}
\newtheorem{theorem}{Theorem}[section]
\newtheorem{proposition}[theorem]{Proposition}

\newtheorem*{theorem*}{Theorem}
\newtheorem{lemma}[theorem]{Lemma}

\theoremstyle{definition}

\newtheorem{notation}[theorem]{Notation}

\newtheorem{ex}[theorem]{Example}

\theoremstyle{remark}
\newtheorem{rem}[theorem]{Remark}

\newcommand{\vu}{\underline{v}}
\newcommand{\tu}{\underline{t}}
\newcommand{\con}{\varsigma}
\newcommand{\obeta}{\overline{\beta}}
\newcommand{\osigma}{\overline{\sigma}}
\newcommand{\pr}{\mathrm{pr}}
\newcommand{\ind}{\mathtt{I}}
\newcommand{\jnd}{\mathtt{J}}

\counterwithin{figure}{section}
\numberwithin{equation}{section}

\title[Plane curve singularities and algebraic knots]{The intrinsic topological nature of the Poincaré series of a plane curve singularity}

\author{Patricio Almir\'on}

\author{Julio-Jos\'e Moyano-Fern\'andez}



\thanks{The first author is supported by Spanish Ministerio de Ciencia, Innovaci\'{o}n y Universidades PID2020-114750GB-C32 and by the IMAG–Maria de Maeztu grant CEX2020-001105-M / AEI /10.13039/501100011033, through a postdoctoral contract in the ‘Maria de Maeztu Programme for Centres of Excellence’. The second author was partially supported by
	MCIN/AEI/10.13039/501100011033 and by ``ERDF -- A way of making Europe", grant PGC2018-096446-B-C22, as well as by Universitat Jaume I, grants UJI-B2021-02 and GACUJIMA/2023/06.}

\address{Instituto de Matemáticas\\
Universidad de Granada\\
18001, Granada, Spain.}
\email{patricioalmiron@ugr.es}

\address{Universitat Jaume I, Campus de Riu Sec, Departamento de Matem\'aticas \& Institut Universitari de Matem\`atiques i Aplicacions de Castell\'o, 12071
Caste\-ll\'on de la Plana, Spain}

\email{moyano@uji.es}

\begin{document}

\begin{abstract}
In this paper we provide some factorization theorems of the Poincaré series $P_C$  of a plane curve singularity $C$ depending on some key values of the semigroup of values of \(C\). These results yield an iterative computation of $P_C$ in purely algebraic terms from the dual resolution graph of $C$. On the other hand, Campillo, Delgado and Gusein-Zade showed in 2003 the equality between $P_C$ and the Alexander polynomial $\Delta_L$ of the corresponding link $L$. Our procedure supplies a new proof of this coincidence. More concretely, we show that our algebraic construction can be translated to the iterated toric structure of the link $L$. Additionally we show that the semigroup algebra can be defined from the fundamental group of the link exterior in the irreducible case. This gives in particular a conceptual reason for the coincidence of $P_C$ and $\Delta_L$.

\end{abstract}

\maketitle

\tableofcontents

\section{Introduction}

Let $C$ be a germ of complex plane curve singularity with \(r\geq 1\) branches. Campillo, Delgado and Kiyek \cite{CDKmanuscr} attached a series 

\[
P_C(\underline{t})=P_C(t_1,\ldots, t_r)=\frac{\displaystyle\prod_{i=1}^r (t_i-1) \cdot\bigg(\displaystyle\sum_{\underline{w} \in \mathbb{Z}^r_{\geq 0}} \dim_{\mathbb{C}}J(\underline{w})/J(\underline{w}+\underline{1})\cdot  \underline{t}^{\underline{w}}\bigg)}{t_1\cdots t_r -1},
\]

where, for every $\underline{w}=(w_1,\ldots, w_r)\in \mathbb{Z}^r$, the ideal \(J_C(\underline{w})=J(\underline{w}):=\{g\in\mathcal{O}~:~\vu(g)\geq \underline{w}\}\) defines a multi-index filtration associated to the valuation \(\vu=(v_1,\dots,v_r)\) at the local ring \(\mathcal{O}:=\mathcal{O}_C\) of \(C\), and $\underline{1}=(1,\ldots , 1)$ and $\underline{t}^{\underline{w}}:=t_1^{w_1}\cdots t_r^{w_r}$. Observe that $P_C(\underline{t})$ is formal power series if $C$ is irreducible, i.e. if $r=1$, and a polynomial if $r>1$. The dimensions of the $\mathbb{C}$-vector spaces $J(\underline{w})/J(\underline{w}+\underline{1})$ are finite and depend on the value semigroup $\Gamma(C)=\{\vu(g): g\in \mathcal{O}, g \neq 0 \}$ of $C$, see e.g. \cite[(3.5)]{MFjpaa}.
\medskip

The interest of this series---called for brevity the Poincar\'e series of $C$---became apparent when Campillo, Delgado and Gusein-Zade \cite{CDG99a} proved its coincidence with the zeta function of the monodromy transformation of an irreducible singularity. They developed a research line to compute the Poincaré series through some invariants of the singularity, not only for plane curves \cite{CDG99a, CDG03a, CDGduke, CDG07} (even in a motivic setting \cite{CDG07}; see also \cite{Mams}) but also for rational surface singularities \cite{CDG04} and curves on them \cite{CDG05}. In these papers, they proposed alternative definitions for $P_C(\underline{t})$ involving techniques of integration with respect to the Euler characteristic, which led to an A'Campo type formula \cite[Theorems 3 and 4]{CDGduke} in terms of the dual graph \(G(C)\) of the minimal embedded resolution of the singularity, namely

\[
P_C(\underline{t})=
\prod_{Q \in G(C)} (\underline{t}^{\underline{v}^Q}-1)^{-\chi (E_Q^{\circ})}. 
\]

This formula à la A'Campo seamlessly blends information which can be read either from the topology or from the algebraic point of view, namely: From the topological side, the Euler characteristic $\chi (E_Q^{\circ})$ of the smooth part of the irreducible component $E_Q$ of the exceptional divisor created in the resolution process and, from the algebraic side, the valuation $\underline{v}^Q$ of the points $Q$ of the dual graph \(G(C).\) Moreover, the well established relation between $\underline{v}^Q$ and the linking invariants of the algebraic link \(L:=C\cap S^3_\varepsilon\) in the \(3\)--sphere \(S^3_\varepsilon\) with radius \(\varepsilon >0\) small enough, allowed them to apply a result by Eisenbud and Neuman \cite[Theorem 12.1]{EN} in order to deduce the connection between \(P_C(\underline{t})\) and the Alexander polynomial $\Delta_L(\underline{t})$ of $L$: 
\[P_C(\underline{t})=\Delta_L(\underline{t})\quad\text{if}\quad r>1\quad\text{and}\quad (t-1)\cdot P_C(t)=\Delta_L(t)\quad \text{if}\quad r=1.\]

\medskip

However, it seems as though this outcome is merely a fortuitous occurrence resulting from two a priori unrelated mathematical entities; paraphrasing the own authors, ``up to now this coincidence has no conceptual explanation. It is obtained by direct computations of both objects in the same terms and comparison of the results" \cite[p.~450]{CDG15}; see also \cite[pp.~271--272]{CDGDocumenta}. To date, this sentence is still valid and one of the aims of this paper is to provide a new proof of this coincidence which proposes a conceptual explanation for it.
\medskip


The main contribution of this paper is to give a purely algebraic proof of some factorization theorems of the Poincar\'e series of \(C\) depending on key values of the semigroup $\Gamma(C)$ which can be read off from the dual resolution graph associated to $C$. The description of the link associated to $C$ as an iterated cabling operation naturally produces a Mayer-Vietoris decomposition which allows to prove decomposition theorems of the Alexander polynomial as in Eisenbud-Neumann \cite{EN} and Sumners-Woods \cite{SW}. Thus, our results can be understood as the algebraic analogues of those purely topological results; as a consequence, they provide a natural setting to extend our results to more general situations from the valuative point of view.

\subsection{Summary of our approach} 

The value semigroup of an irreducible plane curve singularity is a complete intersection numerical semigroup, which means that it can be constructed from a process called gluing \cite{Delormegluing} (see also Section \ref{subsec:iterativepoincare}). The link $C$ is in this case an iterated torus knot. The main idea behind our approach is to realize that the gluing construction of the value semigroup of an irreducible plane curve singularity mimics the construction leading to the description of \(L\) as an iterated torus knot. More concretely, the satellization process describing the iterated toric structure of \(L\) is in one-to-one correspondence with the gluing construction in the value semigroup. Thus, one can see the algebraic operation as a topological operation. This fusion between algebra and topology gives us the hint that a purely algebraic recursive computation of the Poincaré series may provide the conceptual explanation to its coincidence with the Alexander polynomial. Therefore, our main result provides a purely algebraic recursive computation of the Poincaré series. 
\medskip

Our starting point is then to set aside the application of the Eisenbud-Neumann Theorem \cite[Theorem 12.1]{EN} and deepen on the algebraic calculations. To do so, the \textbf{first step} (Theorem \ref{thm:p1p2p3}) is to write the Poincaré series as a product

\begin{equation}\tag{\(\ast\)}\label{eqn:star1}
    P_C(\underline{t})=\frac{1}{\underline{t}^{\underline{v}^{\mathbf{1}}}-1} \cdot  \prod_{i=1}^q \frac{\underline{t}^{\underline{v}^{\sigma_{i}}}-1}{\underline{t}^{\underline{v}^{\rho_{i}}}-1}\cdot (\underline{t}^{\underline{v}^{\sigma_0}}-1)\cdot \prod_{\rho \in \widetilde{\mathcal{E}}}  \frac{\underline{t}^{(n_{\rho}+1)\underline{v}^{\rho}}-1}{\underline{t}^{\underline{v}^{\rho}}-1} \cdot \prod_{s(\alpha) >1}  (\underline{t}^{\underline{v}^{\alpha}}-1)^{s(\alpha)-1},
\end{equation}

where the factors depend on relevant vertices of the dual graph, which are the so-called star points and the fist vertex of the dual graph (cf. Subsection~\ref{subsub:dualgraph}). This yields a pure algebraic, valuative expression for the Poincar\'e series. 
\medskip

In view of the expression \eqref{eqn:star1}, the \textbf{second step} is to establish a suitable ordering of those relevant vertices of $G(C)$ which makes it possible to compute the Poincaré series in an iterative way. This ordering will base on a one-to-one correspondence between the star points of the dual graph and the topologically relevant exponents of the Puiseux series of the branches of the curve. Since we are interested in the topological properties of the curve, we first define a \emph{topological Puiseux series} for each of the branches, which provides us a simplified expression encoding the necessary information (Section \ref{subsec:topologicalpuiseux}). The ordering of the star points will only depend on the minimal generators of the value semigroups of every branch and the contact between the branches. As a consequence, this ordering yields a canonical ordering in the blowing-ups centers of the minimal embedded resolution of the plane curve. Moreover, this allows us to define a sequence of plane curves \(C_{\alpha_1},C_{\alpha_2},\dots,C_{\alpha_l}=C\) depending on the ordered star points (cf.~Subsection \ref{subsubsec:truncationstar}) which approximate \(C,\) so that the last curve in that sequence is \(C.\)
\medskip

Now, in a \textbf{third step}, we provide a method to compute the Poincar\'e series \(P_{\alpha_i}\) of each approximating curve $C_{\alpha_i}$ from the series of the previous approximating curve as products of the form 
\begin{equation}\tag{\(\dagger\)}\label{eqn:dag1}
P_{\alpha_i}(\bullet)=P_{\alpha_{i-1}}(\bullet) \cdot Q(\bullet) \cdot \prod B(\bullet) \cdot Q(\bullet),
\end{equation}
where the polynomials \(Q,B\) are defined in \eqref{eqn:defkeypoly} and the ``\(\bullet\)'' depend on the contact between branches and the minimal generators of the individual semigroups of the branches (cf.~Section \ref{subsec:iterativepoincare}). These polynomials arise naturally following the ordered multiplicity sequence provided by the suitable order of the blowing-ups centers of the minimal embedded resolution of the plane curve (Remark \ref{rem:QB1}). This recursive expression yields a purely algebraic procedure to compute iteratively the Poincaré series \(P_C(\underline{t})\). Moreover, this expression provides the announced decomposition theorems of the Poincaré series. Observe that up to this point, only algebraic methods have been used and they could be naturally extended to more general contexts.
\medskip

In the \textbf{fourth} and last \textbf{step}, we show that this purely algebraic, iterative construction can be translated step by step to the iterated toric structure of the link, confirming the guess motivated by the construction in the irreducible case. To do so, we show first that, in the irreducible case, we can construct the semigroup algebra associated to the semigroup of values from the fundamental group of the knot exterior (Section \ref{subsec:alexirreducible}). Second, we observe that our ordering in the set of branches is equivalent to the study of the algebraic link \(L\) from its innermost to its outermost component; it contrasts with the customary description of the Waldhausen decomposition, which uses to be formulated from outermost to inner (see also \cite{LeCarousels}). This is not at all surprising, as the study of the closed complement of the link ``from outer to inner'' is the natural point of view, appropriate to obtain such a decomposition. We recall here that the vertices of the plumbing diagram correspond to Seifert pieces in the Waldhausen decomposition of the link exterior \cite[Section 22]{EN}.
\medskip

The ordering from innermost to outermost is imposed by the algebraic construction, otherwise it is inconvenient to produce the iterative procedure to compute the Poincaré series. This is reasonable, since we are studying the algebraic properties of the link, and not \emph{a priori} those of its exterior. Here we find the ultimate reason to avoid the theorem of Eisenbud and Neumann: their splicing construction perfectly allows to describe the Waldhausen decomposition, hence this is closer to the perspective of the link from its exterior. A posteriori, one might check with a bit of effort that both procedures encode the same information, but the connection between the algebra and the topology will be lost.
\medskip

A crucial point in the fourth step is the work of Sumners and Woods \cite{SW}. They indicate a recursive way to compute the Alexander polynomial of \(L\) once the components are ordered as in our case. Then, we can show that each step in our procedure coincides with the topological description. As a consequence, we can provide a recursive proof of the coincidence between \(P_C(\underline{t})\) and \(\Delta_L(\underline{t})\) without using the results of Eisenbud-Neumann \cite{EN}, hence without topological guidance. Incidentally, thanks to our algebraic procedure we provide a more explicit expression of the Alexander polynomial in the case of more than three branches than the one given by Sumners and Woods \cite[Section VII]{SW}. Besides, our recursive argument has the advantage that it reveals an intrinsic topological nature of the algebraic operation, as the title of this paper aims at pointing out.

\subsection{Outline}

We will now indicate the parts of the manuscript where each of the above steps, as well as the necessary auxiliary results, are realized.
\medskip

Whereas Section 2 is devoted to introduce the main tools (and notation) which are imperative to understand the remainder of the article (Puiseux series, embedded resolutions of the curve and their dual graphs in Subsection \ref{subsec:dualgraph}, the Noether formula \ref{noehterformula}, and both the value semigroup and the extended semigroup in Subsection \ref{sec:semigroupvalues}), the technical core of the proofs in Section 4 lies in Section 3; observe that Section 4 encloses the first, and third steps explained above; the second step is cleared in Section 3, and the last step is given in Section 5.
\medskip

Indeed, in Subsections 3.1 and 3.3 we recall the main results about maximal contact values and the values of the star points attached to the dual graph $G(C)$ of the minimal embedded resolution of $C$. Subsection 3.2 is devoted to define the topological Puiseux series of the branches. The above mentioned second step is solved in Section 3.4: here we define the ordering on the set of star point in the dual graph and the sequence of approximating curves of $C$; in addition, this subsection gives a detailed description of how to construct the sequence of the approximating curves that will allow us to prove the iterative computation of their Poincar\'e series.
\medskip

Section 4 starts with the definition of the Poincaré series of the curve. After, we work out the first step, namely the proof of \eqref{eqn:star1} in Theorem \ref{thm:p1p2p3}. Subsection 4.2 addresses the step-by-step method of the recursive computation of the Poincaré series, cf.~(\ref{eqn:dag1}). First of all we recall the irreducible case in Proposition \ref{prop:poincareirreduciblecase}, and then we prove the two base cases, in which the only approximating curve is the curve itself (Proposition \ref{prop:poincarelemm1} and Proposition \ref{prop:poincarelemm2}). Eventually in Subsection \ref{subsec:generalprocedure} we present the general process, which completes the third step.
\medskip

Section 5 focuses on topology: We review the topological counterpart of the algebraic constructions of the previous sections. First we describe the process of satellization for the construction of an algebraic link attached to a curve. Second we present the gluing operation as a topological feature in the irreducible case. Then we describe its generalization to reducible curves following the exposition by Sumners and Woods \cite{SW}; here we supply further details missing in their exposition. 
\medskip

A closing, short section with historical remarks has been included for the convenience of the reader. 
\medskip


\subsection{General assumptions and notation}

We will denote by $\mathbb{N}$ the set of nonnegative integers. The cardinality of a finite set $A$ will be denoted by $|A|$. For an element $x$ of a ring $R$, we will write $(x)$ the principal $R$-ideal generated by $x$. 
\medskip

We understand for a \emph{curve} a germ of holomorphic function $f:(\mathbb{C}^2,0)\to (\mathbb{C},0)$ with isolated singular point at $0$. We will write $C:f=0$ and say that $C$ is a curve given by a power series $f\in \mathbb{C}\{x,y\}$, where $\mathbb{C}\{x,y\}$ stands for the ring of convergent power series (in two indeterminates $x$ and $y$). We will assume $f$ (hence $C$) to be reduced. If $f$ is not irreducible, the factorization $f=f_1\cdots f_r$ with $f_i\neq f_j$ if $i\neq j$ into irreducible germs induces irreducible curves $C_1, \ldots , C_r$ (given by the factors 
$f_1,\ldots , f_r$) called branches of $C$. We will write $C=\bigcup_{i=1}^{r} C_i$. We set $\mathtt{I}:=\{1,\ldots , r\}$.
\medskip

A parametrization of a branch $C:f=0$ at $0$ is given by power series $x(t), y(t) \in \mathbb{C}\{t\}$ such that $f(x(t), y(t)) = 0 \in \mathbb{C}\{t\}$ and, if $\tilde{x}(t), \tilde{y}(t)$ satisfy $f(\tilde{x}(t), \tilde{y}(t)) = 0$, then there is a unique unit $u \in \mathbb{C}\{t\}$ such that $\tilde{x}(t) = x(u \cdot t)$ and $\tilde{y}(t) = y(u \cdot t)$. We define therefore the intersection multiplicity of two branches as the total order of one curve on the parametrizations of the branches of the other curve, namely
\[
\big [C,\{g=0\}\big ]_0:= [f, g]_0: = \mathrm{ord}_t g(x(t), y(t)) = \mathrm{sup}\{m \in \mathbb{N} : t^m \ \mbox{divides} \ g(x(t), y(t))\};
\]

we will omit the dependence on $0$ in the writing if this is clear from the context.
\medskip

A branch of a curve can be parametrized by a Puiseux series, which we understand as a formal power series with rational exponents of the form

\[
y=\sum_{j>0} a_j x^{j/n},
\]

for $a_j \in \mathbb{C}$. We agree to have a sort of normal form for the Puiseux series by taking exponents with common denominator $n$ coprime to $\mathrm{gcd}\{j : a_j \neq 0\}$; this $n$ is called the polydromy order of the series.
\medskip

\noindent \textbf{Acknowledgements}. The authors wish to express their gratitude to Prof. F\'elix Delgado de la Mata for many stimulating conversations and helpful suggestions during the preparation of the paper.
\medskip


\section{Invariants associated to a resolution of plane curve singularities}\label{sec:invariantsofresolution}

Let $f \in \mathbb{C}\{x,y\}$ be an irreducible power series defining a plane branch \(C.\) The origin of the Puiseux series goes back to Newton: he looked for solutions of a polynomial equation $f(x,y)=0$ clearing $y$ as a function of $x$ and approximating successively; the $j$th step of his algorithm reads off as

\[
y=x^{q_1/p_1}(a_1+x^{q_2/p_1p_2}( \cdots + \cdots +x^{q_{j-2}/p_1\cdots p_{j-1}}(a_{j}+a_jx^{q_{j-1}/p_1\cdots p_j})\cdots )),
\]

where $a_j\in \mathbb{C}$ and $\mathrm{gcd}(p_j,q_j)=1$ for every $j$. The couples $(p_j,q_j)$ are called the Newton pairs of $C$. Puiseux extended the Newton method to reducible curves and condensed Newton's writing into a formal power series with fractional exponents: the approximations turned to be partial sums of a power series 

\[
y=b_1 x^{m_1/p_1} +b_2 x^{m_2/p_1p_2} +\cdots + b_j x^{m_j/p_1\cdots p_j} +\cdots
\]

The couples $(p_i,m_i)$ are called the Puiseux pairs of $C_i$ for every branch $C_i$ of $C$. They are related to the Newton pairs by the recursion $q_1=m_1, q_i=m_{i}-m_{i-1}p_i$. It is well known that there are only finitely many topologically meaningful terms in the Puiseux development, therefore we will assume without loss of generality that a branch has a Puiseux expansion of the form 

$$
y=b_1 x^{m_1/p_1} +b_2 x^{m_2/p_1p_2} +\cdots + b_k x^{m_k/p_1\cdots p_k}.
$$

\medskip

If we now consider a non-irreducible reduced power series \(f\in\mathbb{C}\{x,y\}\) defining a plane curve \(C\) with branches \(C_1,\dots,C_r\), then we will assume that all the Puiseux developments are of the form 
\[
y=s_i(x)=b_{1,i} x^{m_{1,i}/p_{1,i}} +b_{2,i} x^{m_{2,i}/p_{1,i}p_{2,i}} +\cdots + b_{k_i,i} x^{m_{k_i,i}/p_{1,i}\cdots p_{k_i,i}},
\]

for sufficiently large \(k_i\) with $i\in \texttt{I}$. This means that we may have some branches satisfying \(m_{1,i}/p_{1,i}<1\). Moreover, all along this paper we will assume an ordering in the (set of) branches of a plane curve as follows.
\medskip

Suppose that all the Puiseux developments coincide up to the term \(\alpha,\) so there exists at least one branch whose Puiseux development is different at the (\(\alpha+1\))-th term. Without loss of generality we have then 
\[
\frac{m_{\alpha+1,1}}{p_{\alpha+1,1}}\geq \frac{m_{\alpha+1,2}}{p_{\alpha+1,2}}\geq\cdots \geq\frac{m_{\alpha+1,r}}{p_{\alpha+1,r}}.
\]
Recursively, if \(\{j_1<j_2<\dots<j_s\}=:\jnd\subsetneq \ind\) with \(|\jnd|\geq 2\) and \(C_\jnd=\bigcup_{i\in\jnd}C\) is a subset of branches whose Puiseux developments coincide up to a term \(\alpha'>\alpha\) and at least one of them is different at \(\alpha'+1\), then 
\[
\frac{m_{\alpha'+1,j_1}}{p_{\alpha'+1,j_1}}\geq \frac{m_{\alpha'+1,j_2}}{p_{\alpha'+1,j_2}}\geq\cdots\geq \frac{m_{\alpha'+1,j_s}}{p_{\alpha'+1,j_s}}.
\]

\begin{ex}\label{example:goodorder}
  Observe that this ordering in the set of branches depends on the choice of the Puiseux series and it may vary on the topological class. Consider the curve \(C=\bigcup_{i=1}^{5} C_i\) where the Puiseux series of the branches \(C_i\) are 
  \(y_1=x^4,\) \(y_2=x^{5/2},\)  \(y_3=2x^2+x^{14/3},\) \(y_4=2x^2\) and \(y_5=x^2\) and  \(C'=\bigcup_{i=1}^{5} C'_i\) with Puiseux series 
  \(y'_1=x^4,\) \(y'_2=x^{5/2},\) \(y'_3=2x^2+x^{14/3},\) \(y'_4=2x^2+x^5\) and \(y'_5=x^2\). It is easy to see that \(C\) and \(C'\) are topologically equivalent, but the branches of \(C\) are good ordered and the branches of \(C'\) are not.
\end{ex}
      
  Since we are interested in the study of the topological class, we will introduce in Subsection \ref{subsec:totalorderstar} a refinement in the good order of the set of branches which is canonical in the equisingularity class of the curve and that does not depend on the choice of the Puiseux series of the branches.

\subsection{Embedded resolution of plane curves}\label{subsec:dualgraph}

Let $C=\bigcup_{i=1}^{r} C_i \subseteq (\mathbb{C},0)$ be a (complex) plane curve singularity given by the equation $f=0$, where $f=f_1\cdots f_r$ and $f_i$ is the equation of the branch $C_i$ for every $i=1,\ldots , r$. Consider a simple sequence of blowing-ups of points whose first center is $0=:p_1$:
\begin{equation}\label{seq}
\pi : \;\;\; \cdots \longrightarrow X_n \longrightarrow X_{n-1} \longrightarrow \cdots \longrightarrow X_1 \longrightarrow X_0=(\mathbb{C}^2,0);
\end{equation}
here \emph{simple} means that we only blow-up points $p_i$, for $i >1$, belonging to the exceptional divisor created last. The \emph{cluster of centers} of $\pi$ will be denoted by  $\mathcal{C}_{\pi}=\{p_1=0, p_2, \ldots \}.$ Recall that the pull-back \(\overline{C}=\pi^{\ast}(C)=(f\circ \pi)^{-1}(0)\) of \(C\) is called the total transform of \(C\) by \(\pi.\) The strict transform of \(C\) is defined to be  \(\widetilde{C}:=\overline{\pi^\ast(C\setminus \{0\})}\), where the bar denotes the Zariski closure. Each center \(p\in\mathcal{C}_\pi\) is assigned to an intersection multiplicity with the strict transform of \(C\) at \(p\). Recall also that the strict transform resp. the total transform of \(C\) at \(p\) is \(\widetilde{C}_p:=\overline{\pi_p^\ast(C\setminus \{0\})}\) resp. \(\overline{C}=\pi_p^{\ast}(C)\), where \(\pi_p\) denotes the blowing-up of \(p\). If we write \(E_0\) for the first neighborhood of \(0\), then the \(i\)-th neighborhood of \(0\) is defined as the set of points on the first neighborhood of any point on the \((i-1)\)-th neighborhood of \(0\) for any $i>1$. The points in any neighborhood of \(0\) are called points infinitely near to \(0\), and we denote the set of them by \(\mathcal{N}_0\). Also, \(\mathcal{N}_0\) is endowed with a natural ordering: we write \(p<q\) whenever \(q\) is infinitely near to \(p\). 
\medskip

For every $i\geq 1$ we denote by $E_i$ the exceptional divisor on $X_i$ obtained by blowing-up $p_i$. For $i > j$ we say that a point $p_i$ is proximate to $p_j$, written $p_i \rightarrow p_j$, if $p_i$ belongs to the strict transform of the exceptional divisor $E_j$ obtained by blowing-up $p_j$. We say that $p_i$ and $E_i$ are \emph{satellite} if $p_i$ is proximate to two points of $\mathcal{C}_{\pi}$; otherwise we say that they are \emph{free}.
\medskip

Every point in \(\mathcal{N}_0\) is associated to two integer values corresponding to the intersection multiplicities of both the strict and the total transform at that point. The multiplicity of \(C\) at a point \(p\in\mathcal{C}_\pi\) is defined to be \(e_p:=e_p(\widetilde{C}_p)\). On the other hand, we call the value of \(C\) at \(p\) to \(v_p(C):=e_p(\overline{C}_p)\). Observe that a curve goes through \(p\) if and only if \(e_p\geq 1\). We will denote by \(\mathcal{N}_0(C)\) the set of infinitely near points to \(0\) lying on \(C\). Obviously, the set \(\mathcal{N}_0(C)\) is an infinite set with finitely many satellite points and points with multiplicity strictly bigger than 1 (see \cite[Chapter~3]{casas} for a more detailed treatment).
\medskip

Moreover, thanks to the proximity relations \cite[Theorem~3.5.3]{casas} one can compute the values (and multiplicities) recursively with the aid of Noether's formula (cf.~\cite[Theorem~3.3.1]{casas}).

\begin{proposition}\label{noehterformula}
(Noether's Formula) Let $C_1,C_2$ be germs of curve in $0$. The intersection multiplicity $[C_1,C_2]_0$ is finite if and only if $C_1$ and $C_2$ share finitely many points infinitely near to $0$, and in such a case 

$$[C_1, C_2]_0=\sum_{p \in  \mathcal{N}_0(C_1)\cap\mathcal{N}_{0}(C_2)}e_p(C_1)e_p(C_2).$$

\end{proposition}

\subsubsection{The dual graph of the minimal embedded resolution of $C$}\label{subsub:dualgraph}

It is customary to provide the minimal embedded resolution in form of a weighted graph called the dual (resolution) graph \(G(C)\) of the embedded resolution. The vertices (or points) $P$ of $G(C)$ represent the components of the total transform $\overline{C}$ of the curve $C$ i.e. both the irreducible components $E_P$ of the exceptional divisor $E = \pi^{-1}(0)$ of $\pi$ and the strict transforms $\widetilde{C}_i$ of the branches $C_i$ of $C$; in the latter case they are depicted by arrows. Abusing of notation, we will write the vertex corresponding to the branch $C_i$ as $C_i$. Two vertices (or a vertex and an arrow) of $G(C)$ are connected by an edge if the corresponding components intersect. The resulting graph is an oriented tree, with starting point (i.e. the starting divisor of the resolution) denoted by $\mathbf{1}$. For $i\in \ind$, the geodesic in $G(C)$ joining $\mathbf{1}$ with the arrow corresponding to $f_i$ will be denoted by $\Gamma_i$. The set of vertices of the graph $G(C)$ can be endowed with a partial ordering: we set $P^{\prime} < P$ if and only if the geodesic in $G(C)$ from the vertex $\mathbf{1}$ to the vertex $P$ passes through the vertex $P^{\prime}$. Figure \ref{figA} shows how the dual graph of a reducible curve looks like. 
\medskip

The dual graph is labeled as follows. For $P \in G(C)$ with $P \neq \mathbf{1}$, let $\nu (P)$ be the number of vertices (or arrows) in $G(C)$ connected with $P$. The valence $\nu (\mathbf{1})$ may be different from $1$ in some special situation explained in a paragraph below. The points $P\in G(C)$ with $\nu (P)=2$ are called \emph{ordinary}, those with $\nu(P)=1$ are the \emph{end points} of the graph, and those points satisfying  $\nu(P)\geq 3$ are said to be star points; star points will be denoted by $\alpha$. The set of end points of $G(C)$ will be denoted by $\mathcal{E}$ (we let drop the dependence on $G(C)$ out in the notation for the sake of simplicity).
\medskip

\begin{figure}[h]
$$
\unitlength=0.50mm
\begin{picture}(120.00,110.00)(0,-30)
\thinlines
\put(-20,30){\line(1,0){40}}
\put(38,30){\line(1,0){22}}
\put(60,30){\circle*{2}}

\put(60,30){\line(1,1){40}}
\put(85,55){\line(1,-1){10}}
\put(85,55){\circle*{2}}
\put(100,70){\circle*{2}}

\put(100,70){\line(2,1){20}}
\put(120,80){\vector(0,1){10}}
\put(110,75){\line(1,-2){4}}
\put(100,70){\line(2,-1){20}}
\put(120,60){\vector(1,1){10}}
\put(110,65){\line(-1,-2){4}}

\put(60,30){\line(1,-1){40}}
\put(100,-10){\circle*{2}}
\put(85,5){\line(-1,-1){10}}
\put(85,5){\circle*{2}}
\put(86,6){{}}
\put(75,-5){\circle*{2}}
\put(73,-13){{}}

\put(100,-10){\line(2,1){20}}
\put(120,0){\vector(0,1){10}}
\put(110,-5){\line(1,-2){4}}
\put(100,-10){\line(2,-1){20}}
\put(120,-20){\vector(1,0){10}}
\put(134,-21){{\scriptsize$\tilde{C}_r$}}
\put(110,-15){\line(-1,-2){4}}
\put(102,-17){{}}
\put(110,-15){\circle*{2}}
\put(106,-23){\circle*{2}}
\put(104,-31){{}}

\put(60,30){\line(1,0){40}}
\put(85,30){\line(0,-1){13}}
\put(85,30){\circle*{2}}
\put(100,30){\circle*{2}}

\put(100,30){\line(2,1){20}}
\put(120,40){\vector(0,1){10}}
\put(110,35){\line(1,-2){4}}
\put(100,30){\line(2,-1){20}}
\put(120,20){\vector(1,1){7.5}}
\put(110,25){\line(-1,-2){4}}

\put(-20,30){\circle*{2}}
\put(-5,30){\line(0,-1){15}}
\put(-5,30){\circle*{2}}
\put(-5,15){\circle*{2}}
\put(15,30){\line(0,-1){20}}
\put(15,30){\circle*{2}}
\put(15,10){\circle*{2}}
\put(45,30){\line(0,-1){20}}
\put(45,30){\circle*{2}}
\put(45,10){\circle*{2}}

\put(60,30){\line(0,-1){15}}
\put(60,15){\circle*{2}}

\put(-27,28){{\scriptsize ${\bf 1}$}}
\put(-3,5,10){{\scriptsize$\rho_1$}}
\put(16.5,7){{\scriptsize$\rho_2$}}
\put(58,10){{\scriptsize$\rho_h$}}
\put(-6,33){{\scriptsize$\alpha_1$}}
\put(14,33){{\scriptsize$\alpha_2$}}
\put(25,30){$\ldots$}

\put(53,33){\scriptsize{$\sigma_0$}}
\end{picture}
$$
\caption{The dual graph $G(C)$ of a curve $C$.}
\label{figA}
\end{figure}
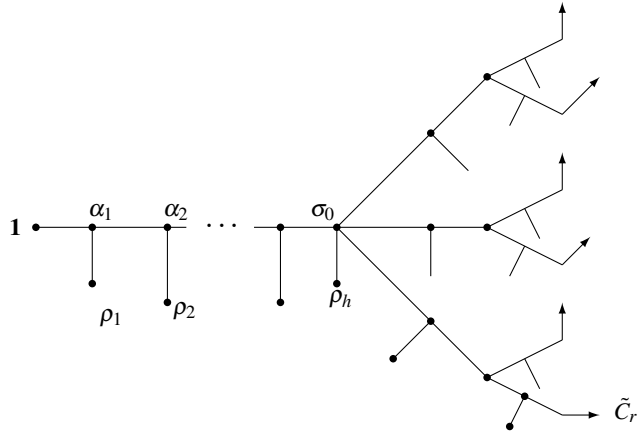

An \emph{arc} in the graph $G(C)$ is a sequence of vertices connected each other in a consecutive order and satisfying that all its vertices are ordinary up to the extremes, i.e. a geodesic joining two nonordinary points. A \emph{dead arc} (or \emph{tail}) of $G(C)$ is an arc with one end point. Hence the tail of $G(C)$ corresponding to the end vertex $L$ consists of all vertices $L^{\prime}$ with $\alpha_{L}< L^{\prime} \le L$. The set of dead arcs is denoted by $\mathcal{D}$. The extreme of a dead arc $L$ will be denoted by $P(L)$, and called dead end. For every end point $L$ of $G(C)$, there is a nearest star point $\alpha_{L}$ such that $\alpha_{L} < L$. 
\medskip

A star point $P$ is said to be \emph{proper} if it does not belong to any dead arc or it does with $\nu(P) \geq 4$. The set of proper star points will be denoted by $\mathcal{R}$, and a proper star point will be denoted by $\sigma$ (the letter $\alpha$ is kept for star points in general).
\medskip

A vertex $R$ is said to be a \emph{separation point} of the graph $G(C)$ if there exist two branches $C_i$ and $C_j$ of the curve $C$ such that $R < C_i$, $R < C_j$, and $R$ is the maximal vertex with these properties; we will also say that $R$ is the separation point between the branches $C_i$ and $C_j$. The first (i.e. minimal) separation point of $G(C)$ will be written $\sigma_0$; in other words, $\sigma_0$ is the last point in $\bigcap_{i=1}^{r} \Gamma_i$. Observe also that the separation points are proper star points. By convention, $\nu (\mathbf{1})$ is the number of edges (or arrows) incident in $P$ plus one, so that $\nu (\mathbf{1})\geq 2$, in the case $\mathbf{1}=\sigma_0$.
\medskip

We can describe the set $\mathcal{R}$ of proper star points in an alternative way: consider the ``smooth part'' $E^{\circ}_P$ of the component $E_P$ i.e. $E_P$ minus intersection points with other components of the total transform of the curve $C$. The cardinality of the set of connected components of the complement $(f \circ \pi)^{-1}(0) \setminus E^{\circ}_P$ is denoted by $s(P)$. Observe that

\begin{align*}
s(P)\geq 1& \Longleftrightarrow \ P \in  \{\sigma_0\} \cup \Big ( \bigcup_{i=1}^r \Gamma_i \setminus \bigcap_{i=1}^{r} \Gamma_i  \Big) \\
s(P)=0 & \Longleftrightarrow \ P\in \bigcup_{L\in \mathcal{D}} \big (L\setminus \{\alpha_L\} \}\big ) \cup \big (\bigcap_{i=1}^r \Gamma_i \setminus \{\sigma_0\}\big).
\end{align*}

\medskip

\medskip





The number $s(P)$ is related to $\nu (P)$ in the following manner: if $P\neq \sigma_0$ with $s(P)\geq 1$, then
\[
s(P)=\left \{
\begin{array}{ll}
  \nu(P)-1,   & \mbox{if there is no dead arc starting with } P, \\
  \nu(P)-2,   &  \mbox{otherwise}.
\end{array}
\right.
\]
Moreover, $s(\sigma_0)=\nu(\sigma_0)-2$ in the first situation, and $s(\sigma_0)=\nu(\sigma_0)-3$ in the second one.
\medskip

From this discussion it is easily deduced that the set of proper star points is
\[
\mathcal{R}=\{ P \in G(C) : s(P)>1\} \cup \{\sigma_0\}.
\]


\subsection{The semigroup of values of a plane curve}\label{sec:semigroupvalues}

Let \(C\) be a (reduced) germ of complex plane curve singularity with equation $f=\prod_{i=1}^{r}f_i=0$.
\medskip

For every branch \(C_i\) there is a discrete valuation \(v_i\) associated to the local ring \(\mathcal{O}_i:=\mathbb{C}\{x,y\}/(f_i)\) of the branch. This valuation can be defined as \(v_i(h):=[f_i,h]_0,\) the intersection multiplicity at the origin. Therefore, we have a multivaluation, say \(\vu,\) in the local ring of the plane curve \(\mathcal{O}:=\mathbb{C}\{x,y\}/(f)\) defined as \(\vu(h)=(v_1(h),\dots,v_r(h))\) for \(h\in\mathbb{C}\{x,y\}.\) The semigroup of values of \(C\) (or \(f\)) is the additive submonoid of \(\mathbb{N}^r\) defined by
\[
\Gamma(C):=\{\vu(h)=(v_1(h),\ldots,v_r(h))\in\mathbb{N}^r\; :\;h\in\mathcal{O},\,h\neq 0\};
\]

we write \(\Gamma=\Gamma(C)\) if no risk of confusion arises. The semigroup \(\Gamma\) has a conductor \(\con=\con(\Gamma),\) which is defined to be the minimal element of \(\Gamma\) such that $\gamma\in\Gamma$ whenever \(\gamma\geq \con\).
\medskip


\subsubsection{The irreducible case.} In the case of \(r=1\) we have a single branch, and the semigroup of values is a numerical semigroup whose minimal generating set is finite and can be computed from the characteristic exponents of \(f\), \cite{ZariskiModuli}. First of all, let us recall the definition of the Puiseux characteristics of the branch. The \emph{Puiseux characteristics} is a finite sequence of natural numbers defined as follows: let $(n)\subseteq \mathbb{Z}$ be the set of multiples of $n$, and set $\beta_1:=\min \{j : a_j \neq 0, j \notin (n)\}$, and recursively
\begin{align*}
e_{i-1} &=\mathrm{gcd}\{ n, \beta_1, \ldots , \beta_{i-1}\}>1\\
\beta_i &= \min \{j : a_j \neq 0, j \notin (e_{i-1})\} \ \ \mbox{for}\;i=1,\dots, g\;\mbox{and}\;e_g=1.
\end{align*}
    
The finite sequence given by the $e_i$ is called the $e$-sequence. According to these numbers, the Puiseux series of the branch can be decomposed as 

\[
s(x)=\sum_{\tiny\begin{array}{c}
		j\in (\beta_0)\\1\leq j<\beta_1
\end{array}}a_jx^{j/\beta_0}+\cdots+\sum_{\tiny\begin{array}{c}
		j\in (e_{i-1})\\\beta_{i-1}\leq j<\beta_i
\end{array}}a_jx^{j/\beta_0}+\cdots+\sum_{\tiny\begin{array}{c}
		j\in \mathbb{Z}\\\beta_{g}\leq j
\end{array}}a_jx^{j/\beta_0}.
\]

Assume that \(C\) has characteristic exponents \(\{\beta_0,\dots,\beta_g\}\) and write \(n_i=e_{i-1}/e_{i}\) for \(i=1,\dots,g.\) Let us define 
\begin{equation}\label{eq:definbetabarra}
\obeta_0=\beta_0,\;\obeta_1=\beta_1\quad\text{and}\quad \obeta_{i+1}=n_i\obeta_i+\beta_{i+1}-\beta_i\quad \mbox{for}\ 1\leq i\leq g.
\end{equation}

\begin{rem}
This recursion provides a relation between the elements \(\obeta_i\) and the Puiseux pairs.
\end{rem}

The semigroup \(\Gamma(C)\) is minimally generated by the elements \(\obeta_0,\dots,\obeta_g,\) i.e. 
\[
\Gamma(C)=\langle\obeta_0,\dots,\obeta_g\rangle=\big \{\gamma\in\mathbb{N}\;: \;\gamma=a_0\obeta_0+\cdots+a_g\obeta_g\ \ \mbox{with} \ \ a_i\in\mathbb{N},\ \ \mbox{for}\; \ i=0,\dots,g\big \}.
\]

It is customary to call the elements \(\obeta_0,\dots,\obeta_g\) the maximal contact elements (or values) of $C$; this terminology comes from the fact that they coincide with the intersection multiplicities of a certain truncation of the Puiseux series of \(C\), as we will see in Subsection \ref{subsec:maximalcontact}. 
\medskip

The main combinatorial properties of the semigroup of values of a plane branch are the following (see e.g. \cite{ZariskiModuli}:
\begin{enumerate}
    \item \label{prop1:irreducsemigroup}  \(n_i\obeta_i<\obeta_{i+1}\) for \(i=1,\dots,g-1.\)
    \item\label{prop2:irreducsemigroup} \(n_i\obeta_i\in\langle\obeta_0,\dots,\obeta_{i-1} \rangle\) for \(i=1,\dots,g.\)
    \item\label{prop3:irreducsemigroup} If \(\gamma\in\Gamma(C),\) then \(\gamma\) can be written in a unique way as \(\gamma=\sum_{i=0}^{g}a_i\obeta_i\) with \(a_0\geq 0\) and \(0\leq a_i\leq n_i-1\) for \(i=1,\dots,g\).
\end{enumerate}
In addition, the value semigroup \(\Gamma\) is symmetric, i.e. \(\gamma\in\Gamma\) if and only if \(\con-1-\gamma\notin \Gamma.\) The symmetry property is the combinatorial counterpart of the Gorenstein property of the local ring of the branch \cite{Kunz70}.

\subsubsection{The case of several branches} If \(r>1,\) the semigroup of values is no longer finitely generated, but it is finitely determined (see \cite{Delgmanuscripta1,CDGlondon}); moreover, \(\Gamma\subsetneq \Gamma(C_1)\times\cdots\times\Gamma(C_r)\).  Before continuing let us establish some notation:

\begin{notation}
    For an index subset \(\jnd\subset \ind=\{1,\dots,r\}\) we set \(f_\jnd:=\prod_{j\in \jnd}f_j\) and \(C_\jnd=\sum_{j\in \jnd}C_j\) for the plane curve with equation \(f_\jnd\).
	We denote by \(\pr_{\jnd}:\mathbb{N}^r\rightarrow\mathbb{N}^{|\jnd|}\) the projection on the indices of \(\jnd\), and \(\alpha_\jnd:=\pr_\jnd(\alpha)\). 
	 Finally, for each \(i\in \ind\) we write \(\Gamma^{(i)}:=\Gamma(C_i).\)
\end{notation}

We will assume that \(\mathbb{N}^r\) is partially ordered: For $\alpha=(\alpha_1,\dots,\alpha_r), \beta=(\beta_1,\dots,\beta_r)\in \mathbb{Z}^r$,
\[
\alpha\leq \beta\Longleftrightarrow\;\alpha_i\leq\beta_i\;\,\mbox{for all}\,\ i \in \ind.
\]

Some elementary properties of the semigroup of values \(\Gamma\) are the following (see \cite{Delgmanuscripta1}):
\begin{enumerate}
	\item If \(\alpha,\beta\in\Gamma\), then \[\alpha\wedge\beta:=\min\{\alpha,\beta\}:=(\min\{\alpha_i,\beta_i\})_{i\in \ind}\in\Gamma.\]
	\item If \(\alpha,\beta\in \Gamma\) and \(j\in \ind\) with \(\alpha_j=\beta_j\), then there exists an \(\epsilon\in \Gamma\) such that \(\epsilon_j>\alpha_j=\beta_j\) and \(\epsilon_i\geq\min\{\alpha_i,\beta_i\}\) for all \(i\in \ind\setminus\{j\}\), with equality if \(\alpha_i\neq\beta_j\). 
	\item The semigroup $\Gamma$ has a conductor, i.e. there exists an element $\varsigma\in \Gamma$ such that $\varsigma+\Gamma \subseteq \Gamma$.
\end{enumerate}

Now, for a given \(\alpha\in\mathbb{N}^r\) and an index subset $\jnd\subset \ind$, we set 
\[
\overline{\Delta}_\jnd(\alpha)=\big\{\beta\in\mathbb{N}^r\;:\;\beta_j=\alpha_j\quad\forall j\in \jnd\quad \text{and}\quad \beta_k>\alpha_k\quad\forall k\notin \jnd \big\},
\]
\[
\overline{\Delta}(\alpha)=\cup_{i=1}^{r}\overline{\Delta}_i(\alpha),\quad \Delta_\jnd(\alpha)=\overline{\Delta}_\jnd(\alpha)\cap \Gamma\quad\text{and}\quad \Delta(\alpha)=\overline{\Delta}(\alpha)\cap \Gamma.
\]

The sets \(\Delta(\alpha)\) are important in order to define those key elements of \(\Gamma\) which allow us to extend the symmetry property viewed in the irreducible case. An element \(\gamma\in \Gamma\) is called a maximal element of \(\Gamma\) if \(\Delta(\gamma)=\emptyset.\) If, moreover,  \(\Delta_\jnd(\gamma)=\emptyset\) for all \(\jnd\subset \ind\) such that \(\emptyset\neq \jnd\neq \ind\), then \(\gamma\) is said to be absolute maximal. On the other hand, if \(\gamma\) is a maximal and if \(\Delta_\jnd(\alpha)\neq \emptyset\) for all \(\jnd\subset \ind\) such that \(|\jnd|\geq 2\), then \(\gamma\) will be called relative maximal. It is easily checked that the set of maximal elements of \(\Gamma\) is finite.
\medskip

We would like to emphasize that, by definition, the element \(\gamma\in\Gamma\) is absolute maximal if and only if there exist absolute maximal elements \(\alpha\) and \(\beta\) such that \(\gamma=\alpha+\beta\). An absolute maximal element that cannot be decomposed as the sum of two nonzero elements of \(\Gamma\) is said to be an irreducible absolute maximal. As a consequence, any absolute maximal can be decomposed as a sum of irreducible absolute maximal elements. 
\medskip

To finish, observe that the semigroup \(\Gamma\) is also a symmetric semigroup \cite{Delgmanuscripta2} in the following sense: \(\gamma\in\Gamma\) if and only if $\Delta(\con-(1,\dots,1)-\gamma)=\emptyset$.

\subsection{The extended semigroup} The values $v_i(g)$ are orders of germs $g\circ \varphi_i$ at the origin of $\mathbb{C}^2$, where $\varphi_i$ denotes a parametrization of $C_i$ for every $i\in \ind$ and $g$ is an element of the ring $\mathcal{O}_{\mathbb{C}^2,0}$ of germs of holomorphic functions at the origin in $\mathbb{C}^2$; this allows us to write
\[
g\circ \varphi_i(t_i)=a_i(g)t_i^{v_i(g)}+\mathrm{terms~of~higher~degree}.
\]

Campillo, Delgado and Gusein-Zade \cite{CDGextended} considered an extension of the value semigroup $\Gamma(C)$ which they called the extended semigroup $\widehat{\Gamma}(C)=\widehat{\Gamma}$ associated to $C$: this is the subsemigroup of $\mathbb{N}^r \times (\mathbb{C}^{\ast})^r$ consisting of all tuples
\[
(\underline{v}(g),\underline{a}(g)):=\big (v_1(g),\ldots , v_r(g), a_1(g),\ldots , a_r(g)\big)
\]
for every $g\in \mathcal{O}_{\mathbb{C}^2,0}$ with $v_i(g)<\infty$ for all $i\in \ind$.
\medskip

They showed that the set $\mathbb{N}^r \times (\mathbb{C}^{\ast})^r$ may be endowed with the structure of semigroup. Although $\widehat{\Gamma}$ depends on the choice of the parametrizations of the branches $C_i$ of $C$, this is not an issue if one considers isomorphism classes induced by invertible changes of coordinates, as described in \cite[Remark 2]{CDGextended}.
\medskip

The relation between $\Gamma$ and $\widehat{\Gamma}$ is given by the surjective homomorphism $\pr:~\widehat{\Gamma} \to \Gamma$ defined by $\pr\big ((\underline{v}(g),\underline{a}(g))\big )=\underline{v}(g)$. The sets $F_{\underline{v}}:=\pr^{-1}(\underline{v})\subseteq \{\underline{v}\}\times (\mathbb{C}^{\ast})^r$ for every $\underline{v}\in \Gamma(C)$ are called fibres of $\widehat{\Gamma}(C)$. Therefore
\[
\widehat{\Gamma}=\bigcup_{\underline{v} \in \Gamma} F_{\underline{v}} \times \{\underline{v}\}.
\]

The extended semigroup will help to the understanding of the proof of Theorem \ref{thm:p1p2p3}.

\section{Star points: the topological guides in the dual graph}\label{sec:descriptionofstar}
In the case of an irreducible plane curve, the minimal generators of the semigroup are ordered by the inequalities \[\obeta_0<n_1\obeta_1<\obeta_2<n_2\obeta_2<\cdots<n_{g-1}\obeta_{g-1}<\obeta_g<n_g\obeta_g.\]
This ordering in the generators of the semigroup induces a total order in the star points in the dual graph and thus a canonical sequence of approximating curves. Unfortunately, in the non-irreducible case, since the semigroup is finitely determined, a priori there is no canonical order in the elements that determines the semigroup.
\medskip

Following the idea of the irreducible case, we will translate the problem of ordering the values determining the semigroup in terms of ordering the star points in the dual graph. In this section we will present a canonical total order of the star points of the dual graph of a plane curve with several branches. This constitutes the foremost algebraic tools in order to provide the iterative construction of the Poincaré series in Section \ref{sec:Poincareseries}. 
\medskip

The core of this section is to precisely describe a distinguished Puiseux series associated a plane curve $C$ expressed in terms of the star points of the dual graph \(G(C)\) of \(C.\) From this Puiseux series, we will provide a canonical total order in the star points \(G(C).\) This total order in the star points is key to provide an ordered sequence of approximating curves to $C$; these will play a central role in Subsection \ref{subsec:iterativepoincare}.

\subsection{The maximal contact }\label{subsec:maximalcontact}
We first recall the interpretation of both the minimal generators of the value semigroup of a branch and the irreducible absolute maximal elements ---for curves with several branches--- in terms of intersection multiplicities. In both cases, the values associated to those elements are called maximal contact values. We follow the exposition of Delgado \cite{Delgmanuscripta1, Delgadoari}.
\medskip

For an irreducible curve $C$, let \(C_q\) be the branch whose Puiseux series coincides with the following truncation of the Puiseux series of \(C\)

\[
\varphi_q(x)=\sum_{\tiny\begin{array}{c}
		j\in (\beta_0)\\\ 1\leq j<\beta_1
\end{array}}a_jx^{j/\beta_0}+\cdots+\sum_{\tiny\begin{array}{c}
		j\in (e_{i-1})\\\beta_{q-1}\leq j<\beta_q
\end{array}}a_jx^{j/\beta_0}.
\]

The germs \(C_q\) have the property that their intersection multiplicity with the curve \(C\) is exactly the \(q\)--th generator of the semigroup of values \mbox{\(\overline{\beta}_q:=[C_q,C]\)}. More generally, any \(\varphi\in\mathbb{C}\{x,y\}\) satisfying \([f,\varphi]=\obeta_q\) will be called a maximal contact element of genus \(q-1\) with \(f\).
\medskip

Consider now a plane curve \(C\) with \(r\) branches; for each \(i\in\ind\) let us denote by \(\{\beta^{i}_0,\dots,\beta^{i}_{g_i}\}\) the Puiseux exponents and \(\{\obeta^{i}_0,\dots,\obeta^{i}_{g_i}\}\) the maximal contact values of a branch \(C_i.\) Let us denote by \(\mathcal{B}^{i}_{n}\)  the set of curves having maximal contact of genus \(n\) with \(f_i\). We will say that \(\varphi\in\mathbb{C}\{x,y\}\) has maximal contact of genus \(n\) with \(f\) if the following two assertions hold:
\begin{itemize}
    \item[$\diamond$] The set \(J_\varphi=\{i\in\ind\;:\;\varphi\in\mathcal{B}^{i}_n\}\) is non empty. 
    \item[$\diamond$] \(J_\varphi\) is maximal for the inclusion ordering, i.e. there exists no branch \(\phi\) such that \(J_\varphi\subsetneq J_\phi.\)
\end{itemize}
The maximal contact values of a plane curve with \(r\) branches can be explicitly computed from the maximal contact values of each of its branches and their intersection multiplicities. For each \(n\in\mathbb{N}\) we set \(\mathcal{A}_n=\{\mathcal{B}^{i}_{n}\;:\;i\in\ind\}\) with the inclusion ordering. Define  
\[
\mathcal{M}_n=\{\jnd\subseteq \ind\;:\;\forall i,j\in\jnd\quad\mathcal{B}^{i}_{n}=\mathcal{B}^{j}_{n}\quad\text{and}\quad \mathcal{B}^{i}_{n}\quad\text{is minimal in}\;\mathcal{A}_n\};
\]
as in \cite[(3.7)]{Delgmanuscripta1}, given \(K\in\mathcal{M}_n\), for $\varphi\in \mathcal{B}^{K}_{n}$ we check that
\begin{equation}\label{eqn:maximalcontactvalues}
    v_j(\varphi)=\left\{\begin{array}{cl}
    \obeta^{j}_{n+1} &\text{if}\;j\in K \\[.3cm]
     \frac{[f_i,f_j]}{e^{i}_n}& \text{if}\;j\notin K\ \ \text{and}\ \ i\in K.
     
\end{array}\right.
\end{equation}

Write \(B_0=(\obeta^{1}_{0},\dots,\obeta^{r}_{0})\); the set of maximal contact values is
\begin{equation}\label{eq:b2}
\overline{\mathcal{B}}=\big \{\vu(\varphi)\;:\;\varphi\in\mathcal{B}^{\jnd}_n,\ \mbox{for} \ \jnd\in\mathcal{M}_n,\ \ n\in\mathbb{N}\big \}\cup \{B_0\}.
\end{equation}
With the above notation, the branches \(f_i\) of \(f\) appear as curves with maximal contact of genus \(g_i\) with \(f\), and obviously \(\vu(f_i)\notin \Gamma\) since \(v_i(f_i)=\infty\). Usually, we will use the maximal contact elements that are finite, i.e. 
\begin{equation}\label{eq:b}
\mathcal{B}=\left(\{\vu(\varphi)\;:\;\varphi\in\mathcal{B}^{\jnd}_n,\  \mbox{for}  \ \jnd\in\mathcal{M}_n,\ \ n\in\mathbb{N}\}\cap \Gamma\right)\cup \{B_0\}.
\end{equation}

The following result (\cite[(3.18)]{Delgmanuscripta1}) provides the announced identification of the maximal contact values and the irreducible absolute maximals of the semigroup of values.
\begin{theorem}[Delgado]
\(\gamma\in\Gamma\) is irreducible absolute maximal if and only if \(\gamma\) is a maximal contact value.
\end{theorem}

If the curve has only one branch, an irreducible absolute maximal element is nothing but a minimal generator of the semigroup of values. However, the generation of the semigroup of a curve with several branches from the irreducible absolute maximal elements is more subtle and one needs to exploit all the combinatorial properties of the semigroup (see \cite{Delgmanuscripta1} and \cite{CDGlondon}).

\subsubsection{The contact pair} 
According to \cite[(1.1.3),(1.1.4)]{Delgadoari}, the contact pair $(f\mid f')=(q,c)$ of two (irreducible) branches $f$ and $f'$ can be defined as follows. Let $\obeta'_1,\ldots , \obeta'_{g'}$ resp. $e'_0,\ldots , e'_{g'}$ be the maximal contact values resp. the $e$-sequence associated to the branch $f'$. Let $t$ be the minimum integer such that
\[
[f,f']\leq \mathrm{min}\big \{e'_t \obeta_{t+1},e_t \obeta'_{t+1} \big\}=p(t),
\]
setting $\obeta_{g+1}=\obeta'_{g'+1}=\infty$ and \(e_{-1}=e'_{-1}=0\) if necessary. Write $\ell_t$ resp. $\ell'_t$ for the integer part of $(\obeta_{t+1} - n_t\obeta_t)/e_t$ resp. of $(\obeta'_{t+1} - n'_t\obeta'_t)/e'_t$. We distinguish three complementary  cases:
\begin{itemize}
\item[$\diamond$] If $[f,f'] < p(t)$, then there exists an integer $c$ with $0 < c \leq \mathrm{min}(\ell_t,\ell'_t)$ such that
$$
[f,f']= e'_{t-1}\obeta_t + ce_te'_{t} = e_{t-1}\obeta'_t + ce_te'_t.
$$
In this case $q = t$ and $(f\mid f') = (t, c)$.
\item[$\diamond$] If $[f,f']=p(t)$ and $e'_t\obeta_{t+1}\neq e_t\obeta'_{t+1}$, then $(q,c)=(t,\mathrm{min}\{\ell_t+1,\ell'_t+1\})$. Recall that, if $[f, f'] = e'_t\obeta_{t+1} < e_t\obeta'_{t+1}$, then $\ell_t < \ell'_t$ and, if $\ell_q < \ell'_q$ then $[f,f'] = e'_t\obeta_{t+1} < e_t\obeta'_{t+1}$.
\item[$\diamond$] Finally, if $[f,f']=p(t)$ and $e'_t\obeta_{t+1}= e_t\obeta'_{t+1}$, then $(q,c)=(t+1,0)$.
\end{itemize}
\medskip

Similarly, if we have $r$ branches $f_1,\ldots, f_r$, then the contact pair of these is defined to be

$$
(f_1\mid f_2 \mid \cdots \mid f_r)=\mathrm{min}\{(f_i\mid f_j) : i \neq j \ \mbox{with} \ i,j\in \mathtt{I}\},
$$
where the minimum is understood to be with respect to the lexicographic ordering in $\mathbb{N}^2$. Obviously, for several branches this means that the Puiseux developments satisfy the previous conditions.
\begin{rem}
The contact pair measures the number of free infinitely near points shared by the branches. Recall that the coefficients of the Puiseux expansion can be seen as the projective coordinates of the free points of the branch (see \cite[Proposition 5.7.1]{casas}).
\end{rem}

\subsubsection{Maximal contact elements in terms of the dual graph}\label{subsec:maximalcontactdualgrapha}

For a point $P\in G(C)$, a curvette at $P$ is defined to be a smooth curve germ $\theta_P$ in the resolution space, transversal to the irreducible component \(E_P\) of the exceptional divisor in a regular point of the exceptional divisor $E$. If $\theta_P$ is given by an element $\varphi\in \mathbb{C}[\![X,Y]\!]$, and $\tilde{\varphi}$ stands for the strict transform of $\varphi$ by $\pi$, we say that $\varphi$ meets a subset $G$ of $G(C)$ if $\tilde{\varphi} \cap E_P$ is a regular point of $E$ for some $P\in G$; furthermore, if $\tilde{\varphi}$ meets $P\in G(C)$ and $\tilde{\varphi}$ is smooth, then we say that $\varphi$ becomes a curvette at $P$. 
\medskip

Since the dual graph $G(C_i)$ of a branch $C_i$ is known to have $g_i$ dead arcs, we write
\[
\mathcal{D}(C_i)=\{L_1^i,\ldots , L_{g_i}^i\},
\]

were the dead arcs are supposed to be ordered as $w(P(L_1^i))<w(P(L_2^i))<\cdots < w(P(L_{g_i}^i))$, were $w(P)$ is the number of blowing-ups needed to build the divisor $E_P$, for any $P\in G(C)$. Then the maximal contact values of the branches $f_i$ can be interpreted as intersection multiplicities as
\[
[f_i,\pi (\theta_{P(L_{j_i}^i)})]=\obeta_{j_i}^i
\]

for any $i\in \ind$ and $j_i\in \{0,\dots, g_i\}$. This means that $\varphi\in \mathbb{C}[\![X,Y]\!]$ has maximal contact of genus $n-1$ with $f_i$ if and only if $\varphi$ becomes a curvette at the end point $P(L^{i}_{j_i})$ of $L^{i}_{j_i}\in \mathcal{D}(C_i)$. Moreover, $\varphi$ becomes a curvette at the point $\mathbf{1}$ if and only if
\[
[f_i,\varphi]=\obeta_0^i,
\]

for any $i\in \mathtt{I}$. Following the exposition of Delgado \cite[(1.2.1)]{Delgadoari}, it can be shown that the maximal contact values can be realized by curvettes at the end points of $G(C)$ or at $\mathbf{1}$; indeed they can be read off in the dual graph as those elements in the set
\[
\mathcal{B}'=\{\vu\big ( \pi (\theta_{P(L)}) \big ) \in \Gamma : L \in \mathcal{D} \} \cup \{\vu (\pi(\theta_{\mathbf{1}}))\}.
\]

Observe that the points of $G(C)$ whose values are elements in $\mathcal{B}'$ are those in $\mathcal{E}\cup \{\sigma_0\}$. Moreover, they are ---by definition--- elements in the value semigroup $\Gamma$, hence $\mathcal{B}'=\mathcal{B}$ for the set $\mathcal{B}$ of eq.~(\ref{eq:b}).

\subsection{Topological Puiseux series}\label{subsec:topologicalpuiseux}
In our context, we are interested in the topological equivalence of plane curves. This equivalence relation can be described in terms of the semigroups of the branches and their intersection multiplicity as remarked by Zariski (see also \cite[Proposition 4.3.9]{wall}): \begin{proposition}\cite{zarsaturationII}\label{prop:equisingularity}
Let \(C,C'\) be curve singularities. They are equisingular if and only if the following conditions holds:
\begin{enumerate}
    \item There is a bijection \(C_i\leftrightarrow C_i'\) between the branches of \(C\) and \(C'\) such that \(\Gamma(C_i)=\Gamma(C_i')\), and
    \item For all \(i,j\in \ind\), we have  \([C_i,C_j]=[C_i',C_j']\) (if $\ind$ indexes the number of branches).
\end{enumerate}
\end{proposition}

Proposition \ref{prop:equisingularity} yields a model of the Puiseux series to work with in terms of topological equiva\-lence. In the case of (plane) branches, Proposition \ref{prop:equisingularity} implies that two branches are topologically equivalent if and only if their semigroups are equal; in particular, this means that they have the same Puiseux characteristics. Thus, if \(B\) is a branch with Puiseux characteristics \(\{\beta_0,\dots,\beta_g\}\), then it is topologically equivalent to a branch with Puiseux series 
\begin{equation}\label{eq:toppuiseuxseriesirr}
    s(x)=\sum_{i=1}^{g}a_i x^{\beta_i/\beta_0}.
\end{equation}
The Puiseux series of eq.~ \eqref{eq:toppuiseuxseriesirr} will be called the topological Puiseux series of the branch \(B\); this obviously provides a relation between the Puiseux pairs and the Puiseux characteristics:
\[(p_{1},m_{1})=\bigg(\frac{e_0}{e_1},\frac{\beta_1}{e_1}\bigg),\;(p_{2},m_{2})=\bigg(\frac{e_1}{e_2},\frac{\beta_2}{e_2}\bigg),\dots,(p_{g},m_{g})=(e_{g-1},\beta_g).
\]

If the curve is irreducible, the terms in the topological Puiseux series can be described in terms of the star points in the geodesic joining \(\mathbf{1}\) with the unique arrow in the graph, so that all the star points correspond with curves of maximal contact. For non-irreducible curves, we need to deal with the intersection multiplicities between branches in order to construct a topological Puiseux series associated to the equisingularity type of the curve.
\medskip

We finish this section showing how to attach the topological Puiseux series to a plane curve, following the idea of the irreducible case. First of all, we need to characterize the contact pair in terms of the Puiseux series of the branches.

\begin{proposition}\label{prop:contactPuiseux}
Let \(C_1,C_2\) be two branches with contact pair \((q,c).\) Denote by \([s_i(x)]_{<k}\) the truncation of the Puiseux series of a branch up to order \(k-1\), and write \(\ell^{i}\) for the integer part of \((\obeta^{i}_q-n^{i}_{q-1}\obeta_{q-1}^{i})/e^{i}_{q-1}\), for \(i=1,2\). Consider also

\[k_i=\left\{\begin{array}{cc}
    \frac{\beta_q^i+(c-1)e_q^i}{\beta_0^i} & \text{if}\;c\neq 0 \\[.3cm]
    \frac{\beta_{q-1}^{i}+\ell^{i}e_{q-1}^{i}}{\beta_0^i}  & \text{if}\;c=0,
\end{array}\right. \]

for \(i=1,2\). Then, 
\[
[s_1(x)]_{<k_1}=[s_2(x)]_{<k_2},
\]

and the series \(s_1,s_2\) differ at the term \((\beta_q^i+ce_q^i)/\beta_0^i\); in fact this is the first term in which they differ. Conversely, if two Puiseux series satisfy these conditions, then the corresponding branches have contact \((q,c)\).

\end{proposition}
\begin{proof}
By definition, if \(C_1,C_2\) have contact pair \((q,c)\), then their intersection multiplicity is
$[f_1,f_2]=e^{2}_{q-1}\obeta^{1}_{q}+ce^{1}_qe^{2}_q$. A straighforward application of Noether's formula \ref{noehterformula} (see also \cite[Sec. 5.7]{casas} allows us to conclude.
\end{proof}

Proposition \ref{prop:contactPuiseux} yields a tool to interpret the contact pair in terms of the Puiseux series. Thus, we are ready to present the topological Puiseux series of the branches of a curve.

\begin{theorem}\label{thm:topologicalPuiseuxseries}
Let \(C=\cup_{i=1}^{r} C_i\) be a curve whose topological type is described by the Puiseux exponents \(\{\beta_0^{i},\dots,\beta_{g_i}^{i}\}\) of every branch together with their contact pairs \((q_{i,j},c_{i,j})=(f_i\mid f_j).\) Then, there exists a plane curve singularity \(C'=\cup_{i=1}^{r} C_i'\) which is topologically equivalent to \(C\) such that the Puiseux series of each branch is

\[
s'_i(x)=\sum_{k=1}^{g_i}a^{(i)}_kx^{\beta^i_k/\beta^i_0}+\sum_{j\in\ind\setminus\{i\}}b^{(i)}_jx^{(\beta^i_{q_{i,j}}+c_{i,j}e^i_{q_{i,j}})/\beta^i_0},
\]
where \(b_j^{(i)}\neq b_i^{(j)}\) for all \(i,j\) and \(a^{(i)}_k\neq 0\) for all \(i,k.\) 
\end{theorem}
\begin{proof}
According to Proposition \ref{prop:equisingularity} it would be enough to show that \(\Gamma(C_i)=\Gamma(C_i')\) and that \([f_i,f_j]=[f'_i,f'_j].\) From the expression of \(s'_i(x)\) it follows straightforward that the Puiseux characteristics of \(C_i'\) are identical to the Puiseux characteristics of \(C_i.\) Therefore the equality \(\Gamma(C_i)=\Gamma(C_i')\) follows from the relation between the Puiseux characteristics and the minimal generators of the semigroup (\ref{eq:definbetabarra}). On the other hand, by Proposition \ref{prop:contactPuiseux} we have that \((q'_{i,j},c'_{i,j})=(q_{i,j},c_{i,j})\) from where we deduce the equality \([f_i,f_j]=[f'_i,f'_j]\). 
\end{proof}

From Theorem \ref{thm:topologicalPuiseuxseries}, we have a one-to-one correspondence between the terms in the topological Puiseux series and the star points in the geodesic to the arrow of the corresponding branch of the dual graph of \(C.\) We will refer to the truncation of the topological Puiseux series at a star point to the truncation up to the term defining the star point (including it); thus we will write the truncations as 

\[
y^{i}_{k}=\sum_{l=1}^{k}a_l^{(i)}x^{\beta^{i}_{l}/\beta^i_0}+\sum_{\tiny \begin{array}{c} 
j\in\ind\setminus\{i\}\\
q_{i,j}\leq k\\
\end{array}}b^{(i)}_jx^{(\beta^i_{q_{i,j}}+c_{i,j}e^i_{q_{i,j}})/\beta^i_0}.
\]

In the topological Puiseux series of a branch \(C_j\) of a non--irreducible plane curve \(C\), two different types of Puiseux pairs appear: on the one hand, there are exactly \(g_j\) Puiseux pairs for which \(p_{i,j}\neq 1\), and for those we have that \[(p_{i_1,j},m_{i_1,j})=\bigg(\frac{e^j_0}{e^j_1},\frac{\beta^j_1}{e^j_1}\bigg),\;(p_{i_2,j},m_{i_2,j})=\bigg(\frac{e^j_1}{e^j_2},\frac{\beta^j_2}{e^j_2}\bigg),\dots,(p_{i_g,j},m_{i_g,j})=(e^j_{g-1},\beta^j_g).
\]

On the other hand, if two branches separate at a free point, then there are  Puiseux pairs with \(p_{i,j}=1\) satisfying the following: 

\begin{align*}
    (p_{k,j},m_{k,j})=   & (1,m_{k,j}) \  \mbox{with}\;m_{k,j}\in \Bigg[1,\bigg\lfloor\frac{\beta^j_1}{\beta^j_0}\bigg\rfloor \Bigg]\cap\mathbb{N}\; \ \mbox{for all} \ k<i_1.\\[5pt]
   (p_{k,j},m_{k,j})=& \Big(1,\frac{\beta^j_{\ell-1}}{e^j_{\ell-1}}+c\Big)\;\text{with}\;c\in \Bigg[1,\bigg\lfloor\frac{\beta^j_{\ell}-\beta^j_{\ell-1}}{e^j_{\ell-1}}\bigg\rfloor\Bigg]\cap\mathbb{N}\\
    &\mbox{for all} \ k<i_{\ell} \;\mbox{and}\ \ell=2,\dots,g_j+1.
\end{align*}

\begin{rem}
    The topological Puiseux series of the branches is a canonical representative of the topological class of \(C\). Since we are interested only in topological aspects, we will assume from now on that the Puiseux series of the branches of \(C\) are topological Puiseux series.
\end{rem}

\begin{rem}
    In \cite{LeCarousels} Lé provides a way to define characteristic exponents for a Puiseux development of a non-irreducible plane curve. This allows to treat the Puiseux series of each branch as a ``single parametrization". In that case, one needs to consider a Puiseux parametrization of the form \((t^k,s(t^k))\) with \(k=\obeta_0^1\cdots\obeta_0^{r}.\) However, the ordering in the set of branches defined by Lé is inverse to our ordering. We will provide further explanation of this fact in Section \ref{sec:topology}.
\end{rem}

\subsection{Values at the proper star points}\label{subsec:valuesproperstar}
As we have seen in the previous section, the set of maximal contact curves is enough to determine the equisingularity class of a branch, as it determines the minimal generators of the semigroup. In contrast to this case, we need further information to determine the equisingularity class of a non--irreducible curve; more precisely, we need to consider the values at the proper star points of the dual graph.
\medskip

In the same spirit as done with the maximal contact values, we would like to remark that the values of the proper star points have also a geometrical interpretation. For \(i,j\in\ind\), let $\Gamma_i$ resp.~$\Gamma_j$ be the geodesic in \(G(C)\) joining the origin with the arrow corresponding to  \(f_i\) resp.~\(f_j.\) The point \(R\in\Gamma_i\cap\Gamma_j\) with maximal weight in \(\Gamma_i \cap \Gamma_j\) is called a \textit{separation point}; then any proper star point is in fact a separation point. 
\medskip

According to the definition, we can distinguish among three different types of proper star points: 
\begin{enumerate}
    \item\label{valuepropertype1} \(R\) is the star point associated to a \(q+1\)--th dead arc of \(G(C_i)\); in this case \([f_i,f_j]=e^j_q\obeta^i_{q+1}\leq e^i_q\obeta^j_{q+1}.\) 
    \item\label{valuepropertype2} \(R\) is an ordinary point on \(G(C_i)\) and \(G(C_j)\); in this case \([f_i,f_j]=e^i_q\obeta^j_{q+1}+ce^i_qe^j_q\) with \(c>0.\) 
    \item \label{valuepropertype3}\(R\) is an ordinary point on \(G(C_i)\) and a dead end in \(G(C_j)\); in this case \([f_i,f_j]=e^j_q\obeta^i_{q+1},\) \(c_{i,j}=\min\{l_q+1,l'_q+1\}\).
\end{enumerate}

The analysis of each of these three cases will provide the values at the proper star points we are looking for. For a start, assume we are in case \eqref{valuepropertype1}, and let \(J_1\subset \ind\) be such that for all \(i,i'\in J_1\) we have \((q_{i,i'},c_{i,i'})\geq (q+1,0)\) and for any \(j\notin J_1\) we have \((q_{i,j},c_{i,j})\leq (q+1,0).\) Then, Noether's formula \ref{noehterformula} (see also \cite[Sect.~2.2]{Delgadoari} yields

\begin{equation}\label{eqn:valueproperspecial}
    \pr_i(\underline{v}^R)=\left\{\begin{array}{cl}
         n_{q+1}^{i}\obeta^{i}_{q+1}& \text{if}\;i\in J_1  \\[.3cm]
         \frac{[f_i,f_j]}{e_{q+1}^{j}}& \text{if}\;i\notin J_1 \ \ \text{and}\ \ j\in J_1
    \end{array}\right.
\end{equation}

Next we assume the case (2); let \(J_1\subset \ind\) be such that for all \(i,i'\in J_1\) we have \((q_{i,i'},c_{i,i'})\geq (q,c)\) and for any \(j\notin J_1\) we have \((q_{i,j},c_{i,j})\leq (q,c).\) Again Noether's formula \ref{noehterformula} (see also \cite[Sect.~2.2]{Delgadoari} yields

\begin{equation}\label{eqn:valueproperordinary}
    \pr_i(\underline{v}^R)=\left\{\begin{array}{cl}
         \frac{[f_i,f_j]}{e_{q}^{i}}& \text{if}\;i,j\in J_1  \\[.3cm]
         \frac{[f_i,f_j]}{e_{q}^{j}}& \text{if}\;i\notin J_1 \ \ \text{and}\ \  j\in J_1
    \end{array}\right.
\end{equation}

Finally, observe that the case \eqref{valuepropertype3} can be treated in the same manner as the case \eqref{valuepropertype2}, since by Noether's formula, the intersection multiplicity for the branches in \(J_1\) will coincide with \(e^j_q\obeta^{i}_q.\)
\medskip

In this way, we can define recursively \(D^1,\dots D^{r-1}\in \Gamma\) values such that \(\pr_j(D^i)=\frac{[f_i,f_j]}{e_{q}^{j}}\) for some $i,j,q$; these values are the values attained by curvettes at the proper star points of \(G(C)\) and we have, analogously to the case of maximal contact values, that
\[
\{D^1,\dots,D^{r-1}\}=\{\vu\big ( \pi (\theta_{P(L)}) \big ) \in \Gamma : L \in \mathcal{R}\}.
\]

The set \(\mathcal{B}\cup \{D^1,\dots,D^{r-1}\} \) is called set of principal values (see \cite[Sect.~2.2]{Delgadoari}) and obviously it contains all the necessary information to recover the equisingularity type of the curve. Moreover, Proposition 2.2.6 in \cite{Delgadoari} shows that

\begin{proposition}\label{prop:numbersq}
    For each proper star point \(Q\) the corresponding \(D_Q=\underline{v}(\pi(\theta_Q))\) appears \(s_Q\) times, where \(s_Q=\nu(Q)-3\) if \(Q\) belongs to a dead arc and \(s_Q=\nu(Q)-2\) otherwise.
\end{proposition}

\begin{rem}
    The discussion at the end of Subsection \ref{subsec:dualgraph} shows that \(s_Q=s(Q)-1\) if \(Q\neq \sigma_0\) and \(s_{\sigma_0}=s(\sigma_0)-2\).
\end{rem}


\subsection{A guide tour through the star points}\label{subsec:totalorderstar}
Let us denote by \(\mathcal{S}:=\mathcal{E}\cup\mathcal{R}\) the set of star points of \(G(C).\) In this part, we introduce a total ordering in \(\mathcal{S}\), which will induce a total ordering in the set of proper star points $\mathcal{R}$. Without loss of generality, we assume that the branches of \(C\) are ordered by the good order of the topological Puiseux series of branches. Since the proper star points mark those terms of the topological Puiseux series in which two series differ (Theorem \ref{thm:topologicalPuiseuxseries}), the total order in the set of star points allows us to provide an ordered way to compare the topological Puiseux series of the different branches. Let us describe how this order translates into the dual graph of \(C.\)
\medskip

\medskip

We start with \(\sigma_0,\) which is the first proper star point. By Theorem \ref{thm:topologicalPuiseuxseries} this is the first term in the topological Puiseux series of the branches where at least two of them differ. Since \(\sigma_0\) is the first proper star point, it cannot be a dead end for any branch, hence we distinguish two cases:
\medskip

\noindent \textbf{(A)} First, assume \(\sigma_0\) is the star point associated to the $(q+1)$--death arc of some \(G(C_i),\) as in case \eqref{valuepropertype1} of Subsection \ref{subsec:valuesproperstar}. Let \((q,c)=(f_1\mid\cdots\mid f_r)\) be the contact pair of the curve \(C=\cup_{i=1}^{r}C_i\) and define 

\[
l_{\ind}:=\min\left\{\left\lfloor\frac{\obeta^i_{q+1}-n^i_q\obeta^i_q}{e^i_q}\right\rfloor\;:\;i\in \ind=\{1,\dots,r\}\right\}.
\]

Since \(\sigma_0\) is a star point to the \(q+1\)--death arc of some \(G(C_i),\) which we denote by \(L_{q+1}^{i},\) we know that \((q,c)\in \{(q,l_\ind +1),(q+1,0)\}.\) Let us denote by \(T\) the end point of  \(L_{q+1}^{i}.\) Thus, we define a partition of \(\ind=(\bigcup_{p=1}^{t} I_{p})\cup(\bigcup_{p=t+1}^{s} I_p)\) as follows:
\begin{enumerate}
    \item [(\(\star\))] \(i\in\bigcup_{p=1}^{t}I_p\) if and only if 
    \(\left\lfloor\frac{\obeta^i_{q+1}-n^i_q\obeta^i_q}{e^i_q}\right\rfloor>l_\ind;\)
    \item[(\(\star\star\))]  \label{pack2}
    \(i\in I_{p}\) and \(j\in I_{p'}\) with \(p\neq p'\) if and only if \(T\) is the separation point of \(f_i,f_j\);
    \item[(\(\star\star\star\))] \label{pack3} setting \(\bigcup_{p=t+1}^{s}I_p:=\ind\setminus(\bigcup_{p=1}^{t}I_p)\), for \(p\geq t+1\) we have \(i\in I_{p}\) and \(j\in I_{p'}\)  with \(p\neq p'\) if and only if \(T\) is not the separation point of \(f_i,f_j\) and \((f_i\mid f_j)\in\{(q,c),(q+1,0)\}.\)
\end{enumerate}
The packages \(I_{t+1},\dots,I_s\) will be called ``singular packages" and the packages \(I_1,\dots,I_t\) will be called ``smooth packages"; see Figure \ref{fig:partitionatsigma}.

\begin{rem}
    Observe that if \(i,i'\in I_p\) with \(p\geq t+1\) then  \((f_i\mid f_j)>(q+1,0),\) if \(i\in \bigcup_{p=1}^{t} I_{p} \) and \(j\in \bigcup_{p=t+1}^{s} I_{p}\) then \((f_i\mid f_j)=(q,l+1),\) and for all \(j\in \bigcup_{p=t+1}^{s} I_{p}\) we have \(\left\lfloor\frac{\obeta^i_{q+1}-n^i_q\obeta^i_q}{e^i_q}\right\rfloor=l_\ind.\)
\end{rem}

 \begin{lemma}\label{lem:orderbetaiscompatible} 
 Under the previous notation, set \(\kappa:=\sum_{p=1}^{t}|I_p|\). Assume that the branches are good ordered. Then, the singular packages \(I_{t+1},\cdots,I_s\) are ordered as 
\begin{equation}\label{eq:refinedorder}
\frac{\obeta^{j_{t+1}}_{q+1}}{e^{j_{t+1}}_q}\geq\cdots\geq \frac{\obeta^{j_s}_{q+1}}{e^{j_s}_q}.
\end{equation}

Moreover, without loss of generality we can set \(\big\{\kappa+1,\dots, \kappa+|I_{t+1}|\big\}=I_{t+1},\dots,\) \(\big\{\kappa+|I_{t+1}|+1,\dots,\kappa+|I_{t+1}|+|I_{t+2}|\big\}=I_{t+2}\) and \(\displaystyle\Big\{\kappa+\sum_{p=t+1}^{s-1}|I_p|+1,...,r\Big\}=I_{s}.\)
 \end{lemma}
 \begin{proof}
 We need only to check that the good order in the topological Puiseux series is compatible with the order in the packages defined by eq.~ \eqref{eq:refinedorder}. First, we show that the packages in \(J:=\displaystyle\bigcup_{p=t+1}^{s}I_p\) are ordered following the good order in the topological Puiseux series. 
 \medskip
 
 As a consequence of Proposition \ref{prop:contactPuiseux}, the topological Puiseux series of the branches in \(J\) are exactly the same up to the terms which are strictly smaller than the term corresponding to \(\beta^j_{q+1}/\beta^j_0\) and all of them differ at that term. Then, the first Puiseux pair in which they become different is of the form \((e^j_q/e^j_{q+1},\beta^j_{q+1}/e^j_{q+1}).\) Therefore, if \(j_k\in I_k\subset J\), then we need to check that 
 
 \[
 \frac{\beta^{j_k}_{q+1}/e^{j_k}_{q+1}}{e^{j_{k}}_q/e^{j_{k}}_{q+1}}\geq \frac{\beta^{j_{k+1}}_{q+1}/e^{j_{k+1}}_{q+1}}{e^{j_{k+1}}_q/e^{j_{k+1}}_{q+1}}.
 \] 

Combining eq.~\eqref{eq:definbetabarra} and eq.~ \eqref{eq:refinedorder} we have 
\[\frac{\obeta^{j_k}_{q+1}}{e^{j_k}_{q}}=\frac{\beta^{j_k}_{q+1}}{e^{j_k}_{q}}-\frac{\beta^{j_k}_{q}}{e^{j_k}_{q}}+n^{j_k}_{q}\frac{\obeta^{j_k}_{q}}{e^{j_k}_{q}}\geq \frac{\beta^{j_{k+1}}_{q+1}}{e^{j_{k+1}}_{q}}-\frac{\beta^{j_{k+1}}_{q}}{e^{j_{k+1}}_{q}}+n^{j_{k+1}}_{q}\frac{\obeta^{j_{k+1}}_{q}}{e^{j_{k+1}}_{q}}.\]

Recall that \(\obeta^{I_p}/e_q^{I_p}\) is independent of \(p\) for \(p=t+1,\dots,s\). Then it is easily seen that \(n^{I_p}_j\) is independent of \(p\) for $j=1,\ldots , q$. Hence the first Puiseux pair in which they become different is of the form \((e^j_q/e^j_{q+1},\beta^j_{q+1}/e^j_{q+1})\), and we have  
\[
-\frac{\beta^{j_k}_{q}}{e^{j_k}_{q}}+n^{j_k}_{q}\frac{\obeta^{j_k}_{q}}{e^{j_k}_{q}}=-\frac{\beta^{j_{k+1}}_{q}}{e^{j_{k+1}}_{q}}+n^{j_{k+1}}_{q}\frac{\obeta^{j_{k+1}}_{q}}{e^{j_{k+1}}_{q}};
\]
this allows us to deduce the desired inequality \(\frac{\beta^{j_k}_{q+1}}{e^{j_k}_{q}}\geq \frac{\beta^{j_{k+1}}_{q+1}}{e^{j_{k+1}}_{q}}.\)
\medskip

On the other hand, by the definition of the packages in \(\ind\setminus J,\) the first term which is different in the topological Puiseux series of a branch in \(I_l\subset \ind\setminus J\) with respect to the topological Puiseux series of a branch not in \(I_p \subset \ind\setminus J\) is of the form \((1,\beta^{I_p}_q/e^{I_p}_q+c)\) with \(c=l_\ind+1\) and the intersection multiplicity  \([f_{I_p},f_{I_{t+1}}]=e_{q-1}^{I_{t+1}}\obeta_q^{I_p}+(l_\ind+1)e_q^{I_{t+1}}e_q^{I_p}\); combining eq.~\eqref{eq:definbetabarra} and eq.~ \eqref{eq:refinedorder} as before we obtain

\[
\beta^{I_p}_q/e^{I_p}_q+c>\beta^{I_{t+1}}_{q+1}/e_q.
\]

Thus if \(i\in \ind\setminus J\) and \(j\in J\), then  we have \(i<j.\)
 \end{proof}

Moreover, we can also set
 \(\big \{1,\dots,|I_1|\big\}=I_1,\) \(\big\{| I_1|+1,\dots,|I_2|\big\}=I_2,\dots,\) \(\displaystyle\Big\{\sum_{i=1}^{t-1}|I_i|+1,\dots,\kappa\Big\}=I_t\); this holds because the term \(\beta_q^{I_p}/e_q^{I_p}+c\) is equal for all \(p=1,\dots,t.\) Therefore, they trivially satisfy the good ordering in the topological Puiseux series. Furthermore, the topological Puiseux series of two branches \(i\in I_p\) and \(j\in I_{p'}\) with \(p\neq p'\) and \(p,p'\leq t\) have different coefficients for the term \(\beta_q^{I_p}/e_q^{I_p}+c.\) Figure \ref{fig:partitionatsigma} describes the dual graph at this stage.

\begin{figure}[h]
$$
\unitlength=0.50mm
\begin{picture}(70.00,110.00)(0,-50)
\thinlines

\put(-60,30){\line(1,0){24}}
\put(-60,30){\circle*{2}}

\put(-40,30){\circle*{2}}
\put(-40,30){\line(0,-1){15}}
\put(-40,15){\circle*{2}}

\put(-17,30){\circle*{2}}
\put(-17,30){\line(0,-1){15}}
\put(-17,15){\circle*{2}}

\put(-34,30){$\ldots$}

\put(-22,30){\line(1,0){18}}


\put(-4,30){\circle*{3}}
\put(-4,30){\line(1,-0.5){70}}

\put(-4,30){\line(-0.2,1){3}}
\put(-4,30){\line(0.2,0.5){5}}
\put(-4,30){\line(1,0.2){15}}


\put(4,43){\scriptsize{$I_{s-1}$}}

\put(-6,47){\scriptsize{$I_{s}$}}

\put(-6,23){\scriptsize{$\sigma_0$}}




\put(31,13){\circle*{2}}
\put(31,13){\line(0,1){14}}
\put(31,13){\line(1,1.3){9}}
\put(31,13){\line(1,0){14}}


\put(66,-5){\circle*{2}}
\put(66,-5){\line(0.1,1){1.5}}
\put(66,-5){\line(1,1.4){9}}
\put(66,-5){\line(1,-0.2){14}}

\put(83,-7){\scriptsize{$I_{t+1}$}}

\put(66,-5){\line(-1,-1){20}}

\put(46,-25){\circle*{2}}
\put(46,-25){\line(1,0.3){15}}
\put(46,-25){\line(1,-0.4){16}}
\put(46,-25){\line(1,-1.4){11}}

\put(64,-30){\scriptsize{$I_{2}$}}
\put(60,-42){\scriptsize{$I_{1}$}}













\end{picture}
$$
\caption{Case (A): Dual graph with \(\sigma_0\) being star point of some branch.}
\label{fig:partitionatsigma}
\end{figure}
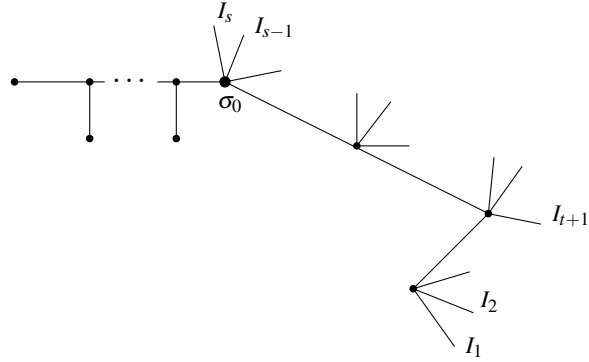

\noindent \textbf{(B)} Now, let us consider the case where \(\sigma_0\) is an ordinary point for all \(G(C_i).\) This case can be treated as the previous case if we consider \(\ind=\bigcup_{p=1}^{t}I_p\). Observe that if \(\sigma_0\) is an ordinary point in the dual graph of all the branches then \(\sigma_0\) is a separation point. Therefore, we can define the partition of \(\ind\) as \(i,j\in I_p\) if and only if \((f_i\mid f_j)>(q,c).\) Thus, the topological Puiseux series of all the branches are the same for order strictly less than \(\beta_q^{I_p}/e_q^{I_p}+c\). At the term \(\beta_q^{I_p}/e_q^{I_p}+c\), which---we recall---is independent of \(p,\) the series have a different coefficient if and only if they belong to a different package. The description of the dual graph is now easier: 

\begin{figure}[h]
$$
\unitlength=0.50mm
\begin{picture}(80.00,40.00)(50,10)
\thinlines

\put(20,30){\line(1,0){30}}
\put(20,30){\circle*{2}}
\put(40,30){\circle*{2}}
\put(40,30){\line(0,-1){15}}
\put(40,15){\circle*{2}}

\put(54,30){$\ldots$}

\put(70,30){\circle*{2}}
\put(70,15){\circle*{2}}
\put(70,30){\line(0,-1){15}}
\put(67,30){\line(1,0){23}}
\put(90,30){\circle*{2}}
\put(83,34){{\scriptsize$\sigma_0$}}

\put(90,30){\line(1,0.5){12}}
\put(90,30){\line(1,1){10}}
\put(90,30){\line(1,-0.1){13}}
\put(90,30){\line(1,-0.4){12}}
\put(90,30){\line(1,-1){10}}


\end{picture}
$$
\caption{Case (B): $\sigma_0$ is an ordinary point of every $G(C_i)$.}
\label{fig4}
\end{figure}
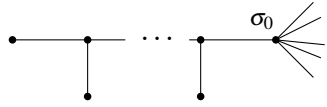

Again, we can put \(\big\{1,\dots,|I_1|\big\}=I_1,\) \(\big\{| I_1|+1,\dots,|I_2|\big\}=I_2,\dots,\) \(\displaystyle\Big\{\sum_{i=1}^{t-1}|I_i|+1,\dots,r\Big\}=I_t\), since this is compatible with the good order. 
\medskip

Once we have ordered the packages at \(\sigma_0,\) we shall continue decomposing each package until we arrive to a decomposition of packages with cardinal exactly one. There are two options: either there is at least one package with cardinal strictly bigger than one or all the packages have cardinal exactly one and then we are done. Assume that there is a package with cardinal strictly bigger than one. We will first consider the case where there is a package in \(\bigcup_{p=1}^{t} I_p\) with \(|I_p|>1.\) Furthermore, assume for simplicity that \(|I_1|>1\) and abusing a bit of notation, let us denote the contact pair of the package as \((q,c)=(f_1\mid\cdots\mid f_{|I_1|}).\) Let  \(\sigma_0^{I_1}\) be the first proper star point of the dual graph of \(C^{I_1},\) where \(C^{I_1}\) is the curve defined by \(f_{I_1}=f_1\cdots f_{|I_1|}.\) For \(\sigma_0^{I_1}\) we have now two different situations to be considered:


\begin{enumerate}
    \item Assume that \(\sigma_0^{I_1}\) is the star point associated to the \(q'\)--death arc of some \(G(C^i),\) \(i\in I_1\) and \(q'>q.\) In this case, \(\sigma_0^{I_1}\neq T,\) where \(T\) is the end point of \(L^r_{q+1}.\) We proceed in the same way as in the case of \(\sigma_0\) to define a partition of \(I_1=(\bigcup_{p=1}^{t_1}I_{1,p})\cup(\bigcup_{t_1+1}^{s_1} I_{1,p}).\) The partition is defined and ordered as in the case of \(\sigma_0;\) to do so, we have to take into account that \(\sigma_0^{I_1}\) plays the role of \(\sigma_0\) in the dual graph of \(G(C^{I_1}).\) Again by Lemma \ref{lem:orderbetaiscompatible} we have that the ordering of the subpackages of \(I_1\) is compatible with the good order in the topological Puiseux series.

    \begin{figure}[h]
$$
\unitlength=0.50mm
\begin{picture}(100.00,110.00)(0,-70)
\thinlines

\put(-60,30){\line(1,0){24}}
\put(-60,30){\circle*{2}}

\put(-40,30){\circle*{2}}
\put(-40,30){\line(0,-1){15}}
\put(-40,15){\circle*{2}}

\put(-17,30){\circle*{2}}
\put(-17,30){\line(0,-1){15}}
\put(-17,15){\circle*{2}}

\put(-34,30){$\ldots$}

\put(-22,30){\line(1,0){18}}

\put(-8,35){\scriptsize{$\sigma_0$}}
\put(13,39){\scriptsize{$I_s$}}
\put(18,33){\scriptsize{$I_{s-1}$}}
\put(32,23){$\ddots$}

\put(-4,30){\circle*{3}}
\put(-4,30){\line(1,-1){30}}

\put(-4,30){\line(1,0.5){15}}
\put(-4,30){\line(1,0.2){20}}
\put(-4,30){\line(1,0){15}}

\put(4.5,21){\circle*{2}}
\put(4.5,21){\line(1,-0.2){9}}
\put(4.5,21){\line(1,0.5){9}}
\put(4.5,21){\line(1,1){6}}

\put(15.5,10){\circle*{2}}
\put(15.5,10){\line(1,1.3){9}}
\put(15.5,10){\line(1,0.4){9}}
\put(15.5,10){\line(1,2){8}}


\put(26,0){\circle*{2}}
\put(26,0){\line(1,1.9){7.5}}
\put(26,0){\line(1,1.3){9}}

\put(26,0){\line(1,.7){14}}
\put(41,5){\scriptsize{$I_{t+1}$}}
\put(35,-10){\scriptsize{$I_{t}$}}
\put(29,-32){\scriptsize{$I_{1}$}}
\put(69,-34){\scriptsize{$\sigma_0^{I_1}$}}

\put(26,0){\line(0,-1){25}}

\put(26,-25){\circle*{2}}
\put(26,-25){\line(1,1.9){7.5}}
\put(26,-25){\line(1,1.3){9}}
\put(26,-25){\line(1,1){10}}
\put(26,-25){\line(1,.7){10}}

\put(26,-25){\line(1,0){25}}

\put(54,-25){$\ldots$}

\put(65,-25){\line(1,0){10}}

\put(38,-25){\circle*{2}}
\put(38,-25){\line(0,-1){15}}
\put(38,-40){\circle*{2}}

\put(46,-25){\circle*{2}}
\put(46,-25){\line(0,-1){15}}
\put(46,-40){\circle*{2}}

\put(75,-25){\circle*{3}}
\put(75,-25){\line(1,1.9){7.5}}
\put(75,-25){\line(1,1.3){9}}
\put(75,-25){\line(1,1){10}}
\put(75,-25){\line(1,.7){10}}

\put(75,-25){\line(1,-0.5){40}}

\put(100,-37.5){\circle*{2}}

\put(100,-37.5){\line(1,1){10}}
\put(100,-37.5){\line(1,1.3){9}}
\put(100,-37.5){\line(1,0.7){10}}
\put(100,-37.5){\line(1,1.7){8}}

\put(115,-45){\circle*{2}}
\put(115,-45){\line(1,0){20}}
\put(115,-45){\line(1,0.2){20}}
\put(115,-45){\line(1,.4){20}}

\put(115,-45){\line(-1,-0.5){20}}

\put(95,-55){\circle*{3}}
\put(89,-55){\scriptsize{$T$}}
\put(95,-55){\line(1,-.7){20}}
\put(95,-55){\line(1,-.5){20}}
\put(95,-55){\line(1,-.3){20}}
\put(95,-55){\line(1,0){20}}

\put(120,-55){\scriptsize{$I_{1,t_1}$}}
\put(120,-65){$\vdots$}
\put(120,-75){\scriptsize{$I_{1,1}$}}

\put(80,-8){\scriptsize{$I_{1,s_1}$}}
\put(140,-39){$\vdots$}
\put(140,-48){\scriptsize{$I_{1,t_1+1}$}}

\put(-67,28){{\scriptsize ${\bf 1}$}}

\end{picture}
$$
\caption{Situation (1): $\sigma_0^{I_1}$ is proper.}
\label{figC}
\end{figure}
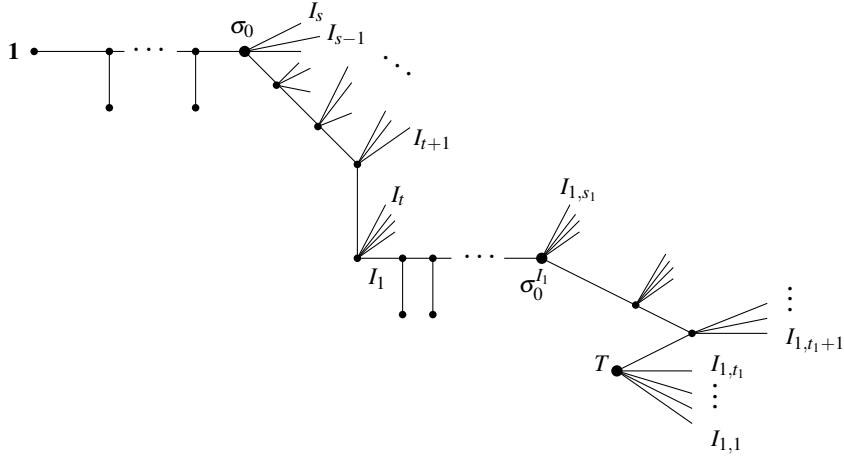

    \item Assume that \(\sigma_0^{I_1}\) is an ordinary point for all \(G(C^i)\), \(i\in I_1.\) Then again the partition of \(I_1\) is defined and ordered as in the case of \(\sigma_0.\) All the packages generated in this partition are smooth packages.

\end{enumerate}

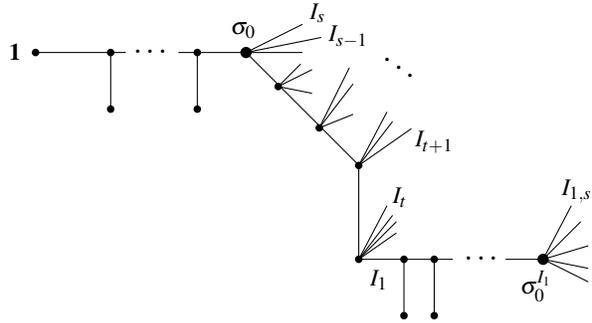
\begin{figure}[h]
$$
\unitlength=0.50mm
\begin{picture}(70.00,110.00)(0,-50)
\thinlines

\put(-60,30){\line(1,0){24}}
\put(-60,30){\circle*{2}}

\put(-40,30){\circle*{2}}
\put(-40,30){\line(0,-1){15}}
\put(-40,15){\circle*{2}}

\put(-17,30){\circle*{2}}
\put(-17,30){\line(0,-1){15}}
\put(-17,15){\circle*{2}}

\put(-34,30){$\ldots$}

\put(-22,30){\line(1,0){18}}

\put(-8,35){\scriptsize{$\sigma_0$}}
\put(13,39){\scriptsize{$I_s$}}
\put(18,33){\scriptsize{$I_{s-1}$}}
\put(32,23){$\ddots$}

\put(-4,30){\circle*{3}}
\put(-4,30){\line(1,-1){30}}

\put(-4,30){\line(1,0.5){15}}
\put(-4,30){\line(1,0.2){20}}
\put(-4,30){\line(1,0){15}}

\put(4.5,21){\circle*{2}}
\put(4.5,21){\line(1,-0.2){9}}
\put(4.5,21){\line(1,0.5){9}}
\put(4.5,21){\line(1,1){6}}

\put(15.5,10){\circle*{2}}
\put(15.5,10){\line(1,1.3){9}}
\put(15.5,10){\line(1,0.4){9}}
\put(15.5,10){\line(1,2){8}}


\put(26,0){\circle*{2}}
\put(26,0){\line(1,1.9){7.5}}
\put(26,0){\line(1,1.3){9}}

\put(26,0){\line(1,.7){14}}
\put(41,5){\scriptsize{$I_{t+1}$}}
\put(35,-10){\scriptsize{$I_{t}$}}
\put(29,-32){\scriptsize{$I_{1}$}}
\put(69,-34){\scriptsize{$\sigma_0^{I_1}$}}

\put(26,0){\line(0,-1){25}}

\put(26,-25){\circle*{2}}
\put(26,-25){\line(1,1.9){7.5}}
\put(26,-25){\line(1,1.3){9}}
\put(26,-25){\line(1,1){10}}
\put(26,-25){\line(1,.7){10}}

\put(26,-25){\line(1,0){25}}

\put(54,-25){$\ldots$}

\put(65,-25){\line(1,0){10}}

\put(38,-25){\circle*{2}}
\put(38,-25){\line(0,-1){15}}
\put(38,-40){\circle*{2}}

\put(46,-25){\circle*{2}}
\put(46,-25){\line(0,-1){15}}
\put(46,-40){\circle*{2}}

\put(75,-25){\circle*{3}}
\put(75,-25){\line(1,1.9){7.5}}
\put(75,-25){\line(1,1){10}}
\put(75,-25){\line(1,0.3){12}}
\put(75,-25){\line(1,-.2){12}}

\put(75,-25){\line(1,-0.5){12}}

\put(80,-8){\scriptsize{$I_{1,s_1}$}}

\put(-67,28){{\scriptsize ${\bf 1}$}}

\end{picture}
$$
\caption{Situation (2): $\sigma_0^{I_1}$ is ordinary.}
\label{figD}
\end{figure}

Continuing with this process we will finally obtain an ordering of the first \(|I_1|\) branches which is compatible with the good ordering of the set of branches induced by the topological Puiseux series. Once we finish with \(I_1\) we repeat this process with each package of \(\bigcup_{p=1}^t I_p\) with \(|I_p|>1.\) In this way we obtain an ordering of the first \(\kappa\) branches which is compatible with the good ordering of the set of branches induced by the topological Puiseux series. Now we shall continue with the package \(I_{t+1}.\) As in the case of \(I_1\) we only need to deal with the packages with cardinal strictly bigger than one. The procedure is the same as in the cases developed for \(I_1.\) After all the iterations, we obtain a partition of \(\ind\) in packages of cardinal one such that the indexing of the packages is compatible with the good order induced by the topological Puiseux series in the set of branches. 
\medskip

Once the branches are ordered with this process, let us denote by \(\prec\) the natural order induced by a geodesic in the vertices of \(G(C)\). Let \(\Gamma_1\) be the geodesic joining \(\mathbf{1}\) with the arrow corresponding to \(f_1\), and define 

\[
\mathcal{S}_1:=\Gamma_1\cap \mathcal{S}=\{\alpha_1\prec\alpha_2\cdots\prec\alpha_{t_1}\}.
\]

Recursively, for \(2\leq i\leq r\) we consider the geodesic \(\Gamma_i\) joining \(\mathbf{1}\) with the arrow corresponding to \(f_i\) and define 

\[
\mathcal{S}_i:=\Gamma_i\cap\bigg(\mathcal{S}\setminus \big (\bigcup_{k<i}\mathcal{S}_k\big)\bigg)=\{\alpha_{t_k+1}\prec \cdots \prec \alpha_{t_i}\}.
\]

From now on we will assume that the set of star points in \(G(C)\) is ordered by the relation \(<\) as follows: if \(\alpha_i,\alpha_j\in \mathcal{S}_k\) for some \(k\), then \(\alpha_i<\alpha_j\) if and only if \(\alpha_i\prec \alpha_j\); if \(\alpha_i\in \mathcal{S}_k\) and \(\alpha_j\in \mathcal{S}_{k'}\), then \(\alpha_i<\alpha_j\) if and only if \(k<k'.\) Obviously, \(<\) produces a total order in \(\mathcal{S}\) by the construction of the sets \(\mathcal{S}_i.\)
\medskip

Summarizing, the set of star points \(\mathcal{S}\) is totally ordered and this order is compatible with the good order in the topological Puiseux series of \(C.\) In fact, given the equisingularity data we have provided a way to compute this order. We will see in Example \ref{ex:examplerefinedgoodorder} that this ordering is a bit more restrictive than just the good order in the topological Puiseux series. From now on we will refer as good order to our refined good order and not just the good order.

\subsubsection{Approximations associated to the star points}\label{subsubsec:truncationstar}

To finish this section, let us define a sequence of truncated plane curves which can be associated to the star points following their total ordering. This sequence will allow us to better understand the ordering \(<\) in \(\mathcal{S}.\) Let \((q,c)=(f_1\mid\cdots\mid f_r)\) be the contact pair of \(C,\) then for \(i=1,\dots,q-1\) we define \(C_{\alpha_i}\) as the irreducible plane curve given by the Puiseux parametrization
\[y_{\alpha_i}:=y^{1}_{i}=\sum_{k=1}^{i} a_k^{(1)}x^{\beta^1_i/\beta^1_0}.\]
Recall that, for \(i=1,\dots,q-1\), the quotient \(\beta^j_i/\beta^j_0\) is independent of \(j=1,\dots,r\) and then \(y_{\alpha_i}\) is a maximal contact curve which is common to all the branches. Now, let us denote by \(T=\alpha_{q-1}\) and \(\sigma=\sigma_0\); we must distinguish two cases: 
\begin{enumerate}
    \item [\(\diamond\)] If \(c>0\) then define \(C_{\alpha_q}\) as the irreducible plane curve with Puiseux series \(y_{\alpha_q}:=y^1_q.\)
    \item [\(\diamond\)] If \(c=0\) then \(\alpha_q=\sigma_0\) and consider the partition  \(\ind=(\bigcup_{p=1}^{t} I_{p})\cup(\bigcup_{p=t+1}^{s} I_p)\) explained before. Then, we define \(C_{\alpha_q}=C_{\sigma_0}\) as the plane curve singularity with \(s\) branches defined by the Puiseux series:
    \[y^{I_p}_{\alpha_q}:= \sum_{l=1}^{q}a_l^{(i)}x^{\beta^{i}_{l}/\beta^i_0}+\sum_{\tiny \begin{array}{c} 
j\notin I_p\\
\end{array}}b^{(i)}_jx^{(\beta^i_{q_{i,j}}+c_{i,j}e^i_{q_{i,j}})/\beta^i_0}\quad\text{with}\quad i\in I_p.\]
\end{enumerate}
In the case \(c>0\) we have \(\sigma_0=\alpha_{q+1}\), and we define \(C_{\alpha_{q+1}}=C_{\sigma_0}\) analogously to the case \(c=0.\)
\medskip

Following the procedure described to order \(\mathcal{S},\) let \(\ind=(\bigcup_{p=1}^{t} I_{p})\cup(\bigcup_{p=t+1}^{s} I_p)\) be the partition created at \(\sigma_0.\) By definition there are \(\epsilon\) proper star points (in fact, \(s-t-\sum (s_Q-1)\)) in \(\mathcal{S}_1\) between \(\sigma_0\) and the separation point of \(I_1\) with \(I_2.\) Let us denote them by \(\osigma_1,\dots,\osigma_\epsilon\) and let \(P\) be the separation point of \(I_1\) and \(I_2.\) We put \(C_{\osigma_i}=C_{\sigma_0}\) for all \(i.\) At this point we have 
\[
\alpha_1\prec \cdots\prec\alpha_q\preceq \sigma_0\preceq \osigma_1\preceq\cdots\preceq\osigma_\epsilon\preceq P.
\]

The consideration of further approximations of $C$ requires to define a recursive procedure which distinguishes several cases. To do so, we rename the distinguished points \(T,\sigma\): Let \(T=\sigma_0\) be a star point where the process start, and let \(\sigma\) be the next star point to be considered in order to define an approximating curve. Then,

\begin{enumerate}
    \item Assume \(|I_1|=1:\)
    \begin{enumerate}
        \item The semigroup \(\Gamma^1\) of the first branch \(C^1\) of \(C\) has \(q\) minimal generators. This implies that \(\mathcal{S}_1=\{\alpha_1\prec \cdots\prec\alpha_q\preceq \sigma_0\preceq \osigma_1\preceq\cdots\preceq\osigma_\epsilon\}\) and \(y^{I_1}_{\sigma_0}\) is the topological Puiseux series of the branch \(C^1.\) Then, we have finished with \(\mathcal{S}_1\) and we move to the package \(I_2.\) We set \(T=\sigma_0\) 
 and write \(\sigma\) for the star point from which \(I_2\) emanates, i.e. \(\sigma=\osigma_i\) for some \(i.\)
        \item The semigroup \(\Gamma^1\) of the first branch \(C^1\) of \(C\) has \(g_1>q\) minimal generators. Then,
        \[\mathcal{S}_1=\{\alpha_1\prec \cdots\prec\alpha_q\preceq \sigma_0\preceq \osigma_1\preceq\cdots\preceq\osigma_\epsilon\prec \alpha_1^{I_1}\prec\cdots \alpha_{g_1-q}^{I_1}\}\]
        where \(\alpha_1^{I_1}\prec\cdots \alpha_{g_1-q}^{I_1}\) are the non-proper star points defining the maximal contact values associated to the remaining generators of the semigroup \(\Gamma^1\). For each \(\alpha^{I_1}_i\) we define the plane curve \(C_{\alpha^{I_1}_i}\)  with \(s\) branches, where the branches \(j=2,\dots,s\) have Puiseux series \(y_{\alpha^{I_1}_i}^{j}=y_{\sigma_0}^{I_j},\) i.e. the branches \(j=2,\dots,s\) are the same as the ones of \(C_{\sigma_0}\), and for \(j=1\) the Puiseux series is
        \[
        y_{\alpha^{I_1}_i}^{1}:=y_{\sigma_0}^{I_1}+\sum_{k=q+1}^{q+i} a_k^{(1)}x^{\beta_k^1/\beta_0^1}.
        \]
        Once we have completed \(\mathcal{S}_1\), we move to the package \(I_2.\) We make \(\sigma\) the star point from which \(I_2\) goes through and \(T=\alpha_{g_1-q}^{I_1}.\)
    \end{enumerate}
    \item Assume \(|I_1|>1.\) There are two cases to be distinguished:
    \begin{enumerate}
        \item For all \(j\in I_1\) we  have \(g_j=q,\) i.e. the semigroups \(\Gamma^{j}\) have \(q\)--minimal generators. This implies that all the star points in \(\mathcal{S}_1\) after \(T\) are proper star points of \(G(C)\) and they are ordinary points of the individual dual graphs \(G(C^j)\) of the branches. Let \(\sigma_0^{I_{1}}\prec\cdots\prec \sigma_{0}^{I_{1,\cdots,1}}\) be the \(l\leq |I_1|\) proper star points from \(T\) to the arrow of \(C^1\) in \(G(C).\) We only need to analyze the situation at the first one, \(\sigma_0^{I_1}\); for the remainder it follows by the recursive process we are defining. For \(\sigma_0^{I_1},\) we have a partition into smooth packages of \(I_1=\bigcup_{k=1}^{s_1} I_{1,k}\) and we define the plane curve \(C_{\sigma_0^{I_1}}\) with \(s_1+s-1\) branches where the first \(s_1\) branches have Puiseux series of the form
        \[y^{I_{1,k}}_{\sigma_1^{I_1}}:=y^{I_1}_{\sigma_0}+\sum_{i\in I_1\setminus{I_{1,k}}}b_j^{j}x^{(\beta^j_{q_{i,j}}+c_{i,j}e^j_{q_{i,j}})/\beta^j_0}\quad\text{with} \quad j\in I_{1,k},\]
        and the last \(s-1\) branches are equal to the branches of \(C_{\sigma_0}.\) We set \(\sigma\) the star point from which \(I_2\) goes through and \(T=\sigma_0^{I_{1,1}}.\)
        
        \item We assume that \(g_j>q\) for some \(j\in I_1.\) We distinguish again two subcases:
        \begin{enumerate}
            \item \(C^1\) has \(g_1\geq q\) and \((f_1|\cdots|f_{|I_1|})\leq (q+1,0)\). Denote by \(\sigma_{0}^{I_1}\) the first proper star point of the package \(I_1.\) 
            Let \(I_1=\bigcup_{k=1}^{s_1} I_{1,k}\) be the partition associated to the proper star \(\sigma_{0}^{I_1}.\) Then, we define \(C_{\sigma_{0}^{I_1}}\) as the plane curve with \(s_1+s-1\) branches, where the last \(s-1\) branches are equal to the last  \(s-1\) branches of \(C_{\sigma_0}\), and the first \(s_1\) branches are defined in the same way we defined \(C_{\sigma_0}\) from \(C_{\alpha_q}.\) In this case \(\sigma_0^{I_1}\) plays the role of \(\sigma_0\) and \(T\) plays the role of \(\alpha_q.\) Thus, we have again a sequence
            \[
            \alpha_1\prec \cdots\prec\alpha_q\preceq \sigma_0\preceq \osigma_1\preceq\cdots\preceq\osigma_\epsilon\prec \sigma_0^{I_1}\preceq \osigma^{I_1}_1\preceq\cdots\preceq\osigma^{I_1}_\epsilon
            \]
            and set \(\sigma=\osigma_{\epsilon_1}^{I_1}\) and \(T=\sigma_0^{I_1}\) to continue the process.   
            \item \(C^1 \) has \(g_1\geq q\) and \((f_1|\cdots|f_{|I_1|})> (q+1,0),\) i.e \((f_1|\cdots|f_{|I_1|})\geq (q+1,c)\) with \(c\neq 0.\) Write \((q_{I_1},c_{I_1}):= (f_1|\cdots|f_{|I_1|}).\) Since \((q_{I_1},c_{I_1})> (q+1,0),\) there are \(q_{I_1}-q\) star points which are non-proper between \(\osigma_\epsilon\) and \(\sigma_0^{I_1},\) i.e. 
            \[\osigma_{\epsilon}\prec \alpha_{q+1}\prec \cdots \prec \alpha_{q_{I_1}}\preceq \sigma_0^{I_1}.\] 
            Then for each \(\alpha_i\) with \(i=1,\dots,q_{I_1}-q\) we define \(C_{\alpha_i}\) as a plane curve with the same number of branches \(s\) as \(C_{\sigma_0}\), where the last \(s-1\) branches are the same as those of \(C_{\sigma_0}\) and the first branch is defined as 
            \[
        y^{1}_{\alpha_i}:=y^{I_1}_{\sigma_0}+ +\sum_{k=q+1}^{q_{I_1}} a_k^{(1)}x^{\beta_k^1/\beta_0^1}
        \]
            similarly to the case \((1)(b).\) 
            Moreover we define \(C_{\sigma_{0}^{I_1}}\) as the plane curve with \(s_1+s-1\) branches, where the last \(s-1\) branches are equal to the last \(s-1\) branches branches of \(C_{\sigma_0}\) and the first \(s_1\) branches are defined in the same way we defined \(C_{\sigma_0}\) from \(C_{\alpha_q}.\) In this case \(\sigma_0^{I_1}\) plays the role of \(\sigma_0\) and \(T\) plays the role of \(\alpha_{q_{I_1}}.\) We set \(T=\sigma_0^{I_1}\) and \(\sigma\) for the next star point that must be considered to define an approximating curve, i.e. it is defined from the partition associated to \(\sigma_0^{I_1}\) in the same way as in the previous cases.
        \end{enumerate}
    \end{enumerate}
\end{enumerate}

We run this process until \(\sigma=\max\{\alpha\in\mathcal{S}\}\); in that case \(C_\sigma=C\) and then we have obtained the given plane curve.
\medskip

Observe that the ordering that we have introduced in the set of branches is now a canonical order for the branches on a fixed equisingularity class and it can be described only by using the dual graph of \(C.\) Moreover, we have showed that this ordering introduced in the dual graph implies the good ordering of the topological Puiseux series of the branches. As we will see in Section \ref{subsec:iterativepoincare}, the total order in the star points of the dual graph together with the order in the set of branches is crucial to provide an iterative construction of the Poincaré series.

\begin{ex}\label{ex:examplerefinedgoodorder}
    Consider one of the plane curves defined in Example \ref{example:goodorder}. Assume first the curve is given by the Puiseux series of the branches \(C_i\) ordered as  \(y_1=2x^2+x^{14/3},\)
  \(y_2=x^4,\) \(y_3=x^{5/2},\)  \(y_4=2x^2\) and \(y_5=x^2.\) The equisingularity class is given by the value semigroups   
  $$
  \Gamma^1=\langle 3,14\rangle, \ \ \Gamma^3=\langle2,5\rangle, \ \ \Gamma^2=\Gamma^4=\Gamma^5=\mathbb{N},
  $$
  and the intersection multiplicities: \([f_1,f_2]=[f_1,f_5]=6,\) \([f_1,f_3]=12,\) \([f_1,f_4]=14,\) \([f_2,f_3]=5,\) \([f_2,f_4]=[f_2,f_5]=[f_4,f_5]=2\) and \([f_3,f_4]=[f_3,f_5]=4.\)
\medskip

  The topological Puiseux series are
  \begin{equation*}
        \begin{array}{ccc}
      y_1=2x^2+x^{14/3}+1/2x^5,& \quad y_2=5x^2+x^{3},&\quad y_3=5x^2+x^{5/2}+3x^3,\\
      y_4=2x^2+x^5 &{\text{and}}  &\quad y_5=x^2.
  \end{array}
  \end{equation*}
  The dual graph of \(C\) with the branches in this order is given in Figure \ref{fig:ex_0}. 
  \begin{figure}[H]
$$
\unitlength=0.50mm
\begin{picture}(80.00,40.00)(50,5)
\thinlines
\put(30,30){\line(1,0){20}}
\put(30,30){\circle*{2}}
\put(50,30){\circle*{3}}
\put(43,34){{\scriptsize$\sigma_0$}}
\put(50,30){\vector(1,1){10}}
\put(50,30){\line(1,0){80}}
\put(70,30){\circle*{2}}
\put(90,30){\circle*{2}}
\put(110,30){\circle*{2}}
\put(130,30){\circle*{2}}
\put(130,30){\vector(1,1){10}}
\put(130,30){\line(0,-1){20}}
\put(130,20){\circle*{2}}
\put(130,10){\circle*{2}}
\put(130,10){\vector(1,0){15}}
\put(50,30){\line(1,-1){15}}
\put(65,15){\circle*{2}}
\put(65,15){\vector(1,0){12}}
\put(65,15){\line(-1,-1){10}}
\put(55,5){\circle*{2}}
\put(55,5){\vector(1,0){12}}

\put(80,15){{\scriptsize$y_{1}$}}
\put(70,3){{\scriptsize$y_{4}$}}
\put(130,40){{\scriptsize$y_{3}$}}
\put(145,15){{\scriptsize$y_{2}$}}

\put(60,45){{\scriptsize$y_5$}}
\end{picture}
$$
\caption{Dual graph of \(C\).}
\label{fig:ex_0}
\end{figure}
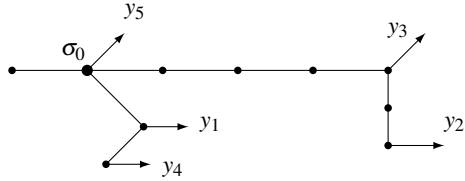

 The given ordering does not coincide with the good ordering we defined, so let us show how to order the topological Puiseux series according to the refined good ordering. The first separation point occurs at \(y_5=x^2\) and the partition at \(\sigma_0\) with this ordering is \(\ind=I_1\cup\ I_2\cup I_3\), where \(I_1=\{1,4\},\) \(I_2=\{2,3\}\) and \(I_3=\{5\}\). As we mentioned in Lemma \ref{lem:orderbetaiscompatible}, we can reorder the branches so that \(I'_1=\{1,2\},\) \(I'_2=\{3,4\}\) and \(I_3=\{5\}\). In doing so, our new ordering at this point is \(y_1'=y_1,\) \(y_2'=y_4,\) \(y_3'=y_3,\) \(y_4'=y_2\) and \(y_5'=y_5.\) Since the separation point is an ordinary point for all the branches, the dual graph of the truncation \(C_{\sigma_0}\) is
  \begin{figure}[h]
$$
\unitlength=0.50mm
\begin{picture}(80.00,40.00)(115,5)
\thinlines
\put(70,30){\line(1,0){20}}
\put(70,30){\circle*{2}}
\put(90,30){\circle*{2}}
\put(83,34){{\scriptsize$\sigma_0$}}
\put(90,30){\vector(1,0.5){12}}
\put(90,30){\vector(1,-0.1){13}}
\put(90,30){\vector(1,-1){10}}
\put(105,36){{\scriptsize$I_3$=\{5\}}}
\put(105,26){{\scriptsize$I_2$=\{2,3\}}}
\put(103,16){{\scriptsize$I_1$=\{1,4\}}}

\put(140,28){$\rightsquigarrow$}

\put(165,30){\line(1,0){20}}
\put(165,30){\circle*{2}}
\put(185,30){\circle*{2}}
\put(178,34){{\scriptsize$\sigma_0$}}
\put(185,30){\vector(1,0.5){12}}
\put(185,30){\vector(1,-0.1){13}}
\put(185,30){\vector(1,-1){10}}
\put(200,36){{\scriptsize$I_3$=\{5\}}}
\put(200,26){{\scriptsize$I_2'$=\{3,4\}}}
\put(198,16){{\scriptsize$I_1'$=\{1,2\}}}

\end{picture}
$$
\caption{Dual graph of \(C_{\sigma_0}\) and reordering of the packages.}
\label{fig:ex_1}
\end{figure}
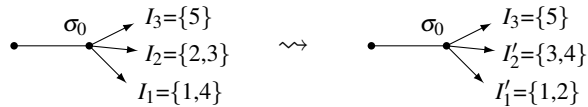

  and \(C_{\sigma_0}\) has three branches defined by \(y_{I_1}=2x^2,\) \(y_{I_2}=5x^2\) and \(y_{I_3}=x^2.\)
  We continue following the \(I_1\) package. The next term where the branches belonging to \(I_1\) separate is \(5>14/3\). Hence, to be good ordered, we must permute the indexing of both branches, namely \(y''_1=y_2'\) and \(y''_2=y_1'.\) Then the truncation at \(C_{\sigma_0^{I_1}}\) has four branches ordered as \(y_{I_{1,1}}=y_4=2x^2+x^5,\) \(y_{I_{1,2}}=y_1=2x^2+x^{14/3}+1/2x^5,\) \(y_{I_2}=5x^2\) and \(y_{I_3}=x^2.\) The dual graph is given in Figure \ref{fig:ex_2}.
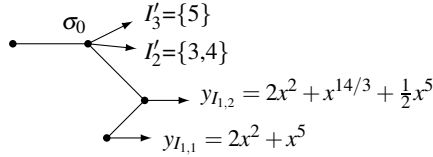
\begin{figure}[H]
$$
\unitlength=0.50mm
\begin{picture}(80.00,40.00)(85,4)
\thinlines
\put(70,30){\line(1,0){20}}
\put(70,30){\circle*{2}}
\put(90,30){\circle*{2}}
\put(83,34){{\scriptsize$\sigma_0$}}
\put(90,30){\vector(1,0.5){12}}
\put(90,30){\vector(1,-0.1){13}}
\put(90,30){\line(1,-1){15}}
\put(105,15){\circle*{2}}
\put(105,15){\vector(1,0){12}}
\put(105,15){\line(-1,-1){10}}
\put(95,5){\circle*{2}}
\put(95,5){\vector(1,0){12}}

\put(120,15){{\scriptsize$y_{I_{1,2}}=2x^2+x^{14/3}+\frac{1}{2}x^5$}}
\put(110,3){{\scriptsize$y_{I_{1,1}}=2x^2+x^5$}}

\put(105,36){{\scriptsize$I'_3$=\{5\}}}
\put(105,26){{\scriptsize$I'_2$=\{3,4\}}}
\end{picture}
$$
\caption{Dual graph of \(C_{\sigma_0^{I_1}}\).}
\label{fig:ex_2}
\end{figure}
As we have finished with the star points belonging to the geodesics of the branches of \(I_1,\) we move to the branches of \(I_2.\) The separation point of both branches is \(3>5/2\) so again we need to permute the branches \(y''_3=y'_4\) and \(y''_4=y'_3.\) Then, the truncation \(C_{\sigma_0^{I_2}}=C\) is our original curve and the branches are good ordered as \(Y_1:=y_{I_{1,1}}=y_4,\) \(Y_2:=y_{I_{1,2}}=y_1,\) \(Y_3:=y_{I_{2,1}}=y_2,\) \(Y_4:=y_{I_{2,2}}=y_3\) and \(Y_5:=y_{I_3}=y_5.\) Then we obtain the dual graph of \(C\) with the branches good ordered, as depicted in Figure \ref{fig:ex_22}.
  \begin{figure}[H]
$$
\unitlength=0.50mm
\begin{picture}(80.00,40.00)(50,5)
\thinlines
\put(30,30){\line(1,0){20}}
\put(30,30){\circle*{2}}
\put(50,30){\circle*{3}}
\put(43,34){{\scriptsize$\sigma_0$}}
\put(50,30){\vector(1,1){10}}
\put(50,30){\line(1,0){80}}
\put(70,30){\circle*{2}}
\put(90,30){\circle*{2}}
\put(110,30){\circle*{2}}
\put(130,30){\circle*{2}}
\put(130,30){\vector(1,1){10}}
\put(130,30){\line(0,-1){20}}
\put(130,20){\circle*{2}}
\put(130,10){\circle*{2}}
\put(130,10){\vector(1,0){15}}
\put(50,30){\line(1,-1){15}}
\put(65,15){\circle*{2}}
\put(65,15){\vector(1,0){12}}
\put(65,15){\line(-1,-1){10}}
\put(55,5){\circle*{2}}
\put(55,5){\vector(1,0){12}}

\put(80,15){{\scriptsize$Y_2$}}
\put(70,3){{\scriptsize$Y_1$}}
\put(130,40){{\scriptsize$Y_{4}$}}
\put(145,15){{\scriptsize$Y_{3}$}}

\put(60,45){{\scriptsize$Y_5$}}
\end{picture}
$$
\caption{Dual graph of \(C\) with the good order in the branches.}
\label{fig:ex_22}
\end{figure}
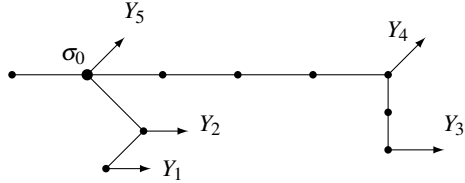
\end{ex}




\section{Poincaré series in terms of the minimal resolution}\label{sec:Poincareseries}

In this section we will first define the Poincar\'e series associated to a plane curve singularity $C=\bigcup_{i=1}^{r} C_i$, and describe it in terms of the dual graph of the minimal embedded resolution of $C$. After that, we will use the ordered sequence of approximating curves defined in the previous section to provide an iterative method to compute the Poincaré series of \(C.\)


\subsection{Poincaré series associated to the curve}\label{subsec:poincarebasic}
In the context of a discrete valuation, it is common to work with the Poincaré series associated to a filtration on the local ring. For \(\vu \in\mathbb{Z}^r\), set \(J_C(\vu)=J(\vu):=\{g\in\mathcal{O}\;|\;\vu(g)\geq \vu\}\). These are ideals which yield a multi-index filtration, and it makes sense to consider the quotiens $J(\vu)/J(\vu + \underline{1})$ which turn to be finite-dimensional $\mathbb{C}$-vector spaces of dimension $c(\vu)$; this leads to the consideration of 

$$
L_C(t_1,\dots,t_r)=\sum_{\vu\in\mathbb{Z}^r}c(\vu)\cdot \underline{t}^{\vu}.
$$

The dimensions $c(\underline{v})$ depend on $\Gamma(C)$ \cite[(3.5)]{MFjpaa}, hence $L_C(\underline{t})$ does so. We will abuse of notation and write $L_C$ rather than $L_{\Gamma(C)}$; this will be consistently done with all the objects occurring in the sequel.
\medskip

In the case of \(r=1,\) the series \(L_C(t)\) is the generating series of the value semigroup of $C$; however, for \(r>1\), this is not a (formal) power series, but  \(L(t_1,\dots,t_r)\in\mathbf{Z}[[t_1,\dots,t_r,t^{-1}_1,\dots,t^{-1}_r]]\); i.e. this is a Laurent series infinitely long in all directions since \(c(\vu)\) can be positive for \(\vu\) with some negative components \(v_i\) as well. As in \cite{CDGduke} we may check that 

$$
P'_C(t_1,\dots,t_r)=L_C(t_1,\dots,t_r)\cdot \prod_{i=1}^{r}(t_i-1)
$$

is a polynomial (if $r>1$) and moreover, it is divisible by \((t_1\cdots t_r-1)\). This lead to the definition of the Poincar\'e series associated to the multi-index filtration given by the ideals $J(\nu)$, thus to the curve $C$, as
\[
P_C(t_1,\dots,t_r)=P'_C(t_1,\dots,t_r)/(t_1\cdots t_r-1)
\]

which is in fact a polynomial if $r>1$.
\medskip

The univariate Poincar\'e series $P_C(t)$ is easily computed from the value semigroup $\Gamma(C)$, since $c(v)=1$ if and only if $v\in \Gamma(C)$. This is not longer true if $r>1$, and the computation of $P_C(t_1,\ldots , t_r)$ becomes complicated. There is a way to compute it in terms of the embedded resolution of $C$ due to Campillo, Delgado and Gusein-Zade \cite{CDG03a} by a formula which is analogous to that of A'Campo \cite{ACampo1,ACampo2} for the zeta function of the monodromy transformation of $C$; it may be also computed by using techniques of integration with respect the Euler characteristic again by Campillo, Delgado and Gusein-Zade, see e.g. \cite{CDG00, CDG02}. In fact, they show that the Poincar\'e series $P_{\widehat{\Gamma}_C}(t_1,\ldots , t_r)$ associated to the (projectivization of the) extended semigroup $\widehat{\Gamma}_C$, which is

$$
\chi(\mathbb{P}\widehat{\Gamma})=\sum\limits_{\underline{v}\in\mathbb{N}^r}\chi\big(\mathbb{P} F_{\underline{v}}\big)\cdot\underline{t}^{\underline{v}},
$$

(it is certainly possible to construct the projectivizations $\mathbb{P} F_{\underline{v}}=F_{\underline{v}}/\mathbb{C}^{\ast}$ of the fibre $F_{\underline{v}}$ \cite{CDGextended}),
coincides with the product over all the irreducible components of the exceptional divisor of the minimal resolution of $C$ of powers of cyclotomic polynomials of the form $(\underline{t}^{\underline{m}}-1)$, where $\underline{m}$ is the multiplicity of the liftings of the functions $f_i$ corresponding to the branches $C_i$ to the resolution space along the irreducible components of the exceptional divisor; in other words, we have
\begin{equation}\label{eq:duke}
    P_{\widehat{\Gamma}(C)}(t_1,\ldots , t_r) = \prod_{P \in G(C)} (t^{\underline{v}^{P}}-1)^{-\chi(E_{P}^{\circ})},
\end{equation}

where $\underline{v}^{P}$ stands for the value of the germ of a nonsingular curve which is transversal to $E_P$ in a smooth point of $E_P$. We further develop (\ref{eq:duke}) in terms of special points of the dual graph of $C$. First we need some notation. 
\medskip

Denote by $\mathcal{D}$ the set of dead arcs of $G(C)$. If $L \in \mathcal{D}$ then let $\rho_L$ (resp. $\sigma_L$)  be its end (resp. star) point. Also, let $\widetilde{\mathcal{D}} := \{L \in
\mathcal{D} \mid \sigma_L > \sigma_0\}$ be the set of dead arcs occurring after the first separation point $\sigma_0$ of $G(C)$. In addition, write $\widetilde{\mathcal{E}}$ for the set of ends for the dead arcs in $\widetilde{\mathcal{D}}$.
\medskip

For any $L \in \widetilde{\mathcal{D}}$ we know that $\underline{v}^{\sigma_L}	= (n_L + 1)\underline{v}^{\rho_L}$ for some integer $n_L \ge1$. We will denote also $n_{\rho} =n_L$ for $\rho=\rho_L \in \widetilde{\mathcal{E}}$.

\medskip

Let $L_0, \ldots ,L_q$ be the dead arcs of $G(C)$ with $\sigma_i =\sigma_{L_i} \le \sigma^0$ for $i \in \{1,\ldots , q\}$ ordered in such a way that $L_0$ has end point the vertex corresponding to $\mathbf{1}$, and $\sigma_1 < \ldots < \sigma_q$. Note that, if $q \ge 1$, then $\sigma_1$ is also the star point of the dead arc $L_0$ starting with $\mathbf{1}$. We denote also $\rho_i =\rho_{L_i}$ for $1\le i \le q$. As in the above case, let us denote by $n_i =n_{\rho_i}$ (for $i \in \{1, \ldots , q\}$) the integers such that $\underline{v}^{\sigma_i}= (n_i + 1)\underline{v}^{\rho_i}$. For the sake of completeness we set also $n_0 = n_{\mathbf{1}} = -1$. Note to which extend the divisor $\mathbf{1}$ and the integer $n_{\mathbf{1}}$ play a special role: if $q \ge 1$, then $\underline{v}^{\sigma_1}$ is also a multiple of $\underline{v}^{\mathbf{1}}$; this is, of course, different from $(n_{\mathbf{1}} + 1)\underline{v}^{\mathbf{1}}$.
\medskip

We define:
\begin{eqnarray}
P_1(\underline{t})  & := &  \frac{1}{\underline{t}^{\underline{v}^{\mathbf{1}}}-1} \cdot \prod_{i=1}^q \frac{\underline{t}^{\underline{v}^{\sigma_{i}}}-1}{\underline{t}^{\underline{v}^{\rho_{i}}}-1}\cdot (\underline{t}^{\underline{v}^{\sigma_0}}-1);  \nonumber \\
P_2(\underline{t})  & := & \prod_{\rho \in \widetilde{\mathcal{E}}}  \frac{\underline{t}^{(n_{\rho}+1)\underline{v}^{\rho}}-1}{\underline{t}^{\underline{v}^{\rho}}-1};  \nonumber \\
P_3(\underline{t})  & := & \prod_{s(\alpha) >1}  (\underline{t}^{\underline{v}^{\alpha}}-1)^{s(\alpha)-1}. \nonumber
\end{eqnarray}

\begin{theorem} \label{thm:p1p2p3}
Let $C$ be a plane curve singularity with dual graph $G(C)$, then
\[
P_C(t_1,\ldots , t_r)=P_1(t_1,\ldots , t_r)\cdot P_2(t_1,\ldots , t_r)\cdot P_3(t_1,\ldots , t_r).
\]
\end{theorem}

\begin{proof}
Campillo, Delgado and Gusein-Zade showed eq.~ (\ref{eq:duke}). Since 
\[
\chi(E_{P}^{\circ})=2-\big | \{\mbox{singular points of } E_P\}\big |=2-\nu (P),
\]
it remains only to compute those values $\nu (P)\neq 2$ for $P\in G(C)$. In view of the previous definitions, the statement follows straightforward.

\end{proof}

\begin{rem}
Theorem \ref{thm:p1p2p3} may be proven with a lot of effort from the explicit computation of the dimensions of the vector spaces $C(\underline{v})$ from the semigroup of values. This circumvents the use of the extended semigroup. 
\end{rem}

\subsection{Iterative construction of the Poincaré series}\label{subsec:iterativepoincare}

In this subsection, we will show how to compute the Poincaré series iteratively from the star points in the dual graph of the curve. This construction will become an extension of the irreducible case and it is highly inspired on it. 
\medskip

The construction in the irreducible case bases on an algebraic operation called gluing, which is available for both numerical and affine semigroups; since the semigroup of a plane curve with more than one branch is none of them, we do not have to our disposal the gluing operation. However, it is possible to reconstruct the process by following the paths marked by ordered the star points which now correspond with maximal contact values and values at proper star points. In this way, we are going to show that the decomposition of the Poincar\'e series given by Theorem \ref{thm:p1p2p3}  can be seen in terms of products of the following polynomials:
\begin{equation}\label{eqn:defkeypoly}
    \begin{array}{cc}
         P(m,n,x)=&\displaystyle \frac{x^{mn}-1}{x^m-1}\cdot\frac{x-1}{x^n-1}  \\[12pt]
         Q(m,n,x,y)=&\displaystyle \frac{(yx^m)^n-1}{yx^m-1}\\[12pt]
         B(m,n,x,y,z)=&(yx^m)^nz^m-1.
    \end{array}
\end{equation}

First we describe the case of one branch.

\subsubsection{The irreducible case}
In the case of a plane curve with a single branch, we have mentioned in Section \ref{sec:semigroupvalues} that its semigroup of values $\Gamma$ is a numerical semigroup minimally generated by \(\{\obeta_0,\dots,\beta_g\}.\) This numerical semigroup is a complete intersection numerical semigroup, therefore it can be constructed by a process defined by Delorme in \cite{Delormegluing} and called gluing by Rosales (see \cite{rosbook}); we explain it briefly.
\medskip

Let \(A=\{a_1, \dots,a_{g_1}\},B=\{b_1,\dots, b_{g_2}\}\) and \(C=\{c_1,\dots, c_{g_0}\}\) be three subsets of natural numbers. A semigroup \(S=\langle C \rangle\) in \(\mathbb{N}\) is said to be a gluing of \(S_1=\langle A\rangle\) and \(S_2=\langle B \rangle\) if its finite set of generators \(C\) splits into two parts, say \(C=k_1 A\sqcup k_2 B\) with \(k_1,k_2\geq 1\), and the defining ideals of the corresponding semigroup rings satisfy that \(I_C\) is generated by \(I_A+I_B\) and one extra element. We will denote the gluing of \(S_1\) and \(S_2\) via \((k_1,k_2)\) as \(k_1 S_1+ k_2S_2.\)
\medskip

The point is that we can construct $\Gamma$ as an iterated gluing. In the notation of Section \ref{sec:semigroupvalues}, we write \(p_i=n_i\) and \(w_i=\overline{\beta}_i/(n_{i+1}\cdots n_g)=\overline{\beta}_i/e_i\). Since \(\gcd(p_i,w_i)=1,\) we start with the numerical semigroup \(\Gamma_1=\langle p_1,w_1 \rangle.\) 
Now, we perform the gluing of \(\Gamma_1\) and the (trivial) semigroup \(\mathbb{N}\) via \((p_2,w_2)\) in order to obtain \(\Gamma_2=p_2 \Gamma_1+ w_2\mathbb{N}.\) It is easily seen that \(\Gamma_2=\langle p_1p_2, p_2w_1,w_2 \rangle\). Recursively, we define \(\Gamma_i=n_{i-1}\Gamma_{i-1}+ w_{i-1}\mathbb{N}\), and again it is a simple matter to check that
\[
\Gamma_i=\langle p_1\cdots p_i,w_1p_2\cdots p_i,\ldots , w_{i-1}p_i,w_i \rangle;
\]
in the end, we get that the semigroup of values is \(\Gamma=\Gamma_g.\) From this point of view, the knowledge of the minimal generators \(\obeta_0,\dots,\obeta_g\) is enough to provide the construction of \(\Gamma\) by gluing. Moreover, the gluing construction of the semigroup can be identified with the successive truncations of the topological Puiseux series of the branch: The key idea is that each star point in the dual graph of the branch defines a gluing operation in the semigroup in an ordered way.
\medskip


On the other hand, the Poincar\'e series $P_{\Gamma_i}(t)$ of $\Gamma_i=\langle \alpha_1,\ldots , \alpha_i \rangle$ is the Hilbert series $H_{k[\Gamma_i]}(t)$ of the (graded) semigroup ring $k[\Gamma_i]$ for a field $k$; this may be identified in the obvious way with the subalgebra $k[t^{\alpha_1},\ldots , t^{\alpha_i}]$ of the polynomial ring. Therefore
$$
 \Gamma_i=n_{i}\Gamma_{i-1} +  w_{i}\mathbb{N}.
$$
This means that 
\begin{equation}\label{eqn:tensordecompsemigr}
   k[\Gamma_{i}]\cong k[n_i\Gamma_{i-1}]\otimes_k k[w_i\mathbb{N}]/(u_1^{n_iw_i}\otimes 1 - 1 \otimes u_2^{n_iw_i}), 
\end{equation}
with \(u_1\) the uniformizing parameter of \(k[n_i\Gamma_{i-1}]\) and \(u_2\) the uniformizing parameter of \(k[w_i\mathbb{N}].\) Therefore
$$
H_{k[\Gamma_{i}]}(t)= (1-t^{n_iw_i})\cdot H_{k[n_i\Gamma_{i-1}]}(t)\cdot  H_{w_i\mathbb{N}}(t) = (1-t^{n_iw_i})\cdot H_{k[\Gamma_{i-1}]}(t^{n_i})\cdot \frac{1}{1-t^{w_i}}.
$$
This allows us to construct the Poincaré series of the gluing as 

\[
P_{\Gamma_{i}}(t)=(t^{n_{i}\overline{\beta}_{i}/e_{i}}-1)\cdot P_{\Gamma_{i-1}}(t^{n_{i}})\cdot P_{\mathbb{N}}(t^{\overline{\beta}_{i}/e_{i}}),
\]
which provides the well known expression for the Poincaré series of the numerical semigroup \(\Gamma_{i+1}.\) If we write 
 \(b_{l,m}:=\prod_{j=l}^{m}n_j\) with \(b_{l,m}=1\) if \(l>m,\) then the following Proposition is easily checked.

\begin{proposition}\label{prop:poincareirreduciblecase}
Let \(C\) be an irreducible plane curve singularity with semigroup \(\Gamma\). With the previous notation, we have
    \[
    P_C(t)=P_\Gamma(t)=(\prod^{g}_{j=1}P(\obeta_j/e_j,n_j,t^{b_{j+1,g}}))/(t-1).
    \]
\end{proposition}

\subsubsection{The base cases}

Before continuing, let us introduce some notation. Let \(C\) be a plane curve singularity with \(r\) branches and let \(\{(p_{i,j},m_{i,j})\;|\;1\leq i\leq k_i,\;1\leq j\leq r\}\) be the set of the Puiseux pairs of the topological Puiseux series of the branches. Recursively, we define 
 \begin{equation}\label{eqn:defw_j}
   w_{1,j}=m_{1,j}\quad \text{and}\quad w_{i,j}=m_{i,j}-m_{i-1,j}p_{i,j}+w_{i-1,j}p_{i-1,j}p_{
i,j}.  
 \end{equation}

Observe that, for \(r=1\) it is \(w_{i,1}=\obeta_i/e_i\), and for \(r>1\) we have that \(w_{i,j}\in\{\obeta^{j}_s/e^{j}_s,[f_j,f_l]/e^j_s\}\) for some \(1\leq s\leq g_j\) and \(1\leq l\leq r.\) 
\medskip

As we have seen, the irreducible case shows the computation of the Poincaré series for the sequence of approximating curves in the star points of the dual graph, obviously in the irreducible case all of them are non-proper. To mimic this process for non-irreducible plane curves we need first to prove two base cases (which are those with a single approximating curve which is the curve itself). The first one corresponds to a plane curve with \(r\) smooth branches all of them with the same contact. Its dual graph has a unique star point, which is in fact a proper star point.

\begin{proposition}\label{prop:poincarelemm1}
    Let \(C=\bigcup_{i=1}^{r} C_i\) such that \(C_i\) is smooth for all \(i\in \ind\) and such that the contact pair is equal for all branches, i.e. \((q,c)=(q_{i,j},c_{i,j})\) for all \(i,j\in\ind\). In particular, the topological Puiseux series of the branches are \(s_i(x)=a_ix^{c}\) with \(a_i\neq a_j\) for $i \neq j$. Then,

    \[
    P_C(\underline{t})=Q\big(c,1,t_1,\prod_{i=2}^{r}t_i^c\big)\cdot \Big(\prod_{i=2}^{r-1}B(1,c,t_i,\prod_{k<i}t_k,\prod_{k>i}t_k^c)\Big)\cdot Q\big(1,c,t_r,\prod_{i=1}^{r-1}t_i\big).
    \]
    
 \end{proposition}
\begin{proof}
    By hypothesis, the dual graph has only one star point, \(\sigma_0.\) Moreover, the value at \(\sigma_0\) is \(v^{\sigma_0}=(c,\dots,c)\) and it has \(s(\sigma_0)=r-1.\) Also, \(v^{\mathbf{1}}=(1,\dots,1)\). Then,
\begin{equation*}
    \begin{split}
       & Q\big(c,1,t_1,\prod_{i=2}^{r}t_i^c\big)\cdot \Big(\prod_{i=2}^{r-1}B(1,c,t_i,\prod_{k<i}t_k,\prod_{k>i}t_k^c)\Big)\cdot Q\big(1,c,t_r,\prod_{i=1}^{r-1}t_i\big)\\
       =& \frac{\displaystyle\prod_{i=1}^{r}t_i^c-1}{\displaystyle\prod_{i=1}^{r}t_i^c-1}\cdot\bigg(\displaystyle\prod_{i=2}^{r-1}\Big(\big(t_i \prod_{k<i}t_k\big)^c\prod_{k>i}t_k^c-1\Big)\bigg)\cdot\displaystyle\frac{\displaystyle\prod_{i=1}^{r}t_i^c-1}{\displaystyle\prod_{i=1}^{r}t_i-1}\\
       =&\frac{\displaystyle\prod_{i=1}^{r-1}\big(\prod_{k=1}^{r}t_k^c-1\big)}{\displaystyle\prod_{k=1}^{r}t_k-1}=\frac{\underline{t}^{\vu^{\sigma_0}}-1}{\underline{t}^{\vu^{\mathbf{1}}}-1}\cdot (\underline{t}^{\vu^{\sigma_0}}-1)^{r-2}=P_C(\underline{t}),
    \end{split}
\end{equation*}
 as desired.
\end{proof}
We can extend the proof of Proposition \ref{prop:poincarelemm1} to the case where at least one of the branches has one Puiseux pair. This constitutes the second base case, which is the case of a plane curve singularity with \(r\) branches with at most one characteristic exponent all of them with contact \((q_{i,j},c_{i,j})\in\{(1,0),(0,l)\}.\) The condition on the contact implies that again the curve itself is the only approximating curve.

\begin{proposition}\label{prop:poincarelemm2}
    Let \(C=\bigcup_{i=1}^{r} C_i\) be a curve with \(\Gamma^i=\langle\obeta_0^i,\obeta^i_1\rangle\) or \(\Gamma^i=\mathbb{N}.\) Assume that for every index \(i\in \ind\) such that \(C_i\) is a singular branch we have \(l:=\Big\lfloor\frac{\obeta^i_1}{\obeta^i_0}\Big\rfloor=\Big\lfloor\frac{\obeta^j_1}{\obeta^j_0}\Big\rfloor\) if \(j\neq i\) and \(C_j\) is singular. Moreover, assume that the contact pairs are of the form \((q_{i,j},c_{i,j})\in\{(1,0),(0,l)\}.\) Then the Poincaré series $ P_C(\underline{t})$ is equal to the product
    \[
   Q\big(w_{1,1},p_{1,1},t_1,\prod_{i=2}^{r}t_i^{w_{1,i}}\big)\cdot\bigg(\prod_{i=2}^{r-1}B\big(p_{1,i},w_{1,i},t_i,\prod_{k<i}t_k^{p_{1,k}},\prod_{k>i}t_{k}^{w_{1,k}}\big )\bigg)\cdot Q\big(p_{1,r},w_{1,r},t_r,\prod_{i=1}^{r-1}t_i^{p_{1,i}}\big).
    \]
\end{proposition}

\begin{proof}
We can assume that at least one branch is singular (otherwise we would be in the conditions of Proposition \ref{prop:poincarelemm1}). Without loss of generality, we assume further that the branches are ordered with the refinement of the good ordering introduced in Subsection \ref{subsec:totalorderstar}; this means in particular that the last branch \(f_r\) is singular, and that \(\sigma_0\) is the unique star point in \(G(C_r)\); this implies that \(\nu^{\sigma_0}=p_{1,r}\eta^{(r)}\), where 
\[
\pr_i(\eta^{(r)})=\left\{\begin{array}{ll}
     \ \obeta^r_1,& \text{if}\quad (C_i\mid C_r)=(1,0); \\
     \ [C_i,C_r]/p_{1,r},& \text{otherwise}.
\end{array}\right.
\]

By the assumption that the contact pairs are of the form \((q_{i,j},c_{i,j})\in\{(1,0),(0,l)\}\), both the Noether formula \ref{noehterformula} and the ordering in the set of branches imply the equality \([C_i,C_j]=p_{1,i}w_{1,j}\) if \(i<j\). Therefore, we have 
\[
Q\big(p_{1,r},w_{1,r},t_r,\prod_{i=1}^{r-1}t_i^{p_{1,i}}\big)=\frac{\underline{t}^{\nu^{\sigma_0}}-1}{\underline{t}^{\nu^{\mathbf{1}}}-1}.
\]

    Suppose first that all the branches are singular, as in Figure \ref{fig:prop2-case1}. In this case, there exists a unique dead end in the dual graph, which we denote by \(T\). The hypothesis on the contact pairs leads to the existence of a maximal contact value of the form \(\nu^{T}=(\obeta^1_1,\dots,\obeta^r_1).\) Moreover, the total order in the set of star points shows that \(\mathcal{S}=\mathcal{S}_1\) in this case, and the arrow corresponding to \(f_1\) goes through the maximal star point in \(\mathcal{S}\); this means that \((n_T+1)\nu^{T}=p_{1,1}\nu^{T}.\) 
We have 

\[
P_1(\tu)\cdot P_2(\tu)=Q\big(w_{1,1},p_{1,1},t_1,\prod_{i=2}^{r}t_i^{w_{1,i}}\big)\cdot Q\big(p_{1,r},w_{1,r},t_r,\prod_{i=1}^{r-1}t_i^{p_{1,i}}\big).
\]

Assume now that there is at least one smooth branch, as in Figure \ref{fig:prop2-case2}. The hypothesis on the contact means that there are no dead ends in the dual graph of \(C\); this makes the factor $P_2(\tu)$ trivial, namely  \(P_2(\tu)=1.\) On the other hand, the good order implies that \(f_1\) is smooth so that \(Q(w_{1,1},p_{1,1},t_1,\prod_{i=2}^{r}t_i^{w_{1,i}})=1\), since \(p_{1,1}=1.\) Therefore, again we have 

\[
P_1(\tu)\cdot P_2(\tu)=Q\big(w_{1,1},p_{1,1},t_1,\prod_{i=2}^{r}t_i^{w_{1,i}}\big)\cdot Q\big(p_{1,r},w_{1,r},t_r,\prod_{i=1}^{r-1}t_i^{p_{1,i}}\big).
\]

What is left is to show 
\[
P_3(\tu)=\prod_{i=2}^{r-1}B\big(p_{1,i},w_{1,i},t_i,\prod_{k<i}t_k^{p_{1,k}},\prod_{k>i}t_{k}^{w_{1,k}}\big),
\]
with independence of the existence of a smooth branch. To complete the proof, we proceed as in the case of \(\sigma_0.\)  Let \(T\) be the point in \(G(C)\) such that \(T\) is the unique end point of \(G(C_{\jnd_{sing}})\) where \(C_{\jnd_{sing}}\) is the package of singular branches. As explained in Subsection \ref{subsec:totalorderstar}, we can decompose \(\jnd_{sing}=\bigcup_{t+1}^{s}J_s\) with the packages defined in \((\star\star\star).\) Consider \(R\in\mathcal{R}\) and denote by \(\jnd\subset\ind\) the set of indices such that the geodesic to an arrow associated to \(j\in\jnd\) finishes at \(R\); this means that there are \(s_R\) packages such that \(\jnd=J_{l_1}\cup\cdots\cup J_{l_1+s_R}.\) Then, \(\nu^R=p_{1,j}\nu^T\) for all \(j\in \jnd\) and 

\[
\pr_i(\nu^T)=\left\{\begin{array}{ll}
     \ \obeta^\jnd_1,& \text{if}\quad i\in\jnd \\
     \ [C_i,C_\jnd]/p_{1,\jnd},& \text{if}\quad i\notin\jnd
\end{array}\right.
\]

For each \(j\in\jnd\) we have a factor 
\[
B\big (p_{1,j},w_{1,j},t_j,\prod_{k<j}t_k^{p_{1,k}},\prod_{k>j}t_{k}^{w_{1,k}}\big)=(\tu^{\nu^R}-1).
\] 

Therefore, the claim follows from Proposition \ref{prop:numbersq} taking into account that each of the previous factors appears \(s_R=s(R)-1\) times.
\end{proof}

\begin{rem}\label{rem:QB1}
    Observe that the polynomials \(Q,B\) appear in a natural way. The polynomial \(Q\) is associated to an end point, the easiest examples for the understanding of the \(Q\)--factors are the following:
    \begin{itemize}
        \item[$\bullet$] Assume \(C\) is an irreducible plane curve with semigroup \(\Gamma=\langle\alpha,\beta\rangle=\alpha\mathbb{N}+\beta\mathbb{N}\) with \(\alpha<\beta.\) It has two end points: the one associated to \(\alpha,\) which is the vertex \(\mathbf{1}\), and the one associated to \(\beta\) which defines the unique dead arc of the graph. Then the Poincaré series is decomposed as
        \[
        \frac{1}{t^\alpha-1}\frac{t^{\alpha\beta}-1}{t^{\beta}-1}=Q(\alpha,0,1,t^{\alpha})Q(\beta,\alpha,t,1).
        \]
        \item[$\bullet$] Assume \(C\) is a plane curve with two branches in the conditions of Proposition \ref{prop:poincarelemm2} and semigroups \(\Gamma^1=\langle\obeta_0^{1},\obeta_1^{1}\rangle,\)  \(\Gamma^2=\langle\obeta_0^{2},\obeta_1^{2}\rangle.\) It has two end points: the one associated to the vertex \(\mathbf{1}\) and the one associated to the end point \(T\) of the unique dead arc of the graph. The good ordering implies that the valuation at the free points up to \(T\) are multiple of \(\beta_0^1\) which provide the factor
        \[Q(\obeta_1^{1},\obeta_0^{1},t_1,t_2^{\obeta_1^{2}})=\frac{(t_1^{\obeta_1^{1}}t_2^{\obeta_1^{2}})^{\obeta_0^{1}}-1}{t_1^{\obeta_1^{1}}t_2^{\obeta_1^{2}}-1}.\]
        In the case where one of the branches is smooth then the graph has a single end point, which is the vertex \(\mathbf{1}.\) This is also reflected in the factor, as \(\beta_0^1=1\) and the \(Q\)--factor does not appear effectively. In a similar way, the other \(Q\)--factor is deduced to obtain the formula of Proposition \ref{prop:poincarelemm2}.
        \item[$\bullet$] Assume \(C\) is plane curve in the conditions of Proposition \ref{prop:poincarelemm2}. Again we have at most two end points, so by Noether's formula \ref{noehterformula}--- and following the previous observations of the bibranch case---we have a \(Q\)--factor associated to each of the end points. The factors \(B\) arise naturally in between these two $Q$--factors, and are associated to proper star points. Since proper star points do not produce a geodesic to an end point, $B$-factors cannot contain a denominator. Therefore, we need to associate a polynomial \(B\) to the \(i\)--th variable that---with the help of the good order---encodes the two different types of multiples that we can have at this point of the valuation, namely 
        \[
B\Big(\obeta_0^{i},\obeta_1^{i},x=t_i,y=\prod_{k<i}t_k^{\obeta_0^{k}},z=\prod_{k>i}t_k^{\obeta_1^{k}}\Big)=(yx^{\beta_0^{i}})^{\beta_1^{i}}z^{\beta_0^{i}}-1.
        \]
        Observe what this factor shows the fact that the valuation at each proper star point of each branch is the crossed product of multiplicities and maximal contact values.

    \end{itemize}
\end{rem}

\begin{figure}
     \begin{subfigure}[b]{0.5\textwidth}
         \centering
       $$
\unitlength=0.50mm
\begin{picture}(70.00,110.00)(-70,-30)
\thinlines

\put(-100,30){\line(1,0){38}}
\put(-100,30){\circle*{2}}

\put(-85,30){\circle*{2}}
\put(-70,30){\circle*{2}}

\put(-57,30){\circle*{2}}

\put(-62,30){\line(1,0){18}}

\put(-44,30){\circle*{2}}
\put(-44,30){\line(1,-1){30}}

\put(-44,30){\vector(1,1){15}}
\put(-44,30){\vector(1,0.6){20}}
\put(-44,30){\vector(1,0.2){17}}

\put(-35.5,21){\circle*{2}}
\put(-35.5,21){\vector(1,-0.2){9}}
\put(-35.5,21){\vector(1,0.5){9}}
\put(-35.5,21){\vector(1,1){6}}

\put(-24.5,10){\circle*{2}}
\put(-24.5,10){\vector(1,1.3){9}}
\put(-24.5,10){\vector(1,0.4){9}}
\put(-24.5,10){\vector(1,2){8}}

\put(-14,0){\circle*{2}}
\put(-14,0){\vector(1,1.7){7.5}}
\put(-14,0){\vector(1,-1.3){9}}

\put(-14,0){\vector(1,.7){14}}

\put(-14,0){\line(-1,-1){13}}
\put(-27,-13){\circle*{2}}

\put(-34,-15){\scriptsize{$T$}}

\end{picture}
$$
\caption{All the branches being singular.}
\label{fig:prop2-case1}
\end{subfigure}
     \hfill
\begin{subfigure}[b]{0.5\textwidth}
        \centering
         $$
\unitlength=0.50mm
\begin{picture}(70.00,110.00)(-70,-30)
\thinlines

\put(-100,30){\line(1,0){38}}
\put(-100,30){\circle*{2}}

\put(-85,30){\circle*{2}}
\put(-70,30){\circle*{2}}

\put(-57,30){\circle*{2}}

\put(-62,30){\line(1,0){18}}

\put(-44,30){\circle*{2}}
\put(-44,30){\line(1,-1){30}}

\put(-44,30){\vector(1,1){15}}
\put(-44,30){\vector(1,0.6){20}}
\put(-44,30){\vector(1,0.2){17}}

\put(-35.5,21){\circle*{2}}
\put(-35.5,21){\vector(1,-0.2){9}}
\put(-35.5,21){\vector(1,0.5){9}}
\put(-35.5,21){\vector(1,1){6}}

\put(-24.5,10){\circle*{2}}
\put(-24.5,10){\vector(1,1.3){9}}
\put(-24.5,10){\vector(1,0.4){9}}
\put(-24.5,10){\vector(1,2){8}}

\put(-14,0){\circle*{2}}
\put(-14,0){\vector(1,1.7){7.5}}
\put(-14,0){\vector(1,-0.6){9}}

\put(-14,0){\vector(1,.7){14}}

\put(-14,0){\line(-1,-1){13}}
\put(-27,-13){\circle*{2}}

\put(-27,-13){\vector(1,-1.3){9}}
\put(-27,-13){\vector(1,0){11}}

\put(-34,-15){\scriptsize{$T$}}

\end{picture}
$$
        \caption{At least one singular branch.}
    \label{fig:prop2-case2}
      \end{subfigure}
      \caption{Some base cases.}
    \label{fig:ddd}
\end{figure}
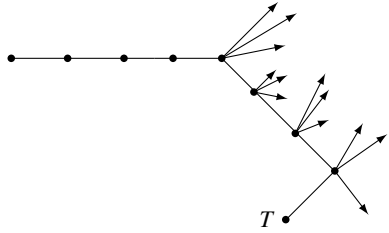
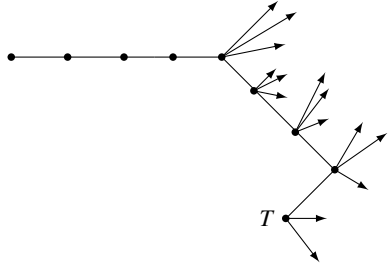

\subsubsection{The general procedure}\label{subsec:generalprocedure}
We are now ready to show the iterative construction of the Poincaré series. Following the notation of Subsection \ref{subsec:totalorderstar}, let \(\mathcal{S}=\mathcal{E}\cup\mathcal{R}\) be the set of star points of the dual graph \(G(C)\) ordered by the total order. We provide an iterative procedure to calculate the Poincaré series based on the computation of the Poincaré series of the truncations at the star points defined in Subsection \ref{subsubsec:truncationstar}. We show that we can compute the Poincaré series of \(C\) iteratively from the Poincaré series of the truncations at the star points.
\medskip

Let \((q,c)=(f_1|\cdots|f_r)\) be the contact pair of \(C.\) Obviously if \(C\) is a plane curve which has only one approximating curve, i.e. \(C\) is the unique approximating curve, then we are in the conditions of Proposition \ref{prop:poincarelemm1} or Proposition \ref{prop:poincarelemm2}, and there is nothing to prove. Therefore, we can assume that \(C\) is a plane curve with at least one approximating curve \(C_{\alpha_1}\neq C.\)
\medskip

 As in Subsection \ref{subsubsec:truncationstar}, let \(\alpha_1\prec\cdots \prec \alpha_{q-1}\) be the first star points in \(\mathcal{S}_1\) which are known to be common to all the branches. For \(1\leq k\leq q-1\) we define the semigroup 
\[
\Gamma^{1}_k=\Big \langle\frac{\obeta^i_0}{e^{i}_k},\dots ,\frac{\obeta^i_k}{e^{i}_k}\Big \rangle,
\]
which is independent of \(i\in\ind\) since the contact pair of \(C\) is \((q,c).\) For \(i=1,\dots, q-1\) let  \(C_{\alpha_i}\) be the irreducible curve associated to the star point \(\alpha_i\) defined in Subsection \ref{subsubsec:truncationstar} and write \(P_{\alpha_i}(t):=P_{C_{\alpha_i}}=P_{\Gamma^{1}_i}\) for the Poincaré series of the plane curve \(C_{\alpha_i}\). By Proposition \ref{prop:poincareirreduciblecase}, we have 
\[
P_{\alpha_i}(t)=\frac{\displaystyle\prod_{k=1}^{i}(t^{n^1_k\obeta^1_k/e^1_i}-1)}{\displaystyle\prod_{k=0}^{i}(t^{\obeta^1_k/e^1_i}-1)}.
\]
Using the gluing property of a numerical semigroup, for \(i=1,\dots,q-2\) we obtain
\[
P_{\alpha_i}(t)=P_{\alpha_{i-1}}(t^{p_{i,1}})\cdot \frac{(t^{p_{i,1}}-1)}{t-1}\cdot P(w_{i,1},p_{i,1},t).
\]

We now proceed as in Subsection \ref{subsubsec:truncationstar} and denote by \(T=\alpha_{q-1}\) and \(\sigma=\alpha_q.\)  We will now show how to compute the Poincar\'e series of \(C_\sigma\) from the Poincaré series of \(C_T.\) As in Subsection \ref{subsubsec:truncationstar}, we need to distinguish two cases:

\begin{enumerate}[wide, labelwidth=!, labelindent=0pt]
    \item [\(\bullet\)] If \(c>0\), then  \(C_{\alpha_q}\) is an irreducible plane curve and we can compute its Poincaré series as in the previous case 
    \[
    P_{\sigma}(t)=P_{T}(t^{p_{q,1}})\cdot \frac{(t^{p_{q,1}}-1)}{t-1}\cdot P(w_{q,1},p_{q,1},t).
    \]
    \item [\(\bullet\)] If \(c=0\), then \(\sigma=\alpha_q=\sigma_0\) and we need to consider the partition  \(\displaystyle\ind=( \bigcup_{p=1}^{t} I_{p})\cup(\bigcup_{p=t+1}^{s} I_p)\). Hence \(C_{\sigma}\) is a plane curve with with \(s\) branches, and we have 

    \begin{lemma} \label{lem:aux1}
    \[
    \begin{split}
    P_\sigma(t_1,\dots,t_s)=&P_T(t_1^{p_{q,I_1}}\cdots t_s^{p_{q,I_s}})\cdot Q\big(w_{q,I_1},p_{q,I_1},t_1,\prod_{k=2}^{s} t_k^{w_{q,I_k}}\big)\\
    \cdot&\bigg(\prod_{j=2}^{s-1}B(p_{q,j},w_{q,j},t_j,\prod_{k<j}t_k^{p_{q,k}},\prod_{k>j}t_{k}^{w_{q,k}})\bigg)\cdot \Big((\prod_{k=1}^{s}t_k^{p_{q,k}})^{w_{q,s}}-1\Big).
    \end{split}
    \]
    \end{lemma}
    \begin{proof}
        The proof goes in an analogous manner as that of Proposition \ref{prop:poincarelemm2}. By Theorem \ref{thm:p1p2p3} we have a decomposition of \(P_\sigma(t_1,\dots,t_s)=P_1\cdot P_2\cdot P_3.\) First, since \(p_{q,I_k}=n_{q,I_k}=e^{I_k}_{q-1}/e^{I_k}_q\) and \(T=\alpha_{q-1}\), we have that 
        \[
        P_T(t_1^{p_{q,I_1}}\cdots t_s^{p_{q,I_s}})=\frac{\displaystyle\prod_{k=1}^{q-1}\big((t_1\cdots t_s)^{n^1_k\obeta^1_k/e^1_q}-1\big)}{\displaystyle\prod_{k=0}^{q-1}\big((t_1\cdots t_s)^{\obeta^1_k/e^1_q}-1\big)}.
    \]
        A similar analysis to that in the part of the proof of Proposition \ref{prop:poincarelemm2} where all the branches are singular, shows that

        \[
    P_1(\underline{t})\cdot P_2(\underline{t})=P_T(t_1^{p_{q,I_1}}\cdots t_s^{p_{q,I_s}})\cdot Q\big(w_{q,I_1},p_{q,I_1},t_1,\prod_{k=2}^{s} t^{w_{q,I_k}}\big) \cdot \Big((\prod_{k=1}^{s}t_k^{p_{q,k}})^{w_{q,s}}-1\Big).
        \]
The only difference with the proof of Proposition \ref{prop:poincarelemm2} lies on the fact that, in the current case, the factor \(\underline{t}^{v^{\mathbf{1}}}-1\) is already contained in \(P_T(t_1^{p_{q,I_1}}\cdots t_s^{p_{q,I_s}}).\) Therefore, instead of adding a factor \(\displaystyle Q\Big(p_{q,s},w_{q,s},t_s,\prod_{k=1}^{s-1}t_k^{p_{q,k}}\Big),\) we only need to add the factor corresponding to \(\underline{t}^{v^{\sigma_0}}-1.\)
\medskip

        It remains to prove that \(P_3(\underline{t})=\displaystyle \prod_{j=2}^{s-1}B\big (p_{q,j},w_{q,j},t_j,\prod_{k<j}t_k^{p_{q,k}},\prod_{k>j}t_{k}^{w_{q,k}}\big).\) To do that, we first observe that since \(c=0\) then \(p_{q,j}=e^{j}_{q-1}/e^{j}_q\) and \(w_{q,j}=\obeta^j_q/e^j_q\) are independent of \(j\) and hence the factor
        \[
        B\Big(p_{q,j},w_{q,j},t_j,\prod_{k<j}t_k^{p_{q,k}},\prod_{k>j}t_{k}^{w_{q,k}}\Big)=\big(\prod_{k<j}t_k^{p_{q,j}}t_k^{p_{q,j}}\big)^{\obeta^j_q/e^j_q}\big(\prod_{k>j}t_{k}^{\obeta^j_q/e^j_q}\big)^{p_{q,j}}-1=(t_1\cdots t_s)^{p_{q,j}\obeta^j_q/e^j_q}-1
        \]
        is repeated \(s-2\) times. By the definition of \(C_\sigma\) and the fact that \(c=0\), we have  \(v^{\sigma}=v^{\sigma_0}=p_{q,j}w_{q,j}(1,\dots,1)\) and the valency of \(\sigma_0\) is \(s+2.\) Since \(\sigma_0\) is the only point proper star point of \(G(C_\sigma)\) and \(s(\sigma_0)-1=v(\sigma_0)-4=s-2\) in this case, the claim follows.
    \end{proof}
\end{enumerate}
If \(c=0,\) then after computing the Poincaré series \(P_\sigma\) we set \(T=\sigma_0\) and \(\sigma\) is the next (with respect to the ordering \(<\) in \(\mathcal{S}\)) star  point to be considered. If \(c\neq 0\), then we set \(T=\alpha_q\) and \(\sigma=\sigma_0.\) In this case, we compute the Poincaré series similarly to the case \(c=0:\)

\begin{lemma} \label{lem:aux2}
    \[\begin{split}
    P_\sigma(t_1,\dots,t_s)=&P_T(t_1^{p_{q,I_1}}\cdots t_s^{p_{q,I_s}})\cdot Q\big(w_{q,I_1},p_{q,I_1},t_1,\prod_{k=2}^{s} t_k^{w_{q,I_k}}\big)\\
    \cdot&\Big(\prod_{j=2}^{s-1}B(p_{q,j},w_{q,j},t_j,\prod_{k<j}t_k^{p_{q,k}},\prod_{k>j}t_{k}^{w_{q,k}})\Big)\cdot \big((\prod_{k=1}^{s}t_k^{p_{q,k}})^{w_{q,s}}-1\big).
    \end{split}
    \]
    \end{lemma}
    \begin{proof}
        The proof is analogous to the proof of Lemma \ref{lem:aux1}~resp.~Lemma \ref{prop:poincarelemm2}. The only difference with respect to the proof of Lemma \ref{lem:aux1} is that the \(p_{q,I_k}\) are now different and we need to proceed as in the proof of Proposition \ref{prop:poincarelemm2} to explicitly compute the polynomial \(P_3.\)
    \end{proof}
    \begin{rem}
        Observe that Lemma \ref{lem:aux2} is the natural interpretation of the base cases when the branches have contact bigger than \((1,0).\) As we explained in Remark \ref{rem:QB1}, the \(Q\)--factors are associated to end points, since in this case the approximating curve \(C_\sigma\) has at most \(1\) more end point than \(C_T\) then only one \(Q\)--factor is required. The last factor is not surprising as it is in fact a \(B\)--factor:
        \[B(p_{q,s},w_{q,s},t_s,\prod_{k<s}t_k^{p_{q,k}},1)=\big((\prod_{k=1}^{s}t_k^{p_{q,k}})^{w_{q,s}}-1\big).\]
    \end{rem}

We shall continue computing the Poincaré series of the approximations of \(C.\) To do that, we follow the procedure to construct the approximations described in Subsection \ref{subsubsec:truncationstar}. The distinguished points \(T,\sigma\) are now at the stage \(T=\sigma_0\) and \(\sigma\) is the next start point to be considered for the computation of the Poincaré series. 
Let \(\ind=(\bigcup_{p=1}^{t} I_{p})\cup(\bigcup_{p=t+1}^{s} I_p)\) be the partition created at \(\sigma_0.\) Denote by \(\osigma_1,\dots,\osigma_\epsilon\) the star points between \(\sigma_0\) and the point $P$ where the geodesics of \(I_1\) go through. At this point we have 

\[
\alpha_1\prec \cdots\prec\alpha_q\preceq \sigma_0\preceq \osigma_1\preceq\cdots\preceq\osigma_\epsilon\preceq P.
\]
Then,

\begin{enumerate}[wide, labelwidth=!, labelindent=0pt]
    \item Assume \(|I_1|=1;\) we distinguish two cases:
 \begin{enumerate}
 \item  The semigroup \(\Gamma^1\) of the first branch \(C^1\) of \(C\) has \(q\) minimal generators. This implies that \(\mathcal{S}_1=\{\alpha_1\prec \cdots\prec\alpha_q\preceq \sigma_0\preceq \osigma_1\preceq\cdots\preceq\osigma_\epsilon\}\) and \(y^{I_1}_{\sigma_0}\) is the topological Puiseux series of the branch \(C^1.\) Then, there is no need to perform any computation in this step and we have finished with \(\mathcal{S}_1\). We move to the package \(I_2\) and we make \(\sigma\) the star point from which the geodesics of \(I_2\) go through and \(T=\sigma_0.\)
        \item The semigroup \(\Gamma^1\) of the first branch \(C^1\) of \(C\) has \(g_1>q\) minimal generators. Then,
        \[
        \mathcal{S}_1=\{\alpha_1\prec \cdots\prec\alpha_q\preceq \sigma_0\preceq \osigma_1\preceq\cdots\preceq\osigma_\epsilon\prec \alpha_1^{I_1}\prec\cdots \prec \alpha_{g_1-q}^{I_1}\}
        \]
        where \(\alpha_1^{I_1}\prec\cdots \alpha_{g_1-q}^{I_1}\) are the non-proper star points defining the maximal contact values associated to the remaining generators of the semigroup \(\Gamma^1\). We are in the case \(T=\sigma_0\) and \(\sigma=\alpha_1^{I_1};\) for \(i=2,\dots,g_1-q\) we will consider \(T=\alpha^{I_1}_{i-1}\) and \(\sigma=\alpha^{I_1}_i.\) At each stage, we need to compute the Poincaré series \(P_{C_{\alpha^{I_1}_i}}=P_{\alpha^{I_1}_i}=P_\sigma\) from the Poincaré series of \(P_T=P_{C_{\alpha^{I_1}_{i-1}}}.\) At each stage, to simplify notation let us denote by \(p_\sigma:=p_{q+i,I_1}\) and \(w_\sigma:=w_{q+i,I_1}.\) Now, for all \(j\notin I_1,\) i.e.~\(j\in I_k\) for some \(k=2,\dots,s,\) by definition of \(C_\sigma\) we have that \(w_{q+i,I_k}:=w_{q+i,j}=[f_{I_1},f_j]/e^{I_1}_{q+i-1}.\) Then,
        \begin{proposition} \label{prop:case1b}
            \[
            P_{\sigma}(\underline{t})=P_T(t_1^{p_\sigma},t_2,\dots,t_s)\cdot Q\big(w_\sigma,p_\sigma,t_1,\prod_{k=2}^{s} t_k^{w_{q+i,I_k}}\big).
            \]
        \end{proposition}
        \begin{proof}
        First of all, we observe that the dual graph \(G(C_\sigma)\) can be obtained from \(G(C_T)\) by adding a single dead arc corresponding to \(\sigma,\) thus \(G(C_T)\) is a subgraph of \(G(C_\sigma).\) Therefore, the set of proper star points in \(G(C_\sigma)\) is equal to the set of proper star points in  \(G(C_T)\). Moreover, if we denote by \(\mathcal{S}_\sigma\) the set of star points of  \(G(C_\sigma)\) and \(\mathcal{S}_T\) the set of star points of  \(G(C_T),\) then \(\mathcal{S}_\sigma\setminus\mathcal{S}_T=\{\sigma\}\in \widetilde{\mathcal{E}}_\sigma.\) Therefore, by Noether's formula \ref{noehterformula} and the definition of the curves \(C_\sigma\) and \(C_T\) we have that 
        \begin{align*}
              P_T(&t_1^{p_\sigma},t_2,\dots,t_s)=\\
             &= \frac{1}{\underline{t}^{\underline{v}^1}-1}\cdot  \prod_{i=1}^q \frac{\underline{t}^{\underline{v}^{\sigma_{i}}}-1}{\underline{t}^{\underline{v}^{\rho_{i}}}-1}\cdot (\underline{t}^{\underline{v}^{\sigma^0}}-1) \cdot \prod_{\rho \in \widetilde{\mathcal{E}}\setminus\{\sigma\} } \frac{\underline{t}^{(n_{\rho}+1)\underline{v}^{\rho}}-1}{\underline{t}^{\underline{v}^{\rho}}-1} \cdot \prod_{s(\alpha) >1}  (\underline{t}^{\underline{v}^{\alpha}}-1)^{s(\alpha)-1}.
        \end{align*}


        Therefore, the proof is completed by showing that 
        \[
        Q\big(w_\sigma,p_\sigma,t_1,\prod_{k=2}^{s} t^{w_{q+i,I_k}}\big)=\frac{\underline{t}^{(n_{\sigma}+1)\underline{v}^{\sigma}}-1}{\underline{t}^{\underline{v}^{\sigma}}-1}.
        \]

        Let \(L^1_{q+i}\) be the dead arc associated to the star point \(\sigma\) and \(P(L^1_{q+i})\) its end point. Then, by eq.~ \eqref{eqn:maximalcontactvalues} in Subsection \ref{subsec:maximalcontactdualgrapha}, the Noether formula \ref{noehterformula}, and the definition of \(C_\sigma\) we have that 

        \[
        \pr_j(\vu^\sigma)=\left\{\begin{array}{cc}
    \ \obeta^{1}_{q+i}/e^{1}_{q+i},
    &\text{if}\;j=1; \\[.3cm]
     \frac{[f_1,f_j]}{e^{1}_ {q+i-1}},& \text{if}\;j\notin I_1.
     \end{array}\right.
     \]
\begin{figure}[H]
$$
\unitlength=0.50mm
\begin{picture}(70.00,110.00)(0,-50)
\thinlines

\put(-100,30){\line(1,0){24}}
\put(-100,30){\circle*{2}}

\put(-80,30){\circle*{2}}
\put(-80,30){\line(0,-1){15}}
\put(-80,15){\circle*{2}}

\put(-57,30){\circle*{2}}
\put(-57,30){\line(0,-1){15}}
\put(-57,15){\circle*{2}}

\put(-74,30){$\ldots$}

\put(-62,30){\line(1,0){18}}

\put(-63,35){\scriptsize{$T=\sigma_0$}}
\put(-8,23){$\ddots$}

\put(-44,30){\circle*{3}}
\put(-44,30){\line(1,-1){30}}

\put(-44,30){\vector(1,1){15}}
\put(-44,30){\vector(1,0.6){20}}
\put(-44,30){\vector(1,0.2){17}}

\put(-35.5,21){\circle*{2}}
\put(-35.5,21){\vector(1,-0.2){9}}
\put(-35.5,21){\vector(1,0.5){9}}
\put(-35.5,21){\vector(1,1){6}}

\put(-24.5,10){\circle*{2}}
\put(-24.5,10){\vector(1,1.3){9}}
\put(-24.5,10){\vector(1,0.4){9}}
\put(-24.5,10){\vector(1,2){8}}


\put(-14,0){\circle*{2}}
\put(-14,0){\vector(1,1.9){7.5}}
\put(-14,0){\vector(1,1.3){9}}

\put(-14,0){\vector(1,.7){14}}
\put(-11,-32){\scriptsize{$I_{1}$}}
\put(-11,-39){\scriptsize{with $|I_{1}|=1$}}

\put(-14,0){\line(0,-1){25}}

\put(-14,-25){\circle*{2}}
\put(-14,-25){\vector(1,1.9){7.5}}
\put(-14,-25){\vector(1,1.3){9}}
\put(-14,-25){\vector(1,1){10}}
\put(-14,-25){\vector(1,.7){10}}

\put(-14,-25){\vector(1,0){13}}

\put(-107,28){{\scriptsize ${\bf 1}$}}
\put(-92,-15){$G(C_T)$}

\put(25,28){$\rightsquigarrow$}

\put(55,28){{\scriptsize ${\bf 1}$}}
\put(70,-15){$G(C_{\sigma})$}

\put(62,30){\line(1,0){24}}
\put(62,30){\circle*{2}}

\put(82,30){\circle*{2}}
\put(82,30){\line(0,-1){15}}
\put(82,15){\circle*{2}}

\put(105,30){\circle*{2}}
\put(105,30){\line(0,-1){15}}
\put(105,15){\circle*{2}}

\put(88,30){$\ldots$}

\put(100,30){\line(1,0){18}}

\put(102,35){\scriptsize{$T=\sigma_0$}}
\put(154,23){$\ddots$}

\put(118,30){\circle*{3}}
\put(118,30){\line(1,-1){30}}

\put(118,30){\vector(1,1){15}}
\put(118,30){\vector(1,0.6){20}}
\put(118,30){\vector(1,0.2){17}}

\put(126.5,21){\circle*{2}}
\put(126.5,21){\vector(1,-0.2){9}}
\put(126.5,21){\vector(1,0.5){9}}
\put(126.5,21){\vector(1,1){6}}

\put(137.5,10){\circle*{2}}
\put(137.5,10){\vector(1,1.3){9}}
\put(137.5,10){\vector(1,0.4){9}}
\put(137.5,10){\vector(1,2){8}}


\put(148,0){\circle*{2}}
\put(148,0){\vector(1,1.9){7.5}}
\put(148,0){\vector(1,1.3){9}}

\put(148,0){\vector(1,.7){14}}

\put(148,0){\line(0,-1){25}}

\put(148,-25){\circle*{2}}
\put(148,-25){\vector(1,1.9){7.5}}
\put(148,-25){\vector(1,1.3){9}}
\put(148,-25){\vector(1,1){10}}
\put(148,-25){\vector(1,.7){10}}

\put(148,-25){\line(1,0){20}}

\put(168,-25){\line(0,-1){15}}

\put(168,-25){\circle*{3}}

\put(162,-31){\scriptsize{$\sigma$}}

\put(168,-40){\circle*{2}}

\put(168,-25){\vector(1,1){10}}

\end{picture}
$$
\caption{Graphs \(G(C_T)\) and \(G(C_\sigma)\) considered in Proposition \ref{prop:case1b}.}
\label{fig47}
\end{figure}

     The claim follows by definition of \(\displaystyle Q\big(w_\sigma,p_\sigma,t_1,\prod_{k=2}^{s} t^{w_{q+i,I_k}}\big)\) and the fact that \((n_{\sigma}+1)\underline{v}^{\sigma}=p_\sigma \vu^\sigma.\)
        \end{proof}
    \end{enumerate}
     \item Assume \(|I_1|>1.\) There are two cases to be distinguished:
    \begin{enumerate}
        \item For all \(j\in I_1\) we  have \(g_j=q,\) i.e. the semigroups \(\Gamma^{j}\) have \(q\)--minimal generators. In this case, \(\sigma=\sigma_0^{I_1}\) is the first separation point of the branches of \(I_1\) and let \(I_1=\bigcup_{k=1}^{s_1} I_{1,k}\) be the induced index partition (see also Subsection \ref{subsubsec:truncationstar}). Since \(\sigma\) is an ordinary point then \(p_{\sigma,I_{1,k}}=1\) for all \(k=1,\dots,s_1\) and \(w_{\sigma,I_{1,k}}=w_{q+1,I_{1,k}}=[f_{I_{1,j}},f_{I_{1,k}}]\) with \(j\neq k\) is independent of \(j,\) i.e. \(w_{\sigma,I_{1,k}}=w_{\sigma,I_{1,k'}}\) if \(k\neq k'.\) Also, for \(i=2,\dots, \) \(w_{\sigma,I_{i}}=[f_{I_{1,k}},f_{I_{i}}]\) is independent of \(k.\) Then,

        \begin{figure}[H]
$$
\unitlength=0.50mm
\begin{picture}(70.00,110.00)(0,-50)
\thinlines

\put(-100,30){\line(1,0){24}}
\put(-100,30){\circle*{2}}

\put(-80,30){\circle*{2}}
\put(-80,30){\line(0,-1){15}}
\put(-80,15){\circle*{2}}

\put(-57,30){\circle*{2}}
\put(-57,30){\line(0,-1){15}}
\put(-57,15){\circle*{2}}

\put(-74,30){$\ldots$}

\put(-62,30){\line(1,0){18}}

\put(-63,35){\scriptsize{$T=\sigma_0$}}
\put(-8,23){$\ddots$}

\put(-44,30){\circle*{3}}
\put(-44,30){\line(1,-1){30}}

\put(-44,30){\vector(1,1){15}}
\put(-44,30){\vector(1,0.6){20}}
\put(-44,30){\vector(1,0.2){17}}

\put(-35.5,21){\circle*{2}}
\put(-35.5,21){\vector(1,-0.2){9}}
\put(-35.5,21){\vector(1,0.5){9}}
\put(-35.5,21){\vector(1,1){6}}

\put(-24.5,10){\circle*{2}}
\put(-24.5,10){\vector(1,1.3){9}}
\put(-24.5,10){\vector(1,0.4){9}}
\put(-24.5,10){\vector(1,2){8}}


\put(-14,0){\circle*{2}}
\put(-14,0){\vector(1,1.9){7.5}}
\put(-14,0){\vector(1,1.3){9}}

\put(-14,0){\vector(1,.7){14}}
\put(-11,-32){\scriptsize{$I_{1}$}}
\put(-11,-39){\scriptsize{with $|I_{1}|>1$}}

\put(-14,0){\line(0,-1){25}}

\put(-14,-25){\circle*{2}}
\put(-14,-25){\vector(1,1.9){7.5}}
\put(-14,-25){\vector(1,1.3){9}}
\put(-14,-25){\vector(1,1){10}}
\put(-14,-25){\vector(1,.7){10}}

\put(-14,-25){\vector(1,0){13}}

\put(-107,28){{\scriptsize ${\bf 1}$}}
\put(-92,-15){$G(C_T)$}

\put(25,28){$\rightsquigarrow$}

\put(55,28){{\scriptsize ${\bf 1}$}}
\put(70,-15){$G(C_{\sigma})$}

\put(62,30){\line(1,0){24}}
\put(62,30){\circle*{2}}

\put(82,30){\circle*{2}}
\put(82,30){\line(0,-1){15}}
\put(82,15){\circle*{2}}

\put(105,30){\circle*{2}}
\put(105,30){\line(0,-1){15}}
\put(105,15){\circle*{2}}

\put(88,30){$\ldots$}

\put(100,30){\line(1,0){18}}

\put(102,35){\scriptsize{$T=\sigma_0$}}
\put(154,23){$\ddots$}

\put(118,30){\circle*{3}}
\put(118,30){\line(1,-1){30}}

\put(118,30){\vector(1,1){15}}
\put(118,30){\vector(1,0.6){20}}
\put(118,30){\vector(1,0.2){17}}

\put(126.5,21){\circle*{2}}
\put(118,17){\scriptsize{$\overline{\sigma}_1$}}
\put(126.5,21){\vector(1,-0.2){9}}
\put(126.5,21){\vector(1,0.5){9}}
\put(126.5,21){\vector(1,1){6}}

\put(137.5,10){\circle*{2}}
\put(137.5,10){\vector(1,1.3){9}}
\put(137.5,10){\vector(1,0.4){9}}
\put(137.5,10){\vector(1,2){8}}


\put(148,0){\circle*{2}}
\put(148,0){\vector(1,1.9){7.5}}
\put(148,0){\vector(1,1.3){9}}

\put(148,0){\vector(1,.7){14}}

\put(148,0){\line(0,-1){25}}

\put(148,-25){\circle*{2}}
\put(148,-25){\vector(1,1.9){7.5}}
\put(148,-25){\vector(1,1.3){9}}
\put(148,-25){\vector(1,1){10}}
\put(148,-25){\vector(1,.7){10}}

\put(148,-25){\line(1,0){20}}


\put(168,-25){\circle*{3}}

\put(162,-31){\scriptsize{$\sigma$}}
\put(142,-31){\scriptsize{$\overline{\sigma}_{\varepsilon}$}}


\put(168,-25){\vector(1,1.4){12}}
\put(168,-25){\vector(1,0.5){14}}
\put(168,-25){\vector(1,-0.2){13}}
\put(168,-25){\vector(1,-0.7){14}}

\end{picture}
$$
\caption{Graphs \(G(C_T)\) and \(G(C_\sigma)\) considered in Proposition \ref{prop:case2a}.}
\label{fig48}
\end{figure}
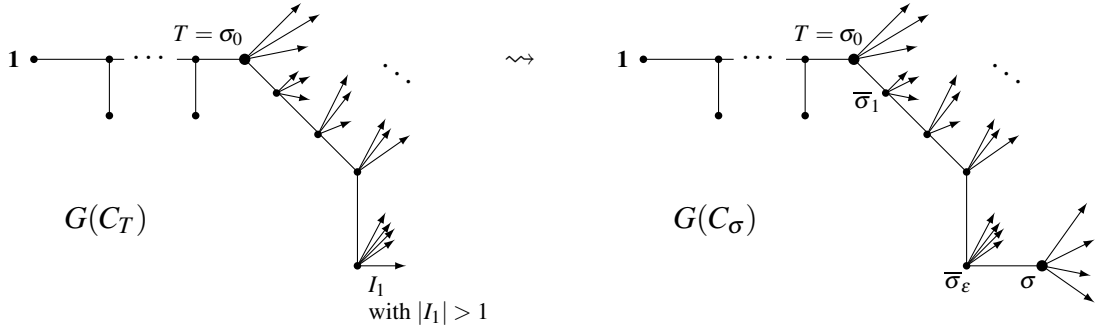

    \begin{proposition} \label{prop:case2a}
    \begin{align*}
    P_\sigma(t_1,\dots,t_{s_1}&,t_{s_1+1},\dots,t_{s_1+s-1})=P_T(t_1\cdots t_{s_1},t_{s_1+1},\dots,t_{s_1+s-1})\\
   \cdot&\bigg(\prod_{j=2}^{s_1}B(w_{q+1,I_{1,j}},1,t_{I_{1,j}},(\prod_{k>j}^{s_1}t_{k}^{w_{q+1,I_{1,k}}})(\prod_{k=s_1+1}^{s_1+s-1}t_{k}^{w_{q+1,I_{k-s_1+1}}}),\prod_{k<j}t_k)\bigg).
    \end{align*}
    \end{proposition}
    \begin{proof}
        As in the proof of Proposition \ref{prop:case1b}, \(G(C_T)\) is a subgraph of \(G(C_\alpha)\) and  \(\mathcal{S}_\sigma\setminus\mathcal{S}_T=\{\sigma\},\) but in this case \(\sigma\) is a proper star point so it  only contributes to the factor \(P_3.\) It is then easy to see that  
        
        \begin{align*}
       P_T(t_1\cdots& t_{s_1},t_{s_1+1},\dots,t_{s_1+s-1})=\\
             =& \frac{1}{\underline{t}^{\underline{v}^1}-1} \cdot \big(\prod_{i=1}^q \frac{\underline{t}^{\underline{v}^{\sigma_{i}}}-1}{\underline{t}^{\underline{v}^{\rho_{i}}}-1}\big)\cdot (\underline{t}^{\underline{v}^{\sigma_0}}-1) \cdot \prod_{\rho \in \widetilde{\mathcal{E}} } \frac{\underline{t}^{(n_{\rho}+1)\underline{v}^{\rho}}-1}{\underline{t}^{\underline{v}^{\rho}}-1} \cdot \prod_{\tiny\begin{array}{c}s(\alpha) >1\\ \alpha\neq\sigma\end{array}}  (\underline{t}^{\underline{v}^{\alpha}}-1)^{s(\alpha)-1}.
        \end{align*}
        
        To finish we need to check that    
        \[
\bigg(\prod_{j=2}^{s_1}B(w_{q+1,I_{1,j}},1,t_{I_{1,j}},(\prod_{k>j}^{s_1}t_{k}^{w_{q+1,I_{1,k}}})(\prod_{k=s_1+1}^{s_1+s-1}t_{k}^{w_{q+1,I_{k-s_1+1}}}),\prod_{k<j}t_k)\bigg)=(t^{\vu^\sigma}-1)^{s(\sigma)-1}.
        \]
        
        To do that, first observe that since \(w_{\sigma,I_{1,k}}=w_{\sigma,I_{1,k'}}\) if \(k\neq k'\) we have 
        \[
        (t_{I_{1,j}}\cdot \prod_{k<j}t_k)^{w_{q+1,I_{1,j}}}=\prod_{k\leq j}t_k^{w_{q+1,I_{1,k}}}.
        \]
        In this way, 
        \begin{align*}
         &B\Big(w_{q+1,I_{1,j}},1,t_{I_{1,j}},\Big(\prod_{k>j}^{s_1}t_{k}^{w_{q+1,I_{1,k}}}\Big)\cdot\Big(\prod_{k=s_1+1}^{s_1+s-1}t_{k}^{w_{q+1,I_{k-s_1+1}}}\Big),\prod_{k<j}t_k\Big) \\
            & \hspace{20pt} =\Big(\big(\prod_{k=s_1+1}^{s_1+s-1}t_{k}^{w_{q+1,I_{k-s_1+1}}}\big)\cdot t_j^{w_{q+1,I_{1,j}}}\Big)\cdot\big(\prod_{k<j}t_k^{w_{q+1,I_{1,k}}}\big)-1\Big)
        \end{align*}
        
        is independent of \(j\) and hence it is repeated \(s_1-1\) times.
        \medskip
        
        On the other hand, since \(q=g_i\) for all \(i=1,\dots, s_1\) then by eq.~ \eqref{eqn:valueproperordinary} we have
        \begin{equation*}
    \pr_j(\underline{v}^\sigma)=\left\{\begin{array}{ll}
         [f_{I_{1,j}},f_{I_{1,k}}]& \text{if}\;j\neq k\;\text{and}\; j\leq s_1  \\
         
         [f_{I_{1,k}},f_{I_{j-s_1+1}}]& \text{if}\; s_1+1\leq j\leq s
    \end{array}\right.
\end{equation*}
Finally, by definition of \(C_\sigma,\) the valency of \(\sigma\) in \(G(C_\sigma)\) is \(s_1+1.\) Therefore, \(s(\sigma)=s_1\) as there is no dead arc starting with \(\sigma.\) Thus, the claim follows.
    \end{proof}
        
        We then do \(T=\sigma_0^{I_{1}}\) and \(\sigma=\sigma_0^{I_{1,1}}\) to continue the process.
        
        \item We assume that \(g_j>q\) for some \(j\in I_1.\) We distinguish again two subcases:
        \begin{enumerate}[wide, labelwidth=!, labelindent=0pt]
            \item Assume \((f_1|\cdots|f_{|I_1|})\leq (q+1,0)\) and for simplicity assume  \(g_1>q\). Let us denote by \(\sigma_{0}^{I_1}\) be the first proper star point of the package \(I_1.\) 
            Let \(I_1=\bigcup_{k=1}^{s_1} I_{1,k}\) be the partition associated to the proper star \(\sigma_{0}^{I_1}.\) In this case, the analysis of the first \(s_1\) branches of \(C_{\sigma_{0}^{I_1}}\) is more delicate as there are some of them that have at least \(q+1\) maximal contact values. For this reason, we subdivide the partition of \(I_1\) as in the proof of Proposition \ref{prop:poincarelemm2}. Let \(\jnd_{sing}=\bigcup_{k=l+1}^{s_1}I_{1,k}\) be the branches of \(C_\sigma\) with \(q+1\) maximal contact values. Observe that the good ordering of the branches implies that the first \(l\) branches have \(q\)-maximal contact values \(\jnd_{sm}=\bigcup_{k=1}^{l}I_{1,k}\)  \ (cf. Lemma \ref{lem:orderbetaiscompatible}) and \(l\) could be \(0\). Thus, \(C_\sigma\) is plane curve that by definition have the first \(l\) branches have \(q\) maximal contact values, the next \(s_1-l\) branches have \(q+1\) maximal contact values and the last \(s-1\) branches have \(q\) maximal contact values. Then, 
            
            \begin{proposition}\label{prop:case2bi}
              \begin{align*}
& P_\sigma(t_1,\dots,t_{s_1},t_{s_1+1},\dots,t_{s_1+s-1})= P_T(t^{p_{q+1,I_{1,1}}}_1\cdots t^{p_{q+1,I_{1,s_1}}}_{s_1},t_{s_1+1},\dots,t_{s_1+s-1})\\
& \hspace{20pt} \cdot  Q\Big(w_{q+1,I_{1,1}},p_{q+1,I_{1,1}},t_1,(\prod_{k=2}^{s_1}t_{k}^{w_{q+1,I_{1,k}}})(\prod_{k=s_1+1}^{s_1+s-1}t_{k}^{w_{q+1,I_{k-s_1+1}}})\Big)\\
&\hspace{20pt} \cdot  \bigg(\prod_{j=2}^{s_1}B(w_{q+1,I_{1,j}},p_{q+1,I_{1,j}},t_{I_{1,j}},(\prod_{k>j}^{s_1}t_{k}^{w_{q+1,I_{1,k}}})(\prod_{k=s_1+1}^{s_1+s-1}t_{k}^{w_{q+1,I_{k-s_1+1}}}),\prod_{k<j}t^{p_{q+1,I_{1,k}}}_k)\bigg).
 \end{align*}
            \end{proposition}
            \begin{proof}
            First of all, observe that if \(\jnd_{sm}=I_1\) then the proof is analogous to the proof of Proposition \ref{prop:case2a}. Therefore, we will assume \(\jnd_{sm}\neq I_1.\) As usual, the dual graph \(G(C_T)\) is a subgraph of \(G(C_\sigma)\) and we need to analyze the new star points.
            \medskip
            
            As in the proof of Proposition \ref{prop:poincarelemm2} we star with the case \(\jnd_{sing}=I_1.\) In this case, there are \(\epsilon+1\) new star points, i.e. \[\mathcal{S}_\sigma\setminus\mathcal{S}_T=\{\sigma_0^{I_1}\preceq \osigma^{I_1}_1\preceq\cdots\preceq\osigma^{I_1}_\epsilon\}:=\mathcal{NS}.\] 
            Since \(\jnd_{sing}=I_1\) there exists a unique end point \(W:=P(L_{q+1}^{1})\) which is in fact common to the dead arcs \(L_{q+1}^{I_{1,k}}\) associated to each star point of \(\mathcal{NS}\) as star point of the corresponding \(G(C_\sigma^{k}).\) Therefore, by the ordering in the set of branches (cf.~proof of Proposition \ref{prop:poincarelemm2}) \(\widetilde{\mathcal{E}}_\sigma\setminus\widetilde{\mathcal{E}}_T=\{\osigma_\epsilon^{I_1}\};\) observe that \((f_1|\cdots|f_{|I_1|})=(q+1,0)\) if and only if \(\osigma_\epsilon^{I_1}=\sigma_0^{I_1}.\) Also, the star points of \(\mathcal{NS}\setminus\{\osigma_\epsilon^{I_1}\}\) are all proper star points. The star point \(\osigma_\epsilon^{I_1}\) may be also a proper star point but it is distinguished since it also belong to \(\widetilde{\mathcal{E}}_\sigma.\) Then, by Noether's formula \ref{noehterformula}, Theorem \ref{thm:p1p2p3} and the definition of the curves \(C_\sigma,C_T\) with a bit of effort one may check that 
            \[
            \begin{split}
            P_T(t^{p_{q+1,I_{1,1}}}_1\cdots t^{p_{q+1,I_{1,s_1}}}_{s_1},t_{s_1+1},\dots,t_{s_1+s-1})= &\frac{1}{\underline{t}^{\underline{v}^1}-1}\cdot  \prod_{i=1}^q \frac{\underline{t}^{\underline{v}^{\sigma_{i}}}-1}{\underline{t}^{\underline{v}^{\rho_{i}}}-1}\cdot (\underline{t}^{\underline{v}^{\sigma^0}}-1)\\
            \cdot& \prod_{\rho \in \widetilde{\mathcal{E}}\setminus\{\osigma_\epsilon\} } \frac{\underline{t}^{(n_{\rho}+1)\underline{v}^{\rho}}-1}{\underline{t}^{\underline{v}^{\rho}}-1} \cdot \prod_{\tiny\begin{array}{c}s(\alpha) >1\\ \alpha\notin\mathcal{NS}\end{array}}  (\underline{t}^{\underline{v}^{\alpha}}-1)^{s(\alpha)-1}.
            \end{split}
            \]
            
            By the good ordering in Subsection \ref{subsec:totalorderstar} we have that \((n_{v^{W}}+1)\vu^{W}=p_{q+1,I_{1,1}}\vu^{W}\), where \(p_{q+1,I_{1,1}}=e^{I_{1,1}}_{q}/e^{I_{1,1}}_{q+1}\) and by eq.~ \eqref{eqn:maximalcontactvalues}
             \[\pr_j(\vu^W)=\left\{\begin{array}{cc}
    \obeta^{I_{1,1}}_{q+1}/e^{I_{1,1}}_{q+1}
    &\text{if}\;j=1 \\[.3cm]
     \frac{[f_1,f_j]}{e^{I_{1,1}}_ {q}}& \text{if}\;j\neq 1
     \end{array}\right.
     \]
     In this way, it is straightforward to check

     \[
     Q\Big(w_{q+1,I_{1,1}},p_{q+1,I_{1,1}},t_1,(\prod_{k=2}^{s_1}t_{k}^{w_{q+1,I_{1,k}}})(\prod_{k=s_1+1}^{s_1+s-1}t_{k}^{w_{q+1,I_{k-s_1+1}}})\Big)=\frac{\underline{t}^{(n_{W}+1)\underline{v}^{W}}-1}{\underline{t}^{\underline{v}^{W}}-1}. 
     \]

     As in the proof of Proposition \ref{prop:poincarelemm2}, let us now assume that \(\jnd_{sm}\neq \emptyset.\) In this case, \(W\) is no longer an end point for \(G(C_{\sigma})\) as the good order of the branches implies that the geodesics of \(I_{1,1}\) pass through \(W.\) In this case \(p_{q+1,I_{1,1}}=1\) and then 
     \[
     Q\Big(w_{q+1,I_{1,1}},p_{q+1,I_{1,1}},t_1,(\prod_{k=2}^{s_1}t_{k}^{w_{q+1,I_{1,k}}})(\prod_{k=s_1+1}^{s_1+s-1}t_{k}^{w_{q+1,I_{k-s_1+1}}})\Big)=1.
     \] 
     Thus, both in the case \(\jnd_{sing}=I_1\) and in the case \(\jnd_{sing}\neq I_1,\) we have proven

\begin{align*}
    & P_1(\underline{t})\cdot P_2(\underline{t})= P_T\big(t^{p_{q+1,I_{1,1}}}_1\cdots t^{p_{q+1,I_{1,s_1}}}_{s_1},t_{s_1+1},\dots,t_{s_1+s-1}\big)\\
& \hspace{20pt} \cdot Q\Big (w_{q+1,I_{1,1}},p_{q+1,I_{1,1}},t_1,(\prod_{k=2}^{s_1}t_{k}^{w_{q+1,I_{1,k}}})(\prod_{k=s_1+1}^{s_1+s-1}t_{k}^{w_{q+1,I_{k-s_1+1}}})\Big).
\end{align*}

     

We are left with the task of checking that

\begin{align*}
&\prod_{j=2}^{s_1}B\Big(w_{q+1,I_{1,j}},p_{q+1,I_{1,j}},t_{I_{1,j}},(\prod_{k>j}^{s_1}t_{k}^{w_{q+1,I_{1,k}}})(\prod_{k=s_1+1}^{s_1+s-1}t_{k}^{w_{q+1,I_{k-s_1+1}}}),\prod_{k<j}t^{p_{q+1,I_{1,k}}}_k\Big)\\
&\hspace{20pt} =\prod_{\alpha\in\mathcal{NS}}  (\underline{t}^{\underline{v}^{\alpha}}-1)^{s(\alpha)-1}.
\end{align*}

     We proceed as in the proof of Proposition \ref{prop:poincarelemm2}. Indeed, eq.~ \eqref{eqn:valueproperspecial} leads to the fact that each factor is of the form 
     \[
     B\Big(w_{q+1,I_{1,j}},p_{q+1,I_{1,j}},t_{I_{1,j}},\big(\prod_{k>j}^{s_1}t_{k}^{w_{q+1,I_{1,k}}}\big)\big(\prod_{k=s_1+1}^{s_1+s-1}t_{k}^{w_{q+1,I_{k-s_1+1}}}\big),\prod_{k<j}t^{p_{q+1,I_{1,k}}}_k\Big)=(t^{\vu^{\alpha}}-1)
     \]
     for some \(\alpha \in \mathcal{NS}\) which depends on \(j.\) By definition, to each \(\alpha\in\mathcal{NS}\) there are \(s(\alpha)-1\) packages (i.e.~branches of \(C_\sigma\)) associated to it; this means that each of the factors is repeated \(s(\alpha)-1\) times, and the proof is complete.
\end{proof}

\begin{figure}[H]
$$
\unitlength=0.50mm
\begin{picture}(70.00,110.00)(-20,-70)
\thinlines

\put(-120,30){\line(1,0){20}}
\put(-120,30){\circle*{2}}

\put(-104,30){\circle*{2}}
\put(-104,30){\line(0,-1){15}}
\put(-104,15){\circle*{2}}

\put(-77,30){\circle*{2}}
\put(-77,30){\line(0,-1){15}}
\put(-77,15){\circle*{2}}

\put(-94,30){$\ldots$}

\put(-82,30){\line(1,0){18}}


\put(-64,30){\circle*{3}}
\put(-64,30){\line(1,-1){30}}

\put(-64,30){\vector(1,1){15}}
\put(-64,30){\vector(1,0.6){20}}
\put(-64,30){\vector(1,0.2){17}}

\put(-34,0){\circle*{2}}


\put(-34,0){\line(0,-1){25}}

\put(-34,-25){\circle*{2}}
\put(-34,-25){\vector(1,1.9){7.5}}
\put(-34,-25){\vector(1,1.3){9}}
\put(-34,-25){\vector(1,1){10}}
\put(-34,-25){\vector(1,.7){10}}

\put(-34,-25){\vector(1,0){13}}

\put(-127,28){{\scriptsize ${\bf 1}$}}
\put(-112,-15){$G(C_T)$}

\put(-15,28){$\rightsquigarrow$}

\put(15,28){{\scriptsize ${\bf 1}$}}
\put(38,-35){$G(C_{\sigma})$}

\put(22,30){\line(1,0){20}}
\put(22,30){\circle*{2}}

\put(38,30){\circle*{2}}
\put(38,30){\line(0,-1){15}}
\put(38,15){\circle*{2}}

\put(69,30){\circle*{2}}
\put(69,30){\line(0,-1){15}}
\put(69,15){\circle*{2}}

\put(48,30){$\ldots$}

\put(60,30){\line(1,0){18}}


\put(78,30){\circle*{3}}
\put(78,30){\line(1,-1){30}}

\put(78,30){\vector(1,1){15}}
\put(78,30){\vector(1,0.6){20}}
\put(78,30){\vector(1,0.2){17}}

\put(108,0){\circle*{2}}


\put(108,0){\line(0,-1){25}}

\put(108,-25){\circle*{2}}
\put(108,-25){\vector(1,1.9){7.5}}
\put(108,-25){\vector(1,1.3){9}}
\put(108,-25){\vector(1,1){10}}
\put(108,-25){\vector(1,.7){10}}

\put(108,-25){\line(1,0){20}}

\put(128,-25){\circle*{2}}
\put(128,-25){\vector(1,-0.2){10}}
\put(128,-25){\vector(1,1){13}}

\put(139,-36){\circle*{2}}
\put(139,-36){\vector(1,-0.2){10}}
\put(139,-36){\vector(1,1){13}}

\put(128,-25){\line(1,-1){20}}

\put(149,-46){\circle*{2}}
\put(149,-46){\vector(1,-0.5){10}}
\put(149,-46){\vector(1,-0.1){12}}
\put(149,-46){\vector(1,0.7){13}}

\put(149,-46){\line(-1,-1){15}}

\put(134,-61){\circle*{2}}
\put(134,-61){\vector(2,0.8){10}}
\put(134,-61){\vector(1,-0.4){12}}
\put(134,-61){\vector(1.2,-0.7){11}}
\put(134,-61){\vector(1,0){10}}

\end{picture}
$$
\caption{Graphs \(G(C_T)\) and \(G(C_\sigma)\) considered in Proposition \ref{prop:case2bi}.}
\label{fig49}
\end{figure}

            \item Assume \((f_1|\cdots|f_{|I_1|})> (q+1,0),\) i.e \((f_1|\cdots|f_{|I_1|})\geq (q+1,c)\) with \(c\neq 0,\) and for simplicity \(g_1> q.\) Let us denote by \((q_{I_1},c_{I_1}):= (f_1|\cdots|f_{|I_1|}).\) Since \((q_{I_1},c_{I_1})> (q+1,0),\) there are \(q_{I_1}-q\) star points which are non-proper between \(\osigma_\epsilon\) and \(\sigma_0^{I_1},\) i.e. 
            \[\osigma_{\epsilon}\prec \alpha_{q+1}\prec \cdots \prec \alpha_{q_{I_1}}\preceq \sigma_0^{I_1}.\] 

            At this stage, the case is analogous to the case \((1)(b)\) as we are dealing with a sequence of star points on one branch of \(C_\sigma\) that are non-proper. Therefore, a slight change in the proof of Proposition \ref{prop:case1b} actually shows the following.

            \begin{proposition}\label{prop:case2bii}
 \[
 P_{\sigma}(\underline{t})=P_T(t_1^{p_\sigma},t_2,\dots,t_s)\cdot Q\big(w_\sigma,p_\sigma,t_1,\prod_{k=2}^{s} t_k^{w_{q+i,I_k}}\big).
 \]                
            \end{proposition}
        \end{enumerate}
    \end{enumerate}
\end{enumerate}

We run this process until \(\sigma=\max\{\alpha\in\mathcal{S}\}.\) 
\begin{rem}
    Observe that in the algorithm we have developed, to simplify the exposition, the case where the operations occur in the first package. In the general case, the indexing must be adjusted as follows:
    \begin{enumerate}[wide, labelwidth=!, labelindent=0pt]
        \item [$\diamond$] In the case of Proposition \ref{prop:case1b} if the operation is performed in the \(j\)--th branch one must use the expression \[
            P_{\sigma}(\underline{t})=P_T(t_1,\dots,t_j^{p_\sigma},t_{j+1},\dots,t_s)\cdot Q\big(w_\sigma,p_\sigma,t_j,\prod_{k\neq j}t_k^{w_{q+i,I_k}}\big).
            \]
    \item[$\diamond$] In the case of Proposition \ref{prop:case2bi}, there is an index \(j\) in which the partition is produced \(I_j=\bigcup_{k=1}^{s_j}I_{j,k}\) for the corresponding proper star point. Then the indexing in Proposition \ref{prop:case2bi} is

              \begin{align*}
& P_\sigma(t_1,\dots,t_{j-1},t_{j},\dots,t_{j+s_j},\dots,t_{s_j+s-1})= P_T(t_1,\dots,t_{j-1},t^{p_{q+1,I_{j,1}}}_j\cdots t^{p_{q+1,I_{j,s_j}}}_{s_j+j},t_{j+s_j+1},\dots,t_{s_j+s-1})\\
& \hspace{20pt} \cdot  Q\Big(w_{q+1,I_{j,1}},p_{q+1,I_{j,1}},t_j,(\prod_{k=2}^{s_j}t_{k+j}^{w_{q+1,I_{j,k}}})(\prod_{k\notin\{j,\dots,s_j\}}t_{k}^{w_{q+1,I_{k}}})\Big)\\
&\hspace{20pt} \cdot  \Big(\prod_{l=2}^{s_j}B(w_{q+1,I_{j,l}},p_{q+1,I_{j,l}},t_{j+l},(\prod_{k=l+1}^{s_j}t_{j+k}^{w_{q+1,I_{j,k}}})(\prod_{k\notin\{j,\dots,s_j+j\}}t_{k}^{w_{q+1,I_{k-s_j+1}}}),\prod_{k=j}^{l-1}t^{p_{q+1,I_{j,k}}}_k)\Big).
 \end{align*}

    \end{enumerate}
    The indexing of Proposition \ref{prop:case2a} and Proposition \ref{prop:case2bii} is adjusted analogously to these cases.
\end{rem}

Altogether, we have proven the following result.
\begin{theorem}
    Let \(C\) be a plane curve singularity, and consider the set \(\mathcal{S}\) of star points of the dual graph \(G(C).\) Assume that the branches are ordered by the good order in their topological Puiseux series, and the points in \(\mathcal{S}\) are ordered by the total order \(<\) defined in Subsection \ref{subsec:totalorderstar}. Then, the Poincaré series \(P_C(\underline{t})\) can be computed recursively, via the previous process, from the Poincaré series \(P_\alpha(\underline{t}')\) of the approximations associated to the star points defined in Subsection \ref{subsubsec:truncationstar}.
\end{theorem}

\begin{ex}
Let us continue with the Example \ref{ex:examplerefinedgoodorder}. We assume that the branches of \(C\) are good ordered, i.e. \(Y_1=2x^2+x^5,\) \(Y_2=2x^2+x^{14/3}+1/2x^5,\) \(Y_3=5x^2+x^3,\) \(Y_4=5x^2+x^{5/2}+3x^3\) and \(Y_5=x^2.\) Following our procedure the first star point in the dual graph is \(\sigma_0\) (see Figure \ref{fig:ex_2}) and the Poincaré series of \(C_{\sigma_0}\) is, by Proposition \ref{prop:poincarelemm1},
\[
P_{\sigma_0}(t_1,t_2,t_3)=Q(2,1,t_1,t_2^2,t_3^2)\cdot B(1,2,t_2,t_1,t_3^2)\cdot Q(1,2,t_3,t_1t_2)=\frac{(t_1^2t_2^2t_3^2-1)^2}{t_1t_2t_3-1}.
\]
The next star point in the dual graph is \(\sigma_0^{I_1},\) we set \(\sigma=\sigma_0^{I_1}\) and \(T:=\sigma_0.\) We are in the case \((2)(b)(i)\) as the package \(I_1\) has two branches and one of them has one Puiseux pair. Then, application of Proposition \ref{prop:case2bi} yields
\begin{align*}
P_\sigma(t_1,t_2,t_3,t_4)&=P_T(t_1t_2^3,t_3,t_4)\cdot Q(14,1,t_1,t_2^{14}t_3^3t_4^2)\cdot B(14,3,t_2,t_3^2t_4^2,t_1)\\
&=\frac{(t_1^2t_2^6t_3^2t_4^2-1)^2(t_1^{14}t_2^{42}t_3^6t_4^6-1)}{t_1t_2^3t_3t_4-1}.
\end{align*}

Finally, the last star point in the dual graph is \(\sigma_0^{I_2}\) and 

\begin{align*}
P_C(\underline{t})&=P_\sigma(t_1,t_2,t_3,t_4,t_5)=P_T(t_1,t_2,t_3t_4^2,t_5)\cdot Q(5,1,t_3,t_4^{5}t_1^2t_2^6t_5^2)\cdot B(5,2,t_4,t_1^2t_2^6t_5^4,t_3)\\
&=\frac{(t_1^2t_2^6t_3^2t_4^4t_5^2-1)^2(t_1^{14}t_2^{42}t_3^6t_4^{12}t_5^6-1)(t_1^4t_2^{12}t_3^5t_4^{10}t_5^4-1)}{t_1t_2^3t_3t_4^2t_5-1},
\end{align*}
which completes the iterative computation of the Poincaré series.
\end{ex}





\section{The Alexander polynomial of the link} \label{sec:topology}

The Alexander polynomial (in $r$ variables) is an invariant of a link with $r$ (numbered) components in the sphere $S^3$. The definition of the multivariable Alexander polynomial bases on the notion of universal abelian covering \(\rho:\widetilde{X}\rightarrow X.\) The group of covering transformations \(H_1(X;\mathbb{Z})=\mathbb{Z}^r\) is a free abelian multiplicative group on the symbols \(\{t_1,\dots,t_r\}\) where each \(t_i\) is geometrically associated with an oriented meridian of an irreducible component of the link. In this way, if \(\widetilde{p}\) is a typical fiber of \(\rho\) then the group \(H_1(\widetilde{X},\widetilde{p};\mathbb{Z})\) becomes a module over \(\mathbb{Z}[t_1,t_1^{-1},\dots,t_r,t_r^{-1}].\) The multivariable Alexander polynomial \(\Delta_L(t_1,\dots,t_r)\) is then defined as the greatest common divisor of the first Fitting ideal \(F_1(H_1(\widetilde{X};\mathbb{Z})).\) Observe that  \(\Delta_L(t_1,\dots,t_r)\) is then well-defined up to multiplication by a unit of \(\mathbb{Z}[t_1^{\pm 1},\dots t_r^{\pm 1}].\)
\medskip

To a plane curve singularity $C = \bigcup_{i=1}^r C_i \subset (\mathbb{C}^2, 0)$ we assign the link $L=C \cap S_{\varepsilon}^3$. It is well known that the link complement \(X=S^{3}-L\) fibers over \(S^1\) and it has an iterated torus structure that can be described via Lé's Carousels \cite{LeCarousels}. For general topological properties of links and its numerical invariants we refer to \cite{EN,SW,Shinohara1,Shinohara2}, for an accurate description of the link complement of a plane curve we refer to \cite{LeCarousels}.
\medskip

In this closing section, we will show the topological counterpart of the algebraic operations in the dual graph described in the previous sections. This dictionary between algebraic and topological operations will allow us to provide a new proof of the coincidence between the Poincaré series of a plane curve and the Alexander polynomial of its associated link.

\subsection{Description of the link: satellization}\label{subsec:satellization}
First, let us briefly recall the iterative construction of the link associated to a plane curve singularity. For that we will follow the beautiful exposition of Weber in \cite{Webertop}. 

\medskip
Consider the map $f:S^1\to S^1 \times S^1$ given by $f(z)=(z^p, z^q)$; if we identify $S^1$ with the complex unit circle, the image of $f$ is a closed curve, say $K$, which turns out to be a knot in $\mathbb{R}^3$ obtained by making $q$ winds longitudinally and $p$ winds transversally: this is a torus knot of type $(p,q)$. In general we can construct a link from a given knot by a process called satellization whose input date are:
\begin{enumerate}
    \item An oriented knot \(K\) in \(S^3\) together with a tubular neighbourhood \(N\) around.
    \item An oriented link \(L\) in the interior of the tubular neighborhood \(V\) of the unknot $U$ such that $L$ is not contained in any ball inside $V$.
    \item The choice of an orientation preserving a diffeomorphism \(\varphi:V\rightarrow N\) such that \(\varphi(U)=K\) and carrying parallels to parallels; the diffeomorphism \(\varphi\) is determined by the parallel \(p\) on \(\partial N\) (up to isotopy) such that \(p=\varphi(p').\)
\end{enumerate}

The process of replacing $K$ by $\varphi (L)$ is called the \emph{satellization} of \(L\) around \(K\). One can provide an iterated torus structure by performing the so called \((p,q)\)--satellization: let \(K\) be an oriented knot in \(S^3.\) An oriented knot \(K'\) is a \((p,q)\)--satellite around \(K\) if it has the same smooth type of a torus knot \((p,q)\) on the boundary of a tubular neighborhood of \(K\) on which meridians are chosen to be non-singular closed oriented curves which have a linking number \(+1\) with \(K\), and parallels are non-singular closed oriented curves which do not link \(K\) and have an intersection number \(+1\) with a meridian.
\medskip

With this construction, the idea is to reproduce the iterated torus structure to the case of links. Consider a concentric tubular neighborhood $V'$ inside the interior \(V^{\circ}\) of $V$. A torus link is the link of $s\geq 1$ torus knots of type $(\alpha, \beta)$ with $\alpha \geq 1$ placed on the boundary $\partial V'$, and possibly together with $U$. We will write $\mathrm{TL}(s,\delta; (\alpha, \beta))$ for the torus link of $s$ torus knots of type $(\alpha, \beta)$; here $\delta$ is equal to 1 if $U$ is chosen as a component of the link, and 0 otherwise.
\medskip

It is certainly possible to build an iterated torus link from a set $L_1,\ldots L_g$ of torus links as follows. First, we select a component $K_1$ of $L_1$ and satellizate $L_2$ around $K_1$. This yields a link, and we select again a component $K_2$ of this link, and satellizate $L_3$ around $K_2$. We repeat this process until we satellizate $L_g$ around the chosen component of the link just obtained. Observe that it is possible to select a same component several times: in this case, each new satellization takes place inside a smaller tubular neighborhood. 
\medskip

In the case of the link $L=C \cap S_{\varepsilon}^3$ associated to a plane curve, it is well known that this link is oriented and that it can be described by successive satellizations \cite{Brauner28} (see also \cite[Sect. 6]{Webertop}). The iterated torus structure of the knot \(K\) associated with an irreducible plane curve singularity with \(g\) Puiseux pairs is described by a theorem of K. Brauner \cite{Brauner28} (see also \cite[Theorem 2.3.2]{LeCarousels}): the knot \(K_g\) is computed recursively from successive $(p_i,w_i)$-satellizations around $K_{i-1}$, where \(K_0\) is the unknot and the \(p_i,w_i\) are defined from the Puiseux pairs and  the expressions in eq.~ \eqref{eqn:defw_j}. Similarly, Brauner \cite{Brauner28} also described the successive \((p_i,w_i)\)--satellizations describing the iterated torus link structure in the non irreducible case (see also \cite{Webertop,LeCarousels}).

\subsubsection{Iterative homology}\label{subsubsec:iterativehomology}
A key part of our new proof of the coincidence of the Alexander polynomial and the Poincaré series are the results of Sumners and Woods \cite{SW}, which allow the computation of the Alexander polynomial iteratively. Those results are based on an iterative homology computation. Let us recall briefly their construction.
\medskip

Given an algebraic link \(L:=L_1\cup\cdots\cup L_r\) with \(r\geq 1\) components,  they consider a tubular neighborhood \(\bigcup_{i=1}^{r}\{V_i\}\) of \(L\) such that each \(V_i\) is a tubular neighborhood of each component \(L_i\) and the \(V_i\)'s are pairwise disjoint. Let \(V'\) be a tubular neighborhood of the unknot \(K_0\) and assume \(K'\) is a knot contained in \(V'\) such that \(K'\) is homologous to \(p\) times \(K_0.\) Let \(\varphi:V' \rightarrow V_r\) be an orientation preserving onto diffeomorphism taking longitude to longitude. Define \(K:=\varphi(K)'\simeq p L_r\). Sumners and Woods show that only the following two types of operations are needed in order to understand the structure of the link:

\begin{enumerate}
    \item  Satellization along one component giving rise to a new link \(L'=L_1\cup\cdots\cup L_{r-1}\cup K\) with the same number of components as the old link \(L\). If \(U'\) is a tubular neighborhood of \(K'\) which is contained in the interior of \(V'\) then the link exterior of \(L'\) is 
    \[X=S^3\setminus \operatorname{Int}(V_1\cup V_{r-1}\cup \varphi(U')).\]
    \item  Adding one new branch giving rise to a new link \(\widehat{L}=L_1\cup\cdots\cup L_{r}\cup L_{r+1}\) with one more component than the old link \(L\). If we denote by \(V^\ast\) a small tubular neighborhood of \(K_0\) which is contained in \(V'\) and misses \(U',\) in this case the link exterior of \(\widehat{L}\) has the form
    \[X=S^3\setminus \operatorname{Int}(V_1\cup V_{r-1}\cup \varphi(U')\cup\varphi(V^\ast)).\]
\end{enumerate}

In order to provide an iterative computation of the Alexander polynomial, Sumners and Woods \cite[Sect. V]{SW} provided a method to compute the homology of the abelian cover of the exterior of a new link created by one of the previous two operations in terms of the old link. To do so, they observed that the previous operations naturally provide a Mayer-Vietoris decomposition of the link exterior which will lift to the abelian cover. Let us denote by \(Y:=S^3\setminus \operatorname{Int}(V_1\cup \cdots\cup V_{r})\) the link exterior of the old link \(L\) and define 
\[W:=\left\{\begin{array}{lc}
    V_r\setminus \operatorname{Int}(\varphi(U')) &\text{in the case \((1)\),}  \\
     V_r\setminus \operatorname{Int}(\varphi(U')\cup\varphi(V^\ast)) &\text{in the case \((2)\).}
\end{array}
\right.\]
If \(T=\partial V_r\) then we obtain the Mayer-Vietoris splitting for \(X=Y\cup_T W.\) 
\medskip

On account of Sumners and Woods \cite{SW} this Mayer-Vietoris splitting lifts to a splitting on the abelian cover spaces of \(X\). We denote by \(\rho:\widetilde{X}\rightarrow X\) the universal abelian cover and \(\Lambda:=\mathbb{Z}[t_1^{\pm 1},\dots,t_r^{\pm 1}].\) Since we will be interested in the computation of the Alexander polynomial, we must describe the \(\Lambda\)--module structure of \(H_1(\widetilde{X},\widetilde{p}).\) To do so, it is convenient to use the Mayer-Vietoris splitting; write  \(\widetilde{Y}=\rho^{-1}(Y),\) \(\widetilde{W}=\rho^{-1}(W)\) and \(\widetilde{T}=\rho^{-1}(T).\) Then, Sumners and Woods \cite[Sect. V]{SW} showed that \(H_1(\widetilde{X},\widetilde{p})\) decomposes as \(\Lambda\)-module in the following forms:
\begin{enumerate}
    \item [(a)] Assume \(r=1\) and we have performed an operation of type \((1)\) or \((2)\), then
    \[H_1(\widetilde{X})\simeq_\Lambda H_1(\widetilde{Y})\oplus H_1(\widetilde{W})/H_1(\widetilde{T}).\]
    \item [(b)] Assume \(r\geq 2\) and we have performed an operation of type \((1)\) or \((2)\), then
    \[H_1(\widetilde{X})\simeq_\Lambda H_1(\widetilde{Y})\oplus H_1(\widetilde{W}).\]
\end{enumerate}

\subsection{Gluing as a topological operation: the irreducible case}\label{subsec:alexirreducible}
Zariski \cite{Zar32} (independently of Burau \cite{Burau33}) shows that knots coming from irreducible plane curve singularities are topologically distinguishable when the involved singularities have different Puiseux pairs. Remarkably, he succeeds in describing the fundamental group of the complement of a knot using algebraic identities. This is our starting point to go one step further in the understanding of the relation between the fundamental group and the value semigroup $\Gamma$ associated to the singularity.
\medskip

Let \(G:=\pi_1(X)\) be the fundamental group of the knot complement $X$. Then, Zariski \cite[\S 4]{Zar32} proved
\[G=\langle b_1,u_1,\dots,u_g\;|\; u_i^{n_i}=b_i^{q_i}u_{i-1}^{n_{i-1}n_i}\; \ \text{for}\ \;i=1,\dots,g\rangle,\]
where \(u_0=1,\) the elements \(b_2,\dots,b_g\) are determined by the relations
\[b_{i+1}b_i^{y_i}u_{i-1}^{n_{i-1}x_i}=u_i^{x_i}\quad\text{for}\;i=1,\dots,g-1,\]
in which the positive integers \(x_i,y_i\) are defined such that \(x_iq_i=y_in_i+1\), and the numbers \(q_i,n_i\) are defined from the Puiseux characteristics as \(q_i=(\beta_i-\beta_{i-1})/e_i,\) \(n_i=e_{i-1}/e_i\) and \(e_i=\gcd(\beta_0,\dots,\beta_i).\) 
\begin{rem}
    Observe that one can also obtain that presentation of the fundamental group taking into account the iterative structure of the knot. One can use the Mayer-Vietoris decomposition to apply Seifert-Van Kampen theorem in order to iteratively compute the fundamental group of the knot complement. We refer to \cite[Chapter 4]{BurdeKnots} for the details.
\end{rem}

Now, if $G$ is abelianized by factoring out its commutator subgroup, the abelianization \(\operatorname{Ab}(G)\) of \(G\) is the infinite free group generated by one element, see e.g.\cite[\S 5]{Zar32}. Let us denote it by \(\langle t\rangle=\operatorname{Ab}(G).\) Hence, in \(\operatorname{Ab}(G)\) all the elements become powers of \(t\) and it is easily checked that 
\[b_i=t^{n_i\cdots n_g},\quad u_i=t^{\obeta_i}.\]
We recall that \(n_1\cdots n_g=\obeta_0,\obeta_1,\dots,\obeta_g\) are the minimal generators of the semigroup of values of the curve and they are related to the Puiseux characteristic by eq.~ \eqref{eq:definbetabarra}. Thus, the semigroup algebra is obtained from the fundamental group $G$ of the knot complement and its abelianization as follows.
\medskip

Let \(R:=\mathbb{K}[b_1,u_1\dots,u_g]\) be the polynomial ring in the generators of the fundamental group. Let us endow \(R\) with the grading induced by the abelianization of the group, i.e. \(\deg(b_1)=\obeta_0\) and \(\deg(u_i)=\obeta_i\) for $i=1,\ldots , g$. Then, we have the following morphism induced by the defining relations fundamental group

\[\begin{array}{ccccccccc}
0&\longrightarrow& \mathbb{K}[b_1,u_1,\dots,u_g]^{g}&\longrightarrow& \mathbb{K}[b_1,u_1,\dots,u_g]&\overset{\varphi}{\longrightarrow}& \mathbb{K}[t^\nu\;:\;\nu\in \Gamma]&\longrightarrow& 0\\
& &&& b_1&\mapsto& t^{\obeta_0}& &\\
& &&& u_i&\mapsto& t^{\obeta_i}& &
\end{array}
\]
Here \(\ker(\varphi)\) is generated by the relations defining the fundamental group \(G\) i.e. we have \(\ker(\varphi)=(u_i^{n_i}=b_i^{q_i}u_{i-1}^{n_{i-1}n_i})_{i=1}^{g}.\) Moreover, each relation is homogeneous of degree \(n_i\obeta_i\) as, by eq.~\eqref{eq:definbetabarra}, we have
\[
q_in_i\cdots n_g+\obeta_{i-1}n_{i-1}n_i=n_i\big (q_in_{i+1}\cdots n_g+\obeta_{i-1}n_{i-1}\big )=n_i\Big(\frac{\beta_i-\beta_{i-1}}{e_i}e_i+\obeta_{i-1}n_{i-1}\Big)=n_i\obeta_i.
\]

 In this way we obtain the monomial curve associated to the semigroup algebra of the semigroup of values of the irreducible plane branch. Therefore it is clear that the gluing construction comes from the satellization via the application of the Seifert-Van Kampen theorem. This shows that the amalgamated decomposition of the fundamental group modulo a relation translates into the tensor product decomposition modulo the relation of the semigroup algebra, which is the algebraic interpretation of the gluing construction on the semigroup structure. 

In short, adding a characteristic exponent in the Puiseux expansion is translated to the gluing operation on the algebraic side, and to the satellization operation on the topological side. But these two operations are reflected in the corresponding polynomial invariant (namely, $(t-1)P_C(t)$ in the algebraic side resp. $\Delta_C(t)$ in the topological side), in the very same manner. Therefore, the coincidence between the Alexander polynomial and the Poincaré series in the irreducible case is a natural consequence of the coincidence of both operations.
\medskip

The sequence of iterations on the topological side corresponding to the algebraic gluing process is provided by knots obtained from the approximations of \(C\) defined by the \(g\) star points of the dual graph of \(C\) associated with the maximal contact values. Indeed, if \(K_1,\dots,K_g\) is the sequence of approximating knots, then Theorem 5.1 \cite{SW} (see also \cite[Theorem II]{Seifert50}) shows
\begin{proposition}\label{prop:alexirre}
\[
\Delta_{K_i}(t)=\Delta_{K_{i-1}}(t^{n_i})\cdot P(\obeta_i/e_i,n_i,t).
\]
\end{proposition}
Therefore Proposition \ref{prop:poincareirreduciblecase} and Proposition \ref{prop:alexirre} imply the relation \((t-1)P_C(t)=\Delta_K(t).\) This identification is indeed very deep: its ultimate reason is the correspondence between the presentations of the semigroup algebra and the first homology group of the link exterior.

\subsection{Iterative computation of the Alexander polynomial}
Our ordering in the set of branches implies an ordering in the set of the link components from the innermost to the outermost component. This allows us to compute the Alexander polynomial by means of the algorithm provided by Sumners and Woods in \cite{SW}. In recalling their procedure we will show its algebraic counterpart with our iterative computation of the Poincaré series explained in Section \ref{subsec:iterativepoincare}. Here we provide a more detailed description than the one in \cite{SW}: they only write down the explicit factorization in the case of a curve with 2 and 3 branches, whereas our explicit expressions for the factorizations of the Poincaré series are valid in the case with more than three branches.
\medskip

Before starting with the topological interpretation of our algebraic procedure, we need to present the building blocks of Sumners and Woods computation for the Alexander polynomials. Thanks to the iterative homology computations done in \cite[Sect.~V]{SW} (see also Subsection \ref{subsubsec:iterativehomology}), Sumners and Woods introduce a decomposition of the Alexander polynomial in each of the cases:

\begin{theorem} \cite[Theorems 5.2, 5.3 and 5.4]{SW}\label{thm:SW5}
  With the notation of Subsection \ref{subsubsec:iterativehomology}, let $\langle L,L'\rangle$ denote the homological linking number of \(L\) and \(L'\).
  \begin{enumerate}
      \item Assume \(L\) is a link with \(r\geq 2\) components and \(L'\) be the link obtained from an iteration of type \((1)\) via the knot \(K'\) with winding number \(p\neq 0.\) Then,
    \[
    \Delta_{L'}(t_1,\dots,t_r)=\Delta_L(t_1,\dots,t_r^p)\cdot\Delta_M\big(t_r,\prod_{i=1}^{r-1}t_i^{\langle L_i,L_r\rangle}\big),
    \]
    where \(M\) denotes the model link of two components formed by \(K'\) and the unknotted meridian curve on the boundary torus containing \(K'\).
      \item Assume \(L\) is a knot and \(\widehat{L}\) is the link with two components obtained from an iteration of type \((2)\) via the knot \(K'\) with winding number \(p\neq 0.\) Then,
    \[
    \Delta_{\widehat{L}}(t_1,t_2)=\Delta_L(t_1,t_2^p)\cdot \Delta_N(t_1,t_2),
    \]
    where \(N\) denotes the model link of two components formed by \(K'\) and the unknotted core of the torus containing \(K'.\)
      \item Assume \(L\) is a link with \(r\geq 2\) components and \(\widehat{L}\) be the link obtained from an iteration of type \((2)\) via the knot \(K'\) with winding number \(p\neq 0.\) Then,
   \[
   \Delta_{\widehat{L}}(t_1,\dots,t_r)=\Delta_L(t_1,\dots,t_rt_{r+1}^p)\cdot \Delta_P\big(t_r,t_{r+1},\prod_{i=1}^{r-1}t_i^{\langle L_i,L_r\rangle}\big),
   \]
    where \(P\) denotes the model link of three components formed by \(K',\) the unknotted meridian curve on the boundary of the torus containing \(K'\) and the unknotted core of the torus containing \(K'.\)
  \end{enumerate}
\end{theorem}

Moreover, Sumners and Woods \cite[Sect.~VI]{SW} show that the Alexander polynomials of the model links can be computed as follows:
\begin{theorem}\cite[Theorems 6.1, 6.2, 6.3]{SW} \label{thm:SW1} Let \(L_1\) be a torus knot of type \((\alpha,\beta),\) let \(M\) be a link of two components formed by the torus knot \(L_1\) linked with its unknotted exterior core, let \(N\) denote a link of two components formed by \(L_1\) and the unknotted core of the torus containing \(L_1\) and let \(P\) be a link of three components formed by the torus knot \(L_1\) linked with both its exterior unknotted core and its interior unknotted core. Let \(P(m,n,x),Q(m,n,x,y),B(m,n,x,y,z)\) be the polynomials defined in \eqref{eqn:defkeypoly}. Then, their Alexander polynomials are
\[\begin{split}
    \Delta_{L_1}(t)=P(\alpha,\beta,t),&\quad \Delta_{M}(t_1,t_2)=Q(\alpha,\beta,t_2,t_1),\\
    \Delta_N(t_1,t_2)=Q(\beta,\alpha,t_2,t_1)&\quad\text{and}\quad \Delta_{P}(t_1,t_2,t_3)=B(\alpha,\beta,t_2,t_1,t_3).
\end{split} \]
\end{theorem}
Now, we will use these results to show the topological counterpart of our algebraic iterative computation of the Poncaré series described in Section \ref{subsec:iterativepoincare}. We will see that as in the irreducible case, the coincidence between the Poincaré series and the Alexander polynomial comes from the equivalence between the algebraic process and the topological one. As a consequence, we provide an alternative proof of the Theorem of  Campillo, Delgado and Gusein-Zade \cite{CDGduke} and this new proof shows the intrinsic reason for the coincidence between both invariants. Moreover, we improve the computations of Sumners and Woods \cite{SW} for the algebraic link associated with a plane curve singularity as they provide closed formulas only for the cases of two and three branches and we describe the process in its full generality for any number of branches in terms of the value semigroup.
\subsubsection{The base cases}
As in Section \ref{subsec:iterativepoincare}, we start with the two base cases. First we consider the case where all the branches are smooth with the same contact. 
\begin{proposition}\label{prop:alexlemm1}
    Let \(C=\bigcup_{i=1}^{r} C_i\) such that \(C_i\) is smooth for all \(i=1,\dots,r\) and such that the contact pair is equal for all branches, i.e. \((q,c)=(q_{i,j},c_{i,j})\) for all \(i,j\in\ind\). Then,
     \[
     \Delta_C(\underline{t})=P_C(\underline{t}).
    \]
    \end{proposition}
    \begin{proof}
    We start with the oriented unknot \(K_1\) and first linking it with a knot of type \((c,1).\) After that, we consider the resulting link and start linking knots of type \((c,1)\) proceeding by induction on the number of branches, i.e. the number of components of the link. By combining \cite[Theorem 5.3 and Theorem 5.4]{SW}, Theorem \ref{thm:SW1} and Lemma \ref{prop:poincarelemm1}, and taking into account that we are operating from the innermost to the outer component, the claim follows.
    
    \end{proof}
We continue with the second base case,
    \begin{proposition}\label{prop:alexlemm2}
    Let \(C=\bigcup_{i=1}^{r} C_i\) be such that \(\Gamma^i=\langle\obeta_0^i,\obeta^i_1\rangle\) or \(\Gamma^i=\mathbb{N}.\) Assume that for each \(i\in \ind=\{1,\dots,r\}\) such that \(C_i\) is a singular branch we have \(l:=\Big\lfloor\frac{\obeta^i_1}{\obeta^i_0}\Big\rfloor=\Big\lfloor\frac{\obeta^j_1}{\obeta^j_0}\Big\rfloor\) if \(j\neq i\) and \(C_j\) is singular. Moreover, assume that the contact pairs are of the form \((q_{i,j},c_{i,j})\in\{(1,0),(0,l)\}.\) Then,

    \[
    \Delta_L(\underline{t})=P_C(\underline{t}).
    \]
\end{proposition}
\begin{proof}
    Since the branches of \(C\) are ordered in such a way the components of the link go from the innermost to the outermost, the first components are those corresponding to smooth branches, hence their corresponding link components are of type \((k,1)\), where \(k\) depends on the contact between them. The last components are those corresponding to the singular branches, which have associated components of type \((\obeta_0,\obeta_1).\) Therefore, by using \cite[Theorem 5.3 and Theorem 5.4]{SW}, Theorem \ref{thm:SW1} and Lemma \ref{prop:poincarelemm2}, and taking into account that we are operating from the innermost to the outermost component, the claim follows.
\end{proof}
\begin{rem}
We have explicitly given the expression of those base cases as in \cite{SW} the Alexander polynomial is not explicitly computed.
\end{rem}
    
\subsubsection{The general procedure}

We now proceed to show the general procedure of the coincidence between both invariants. To do so we will follow the structure of Section \ref{subsec:iterativepoincare}. Let \((q,c)\)
be the contact pair of \(C\) and let \(C_{\alpha_{q-1}}\) be the corresponding approximating curve as in Subsection \ref{subsec:generalprocedure}. The curve \(C_{\alpha_{q-1}}\) is irreducible and then we can use the results of Section \ref{subsec:alexirreducible} to check that \((t-1)P_{C_{\alpha_{q-1}}}(t)=\Delta_{K_{\alpha_{q-1}}}(t).\) We will simplify notation as in Section \ref{subsec:iterativepoincare} and we will denote \(\Delta_\sigma(t)\) to \(\Delta_{L_\sigma}(t)\) where \(L_\sigma\) is the link of the curve \(C_\sigma.\) Now we distinguish the cases \(c>0\) and \(c=0,\) as done in Section \ref{subsec:iterativepoincare}. There is nothing to prove if \(c>0\) as the approximating curve is still irreducible: then, let us assume \(c=0.\) In this case the corresponding approximating curve \(C_{\sigma_{0}}\) has \(s\)~branches, hence its associated link has \(s\) components and we can compute its Alexander polynomial as follows:

\begin{proposition}\label{prop:aux1alex}
    \[\begin{split}
        \Delta_\sigma(t_1,\dots,t_s)=&\Delta_T(t_1^{p_{q,I_1}}\cdots t_s^{p_{q,I_s}})\cdot Q\Big(w_{q,I_1},p_{q,I_1},t_1,\prod_{k=2}^{s} t_k^{w_{q,I_k}}\Big)\\ \cdot & \bigg ( \prod_{j=2}^{s-1}B\big(p_{q,j},w_{q,j},t_j,\prod_{k\cdot n<j}t_k^{p_{q,k}},\prod_{k>j}t_{k}^{w_{q,k}}\big)\bigg)\cdot Q\Big(p_{q,s},w_{q,s},t_s,\prod_{k=1}^{s-1}t_k^{p_{q,k}}\Big).
    \end{split}\]
    In particular,
    \(P_\sigma(t_1,\dots,t_s)=\Delta_\sigma(t_1,\dots,t_s)\).
\end{proposition}
\begin{proof}
    The formula for the Alexander polynomial follows from the application of \cite[Theorems 5.2, 5.3 and 5.4]{SW}. The equality between the Poincaré series and the Alexander polynomial can be now deduced from Lemma \ref{lem:aux1} as follows: we have seen that
    \[
    \Delta_T(t_1^{p_{q,I_1}}\cdots t_s^{p_{q,I_s}})=P_{C_{\alpha_{q-1}}}(t_1^{p_{q,I_1}}\cdots t_s^{p_{q,I_s}})\cdot (t_1^{p_{q,I_1}}\cdots t_s^{p_{q,I_s}}-1).\]
By Lemma \ref{lem:aux1}, we only need to check that 

\[\Delta_T(t_1^{p_{q,I_1}}\cdots t_s^{p_{q,I_s}})\cdot Q\big(p_{q,s},w_{q,s},t_s,\prod_{k=1}^{s-1}t_k^{p_{q,k}}\big)=P_{C_{\alpha_{q-1}}}(t_1^{p_{q,I_1}}\cdots t_s^{p_{q,I_s}})\cdot \big((\prod_{k=1}^{s}t_k^{p_{q,k}})^{w_{q,s}}-1\big),\]
which follows by the previous identity and the definition of the \(Q\)--polynomial.
\end{proof}
Despite of the fact that we do not have the gluing operation at our disposal in the case of more than one branch, this equality shows that the algebraic construction produced in Section \ref{sec:Poincareseries} is the exact analogue to the topological construction: The main idea is to realize that the polynomial \(Q\) will be used each time when a maximal contact value is added to the value semigroup of the plane curve; observe that, if there is no maximal contact curve, i.e. the case where there exist branches in the smooth packages, then \(Q=1\) since \(p=1.\) On the other hand, the polynomial \(B\) will be used each time when a proper star point is added to the dual graph. In appending each type of value to the semigroup we need to be careful with the order, as our process extremely depends on the total order of the star points, i.e. the set of principal values of the semigroup. To go on showing the equivalence between the algebraic and topological constructions is now an easy routine. 
\medskip

Let us continue with the case where \(c>0\). In this case, the same proof of Proposition \ref{prop:aux1alex} but with the use of Lemma \ref{lem:aux2} shows that the Poincaré series coincides with the Alexander polynomial for the approximate curve associated with the proper star point \(\sigma_0.\) Once \(\sigma_0\) is already computed, we continue with the procedure of Subsection \ref{subsec:generalprocedure}. The distinguished points \(T,\sigma\) are now at the stage \(T=\sigma_0\), and \(\sigma\) the next start point to be considered to compute the Poincaré series. 
\medskip

Let \(\ind=(\bigcup_{p=1}^{t} I_{p})\cup(\bigcup_{p=t+1}^{s} I_p)\) be the partition created at \(\sigma_0.\) Denote by \(\osigma_1,\dots,\osigma_\epsilon\) the star points between \(\sigma_0\) and the point where the geodesics of \(I_1\) goes through. Recall that at this point we have 
\[
\alpha_1\prec \cdots\prec\alpha_q\preceq \sigma_0\preceq \osigma_1\preceq\cdots\preceq\osigma_\epsilon\preceq P.
\]
Then,
\begin{enumerate}[wide, labelwidth=!, labelindent=0pt]
    \item Assume \(|I_1|=1,\) we have two cases to consider:
 \begin{enumerate}
 \item  The semigroup \(\Gamma^1\) of the first branch \(C^1\) of \(C\) has \(q\) minimal generators. As in Subsection \ref{subsec:generalprocedure} in this case there is nothing to prove.
        \item The semigroup \(\Gamma^1\) of the first branch \(C^1\) of \(C\) has \(g_1>q\) minimal generators. Then,
        \[\mathcal{S}_1=\{\alpha_1\prec \cdots\prec\alpha_q\preceq \sigma_0\preceq \osigma_1\preceq\cdots\preceq\osigma_\epsilon\prec \alpha_1^{I_1}\prec\cdots \prec \alpha_{g_1-q}^{I_1}\}.\]
        We are in the case \(T=\sigma_0\) and \(\sigma=\alpha_1^{I_1};\) for \(i=2,\dots,g_1-q\) we will consider \(T=\alpha^{I_1}_{i-1}\) and \(\sigma=\alpha^{I_1}_i.\) At each stage, we need to show that the Poincaré series \(P_{C_{\alpha^{I_1}_i}}=P_{\alpha^{I_1}_i}=P_\sigma\) equals the Alexander polynomial. At each stage, for simplicity of notation we write \(p_\sigma:=p_{q+i,I_1}\) and \(w_\sigma:=w_{q+i,I_1}.\) By definition of \(C_\sigma\) its associated link \(L_\sigma=L_1\cup\cdots \cup L_s\) is obtained from the link of \(C_T\) from an operation of type \((1)\) about \(L_1\) via a torus knot of type \((p_\sigma,w_{q+i,I_k})\) with winding number \(p_\sigma.\) Then, applying Theorem \ref{thm:SW5}(1), Theorem \ref{thm:SW1}, Proposition \ref{prop:case1b} and Proposition \ref{prop:aux1alex} we have 
        \begin{proposition} \label{prop:case1balex}
            \[P_{\sigma}(\underline{t})=\Delta_T(t_1^{p_\sigma},t_2,\dots,t_s)\cdot Q\big(w_\sigma,p_\sigma,t_1,\prod_{k=2}^{s} t^{w_{q+i,I_k}}\big)=\Delta_\sigma(\underline{t}).\]
        \end{proposition}
       
    \end{enumerate}
     \item Assume \(|I_1|>1.\) There are two cases to be distinguished:
    \begin{enumerate}
        \item For all \(j\in I_1\) we  have \(g_j=q,\) i.e. the semigroups \(\Gamma^{j}\) have \(q\)--minimal generators. In this case, \(\sigma=\sigma_0^{I_1}\) is the first separation point of the branches of \(I_1\) and let \(I_1=\bigcup_{k=1}^{s_1} I_{1,k}\) be the induced index partition (see Subsection \ref{subsubsec:truncationstar}). By definition of \(C_\sigma\) its associated link \(L_\sigma=L_1\cup\cdots \cup L_{s_1}\cup L_{s_1+1}\cup\cdots \cup L_{s_1+s-1}\) is obtained from the link \(L_T=L'_1\cup L_2\cup \cdots \cup L_{s_1+s-1}\) of \(C_T\) from successive operations of type \((2)\) about \(L'_1\) with winding number \(p_{\sigma,I_{1,k}}=1\) since \(\sigma\) is an ordinary point and thus the term in the topological Puiseux series is not a characteristic exponent. Then, the application of Theorem \ref{thm:SW5}(3), Theorem \ref{thm:SW1}, Proposition \ref{prop:case2a} and Proposition \ref{prop:aux1alex} yields
        \begin{proposition} \label{prop:alexcase2a}
    \[
    \begin{split}
    P_\sigma(t_1,\dots,t_{s_1},t_{s_1+1},\dots,&t_{s_1+s-1})=\Delta_T(t_1\cdots t_{s_1},t_{s_1+1},\dots,t_{s_1+s-1})\\
    \cdot&\Big(\prod_{j=2}^{s_1}B(w_{q+1,I_{1,j}},1,t_{I_{1,j}},(\prod_{k>j}^{s_1}t_{k}^{w_{q+1,I_{1,k}}})(\prod_{k=s_1+1}^{s_1+s-1}t_{k}^{w_{q+1,I_{k-s_1+1}}}),\prod_{k<j}t_k)\Big)\\
    &=\Delta_\sigma(t_1,\dots,t_{s_1},t_{s_1+1},\dots,t_{s_1+s-1}).
    \end{split}
    \]
    \end{proposition}

        We then do \(T=\sigma_0^{I_{1}}\) and \(\sigma=\sigma_0^{I_{1,1}}\) to continue the process.
        
        \item We assume that \(g_j>q\) for some \(j\in I_1.\) We distinguish again two subcases:
        \begin{enumerate}[wide, labelwidth=!, labelindent=0pt]
            \item[\textbf{(i)}] Assume \((f_1|\cdots|f_{|I_1|})\leq (q+1,0)\) and for simplicity assume  \(g_1>q\). Let us denote by \(\sigma_{0}^{I_1}\) be the first proper star point of the package \(I_1.\) 
            Let \(I_1=\bigcup_{k=1}^{s_1} I_{1,k}\) be the partition associated with the proper star \(\sigma_{0}^{I_1}.\) In this case, the link of \(C_\sigma\) is obtained from the link of \(C_T\) by successive operations of type \((2)\) together with one operation of type \((1)\) all of them performed along the first component at each step. Then, applying Theorem \ref{thm:SW5}(1), Theorem \ref{thm:SW5}(3), Theorem \ref{thm:SW1}, Proposition \ref{prop:case2bi} and Proposition \ref{prop:aux1alex} we have 
            
            \begin{proposition}\label{prop:alexcase2bi}
              \begin{align*}
& P_\sigma(t_1,\dots,t_{s_1},t_{s_1+1},\dots,t_{s_1+s-1})= \Delta_T(t^{p_{q+1,I_{1,1}}}_1\cdots t^{p_{q+1,I_{1,s_1}}}_{s_1},t_{s_1+1},\dots,t_{s_1+s-1})\\
& \hspace{20pt} \cdot  Q\Big(w_{q+1,I_{1,1}},p_{q+1,I_{1,1}},t_1,(\prod_{k=2}^{s_1}t_{k}^{w_{q+1,I_{1,k}}})(\prod_{k=s_1+1}^{s_1+s-1}t_{k}^{w_{q+1,I_{k-s_1+1}}})\Big)\\
&\hspace{20pt} \cdot  \Big(\prod_{j=2}^{s_1}B(w_{q+1,I_{1,j}},p_{q+1,I_{1,j}},t_{I_{1,j}},(\prod_{k>j}^{s_1}t_{k}^{w_{q+1,I_{1,k}}})(\prod_{k=s_1+1}^{s_1+s-1}t_{k}^{w_{q+1,I_{k-s_1+1}}}),\prod_{k<j}t^{p_{q+1,I_{1,k}}}_k)\Big)\\
&=\Delta_\sigma(t_1,\dots,t_{s_1},t_{s_1+1},\dots,t_{s_1+s-1}).
 \end{align*}
            \end{proposition}

            \item[\textbf{(ii)}] Assume \((f_1|\cdots|f_{|I_1|})> (q+1,0),\) i.e \((f_1|\cdots|f_{|I_1|})\geq (q+1,c)\) with \(c\neq 0,\) and for simplicity \(g_1> q.\) Let us denote by \((q_{I_1},c_{I_1}):= (f_1|\cdots|f_{|I_1|}).\) Since \((q_{I_1},c_{I_1})> (q+1,0),\) there are \(q_{I_1}-q\) star points which are non-proper between \(\osigma_\epsilon\) and \(\sigma_0^{I_1},\) i.e. 
            \[\osigma_{\epsilon}\prec \alpha_{q+1}\prec \cdots \prec \alpha_{q_{I_1}}\preceq \sigma_0^{I_1}.\] 
The situation, in this case, is the same as in the case \((1)(b),\) and analogous reasoning yields the following: 

            \begin{proposition}
 \[P_{\sigma}(\underline{t})=\Delta_T(t_1^{p_\sigma},t_2,\dots,t_s)\cdot Q\big(w_\sigma,p_\sigma,t_1,\prod_{k=2}^{s} t_k^{w_{q+i,I_k}}\big)=\Delta_{\sigma}(\underline{t}).\]                
            \end{proposition}
        \end{enumerate}
    \end{enumerate}
\end{enumerate}

Our method shows the coincidence between both invariants from a very explicit point of view; moreover our proof makes no appeal to the results of Eisenbud and Neumann results \cite{EN}. Thus, it constitutes an alternative proof of the one given by Campillo, Delgado and Gusein-Zade \cite{CDGduke}. This new proof provides the intrinsic topological nature of the Poincaré series of the value semigroup: it shows that the algebraic operation of adding one maximal contact value is in correspondence with a topological operation of satellization along one component of the link; the algebraic operation of adding one value associated to a proper star point is in correspondence with the topological operation of adding one branch to the associated link. Moreover, those are the only operations to be aware of in order to construct the value semigroup of a plane curve singularity. The associated Mayer-Vietoris splitting is translated into the algebraic setting via a generalization of the gluing construction valid in the irreducible case. From this perspective, it is reasonable to ask whether the property of being a complete intersection isolated singularity is merely an algebraic characteristic or it possesses a deeper, intrinsic topological significance.

\section{Historical remarks}
Connections between knot theory and plane curve singularities were apparently realized for the first time by Poul Heegaard at the end of the 19th century. He pursued to develop topological tools for the investigation of algebraic surfaces \cite[\S76,\S 77]{Epple}. Wilhelm Wirtingen was aware of this, and proposed his student Karl Brauner the following problem: the description of the singularities of an algebraic function of two variables by means of the intersection of its discriminant curve with the spherical boundary in a small neighborhood of a singular point, as well as the associated branched covering. Brauner solved the problem both for irreducible and reducible discriminant curves \cite{Brauner28}. He actually described the topology of the link in terms of repeated cabling, as well as an explicit presentation of the fundamental group of the complement of the link \cite{Neu03}. Erich K\"ahler reproved this using a more modern approach based on Puiseux pairs \cite{Kahler}. However, Brauner left an open question: Is it possible that two plane curve singularities with different Puiseux pairs could have topological equivalent neighborhoods? At this point, the fact that the knot associated to the singularity is an iterated torus knot characterized by the Puiseux pairs was known.
\medskip

Werner Burau answered the above question to the positive in both the irreducible \cite{Burau33} and the particular reducible case of two branches \cite{Burau34}; at the same time, Oskar Zariski solved the question out in the irreducible case by computing the Alexander polynomial of the knot \cite{Zar32}. He had already investigated the connections between knots and plane curves in \cite{Zar29}. The Alexander polynomial is an important invariant in knot theory; it was introduced by J.W. Alexander in order to determine the knot type \cite{Alexander}. Indeed, a crucial question in the mathematical atmosphere at that time was about the invariants which completely determine the topology of a plane curve singularity. As already mentioned, the results of Burau \cite{Burau33, Burau34} showed that the Alexander polynomial completely determines the topology in the irreducible case and in the reducible case of two branches. Surprisingly, it was not until the 80's when a full answer to this question was given by Yamamoto \cite{Yamamoto}, who proved that the Alexander polynomial classifies the topological type of plane curve singularities. It is interesting to point out here that the clue of Yamamoto's result is also the use of Sumners and Woods \cite{SW} description of the iterative computation of the Alexander polynomial we used for our purpose.
\medskip

A key fact to understand the theory of this topic developed in the second half of the twentieth century is the connection of algebraic links and \(3\)--manifolds: the classical example of \(3\)--manifold is the exterior of an algebraic link embedded in the \(3\)-sphere. In 1967, Friedhelm Waldhausen \cite{waldhausenII} first introduced the concept of graph manifold, from which link exteriors are the canonical example. Waldhausen \cite{waldhausenI,waldhausenII} provided in fact a decomposition of the \(3\)--manifold with very nice geometrical properties further explored by Jaco-Shalen \cite{JacoShalen} and Johannson \cite{Johannson} independently. This decomposition is at the core of the theoretical approach to the study of link exteriors proposed by Eisenbud and Neumann \cite{EN} in 1985.
\medskip

Eisenbud and Neumann's theory provides a ``good" setting to compute invariants of a link in an additive way deepening the pioneering results on the Alexander polynomial of Seifert \cite{Seifert50} and Torres \cite{Torres} from a modern perspective. In fact, another advantage of the Eisenbud and Neumann theory is that their construction allows to compute some other invariants of the link and not only the Alexander polynomial. The occurrence of the Eisenbud and Neumann theory in the context of Thurston developments on the geometry of \(3\)--manifolds \cite{Thurston1, Thurston2, Thurston3} and its connection with the Poincaré conjecture may suggest why Eisenbud and Neumann had such an impact in the research perspective of the topic at the end of the twentieth century.   
\medskip

From the algebraic side, the value semigroup of an irreducible plane curve singularity was introduced by Roger Apèry \cite{apery} in 1946. However, the milestone in the study of singularities from a purely algebraic point of view was the foundational articles of Zariski \cite{zarequiI,zarequiII,zarequiIII,zarsaturationI,zarsaturationII,zarsaturationIII} about the notion of equisingularity and the saturation of local rings. Combining the new concept of equisingularity with his previous results \cite{Zar32}, Zariski showed \cite{zarequiII} the first connection between both approaches, the topological and the algebraic. He proved that the semigroup of values is in fact a topological invariant of an irreducible plane curve; thus equisingularity becomes topological invariance for irreducible plane curves. 
\medskip

As pointed out by Félix Delgado \cite{Delgmanuscripta1}, R. Waldi \cite{Waldi} proved that the value semigroup of a plane curve with several branches is invariant in an equisingularity class. After that, Delgado himself \cite{Delgmanuscripta1,Delgmanuscripta2} (see also the works of A. Garc\'ia \cite{Garcia} and V. Bayer \cite{Bayer} for the bibranch case) provided the full combinatorial algebraic description of this semigroup together with the natural generalization of some properties of the irreducible case. After Delgado's description of the value semigroup for a reducible plane curve singularity, A. Campillo, K. Kiyek and himself introduced a sort of generating function \cite{CDKmanuscr} which can be associated with the value semigroup, the so-called Poincar\'e series. This has been profusely studied by Campillo, Delgado and Gusein-Zade later on, e.g. in \cite{CDG99a, CDG99b, CDG00, CDG02, CDG03a, CDG04, CDG05, CDG07}; surprisingly, in one of his investigations they showed that the Poincar\'e series ---which turns out to be a polynomial in the case of a plane curve singularity with more than one branch--- coincides with the Alexander polynomial associated to the link of the singularity \cite{CDGduke}. They left however open the ultimate reason that could explain this fortunate circumstance, whose answer is sketched---we believe---in our paper.



\printbibliography

@book {casas,
    AUTHOR = {Casas-Alvero, Eduardo},
     TITLE = {Singularities of plane curves},
    SERIES = {London Mathematical Society Lecture Note Series},
    VOLUME = {276},
 PUBLISHER = {Cambridge University Press, Cambridge},
      YEAR = {2000},
     PAGES = {xvi+345},
      ISBN = {0-521-78959-1},
   MRCLASS = {14H20 (32S05 32S15 32S50)},
  MRNUMBER = {1782072},
MRREVIEWER = {Alejandro Melle-Hern\'{a}ndez},
       DOI = {10.1017/CBO9780511569326},
       URL = {https://doi.org/10.1017/CBO9780511569326},
}

@article {CDGDocumenta,
    AUTHOR = {Campillo, A. and Delgado, F. and Gusein-Zade, S. M.},
     TITLE = {Equivariant {P}oincar\'{e} series and topology of valuations},
   JOURNAL = {Doc. Math.},
  FJOURNAL = {Documenta Mathematica},
    VOLUME = {21},
      YEAR = {2016},
     PAGES = {271--286},
      ISSN = {1431-0635},
   MRCLASS = {14B05 (13A18 14R20)},
  MRNUMBER = {3505134},
MRREVIEWER = {Dmitry Kerner},
}

@book {ZariskiModuli,
    AUTHOR = {Zariski, Oscar},
     TITLE = {The moduli problem for plane branches},
    SERIES = {University Lecture Series},
    VOLUME = {39},
      NOTE = {With an appendix by Bernard Teissier,
              Translated from the 1973 French original by Ben Lichtin},
 PUBLISHER = {American Mathematical Society, Providence, RI},
      YEAR = {2006},
     PAGES = {viii+151},
      ISBN = {978-0-8218-2983-7; 0-8218-2983-1},
   MRCLASS = {14H20 (01A75 14H45)},
  MRNUMBER = {2273111},
       DOI = {10.1090/ulect/039},
       URL = {https://doi.org/10.1090/ulect/039},
}

@article {zarsaturationIII,
    AUTHOR = {Zariski, Oscar},
     TITLE = {General theory of saturation and of saturated local rings.
              {III}. {S}aturation in arbitrary dimension and, in particular,
              saturation of algebroid hypersurfaces},
   JOURNAL = {Amer. J. Math.},
  FJOURNAL = {American Journal of Mathematics},
    VOLUME = {97},
      YEAR = {1975},
     PAGES = {415--502},
      ISSN = {0002-9327},
   MRCLASS = {14B05},
  MRNUMBER = {389893},
MRREVIEWER = {Joseph Lipman},
       DOI = {10.2307/2373720},
       URL = {https://doi.org/10.2307/2373720},
}

@article {zarsaturationI,
    AUTHOR = {Zariski, Oscar},
     TITLE = {General theory of saturation and of saturated local rings.
              {I}. {S}aturation of complete local domains of dimension one
              having arbitrary coefficient fields (of characteristic zero)},
   JOURNAL = {Amer. J. Math.},
  FJOURNAL = {American Journal of Mathematics},
    VOLUME = {93},
      YEAR = {1971},
     PAGES = {573--648},
      ISSN = {0002-9327},
   MRCLASS = {13.95},
  MRNUMBER = {282972},
MRREVIEWER = {Joseph Lipman},
       DOI = {10.2307/2373462},
       URL = {https://doi.org/10.2307/2373462},
}

@article {zarsaturationII,
    AUTHOR = {Zariski, Oscar},
     TITLE = {General theory of saturation and of saturated local rings.
              {II}. {S}aturated local rings of dimension {$1$}},
   JOURNAL = {Amer. J. Math.},
  FJOURNAL = {American Journal of Mathematics},
    VOLUME = {93},
      YEAR = {1971},
     PAGES = {872--964},
      ISSN = {0002-9327},
   MRCLASS = {14M05 (13H10)},
  MRNUMBER = {299607},
MRREVIEWER = {Joseph Lipman},
       DOI = {10.2307/2373741},
       URL = {https://doi.org/10.2307/2373741},
}

@article {zarequiI,
    AUTHOR = {Zariski, Oscar},
     TITLE = {Studies in equisingularity. {I}. {E}quivalent singularities of
              plane algebroid curves},
   JOURNAL = {Amer. J. Math.},
  FJOURNAL = {American Journal of Mathematics},
    VOLUME = {87},
      YEAR = {1965},
     PAGES = {507--536},
      ISSN = {0002-9327},
   MRCLASS = {14.18},
  MRNUMBER = {177985},
MRREVIEWER = {Y. Nakai},
       DOI = {10.2307/2373019},
       URL = {https://doi.org/10.2307/2373019},
}

@article {zarequiII,
    AUTHOR = {Zariski, Oscar},
     TITLE = {Studies in equisingularity. {II}. {E}quisingularity in
              codimension {$1$} (and characteristic zero)},
   JOURNAL = {Amer. J. Math.},
  FJOURNAL = {American Journal of Mathematics},
    VOLUME = {87},
      YEAR = {1965},
     PAGES = {972--1006},
      ISSN = {0002-9327},
   MRCLASS = {14.18},
  MRNUMBER = {191898},
MRREVIEWER = {Y. Nakai},
       DOI = {10.2307/2373257},
       URL = {https://doi.org/10.2307/2373257},
}

@article {zarequiIII,
    AUTHOR = {Zariski, Oscar},
     TITLE = {Studies in equisingularity. {III}. {S}aturation of local rings
              and equisingularity},
   JOURNAL = {Amer. J. Math.},
  FJOURNAL = {American Journal of Mathematics},
    VOLUME = {90},
      YEAR = {1968},
     PAGES = {961--1023},
      ISSN = {0002-9327},
   MRCLASS = {14.18},
  MRNUMBER = {237493},
MRREVIEWER = {Joseph Lipman},
       DOI = {10.2307/2373492},
       URL = {https://doi.org/10.2307/2373492},
}

@article {Kunz70,
    AUTHOR = {Kunz, Ernst},
     TITLE = {The value-semigroup of a one-dimensional {G}orenstein ring},
   JOURNAL = {Proc. Amer. Math. Soc.},
  FJOURNAL = {Proceedings of the American Mathematical Society},
    VOLUME = {25},
      YEAR = {1970},
     PAGES = {748--751},
      ISSN = {0002-9939},
   MRCLASS = {13.95},
  MRNUMBER = {265353},
MRREVIEWER = {C. Vraciu},
       DOI = {10.2307/2036742},
       URL = {https://doi.org/10.2307/2036742},
}

@book {wall,
    AUTHOR = {Wall, C. T. C.},
     TITLE = {Singular points of plane curves},
    SERIES = {London Mathematical Society Student Texts},
    VOLUME = {63},
 PUBLISHER = {Cambridge University Press, Cambridge},
      YEAR = {2004},
     PAGES = {xii+370},
      ISBN = {0-521-83904-1; 0-521-54774-1},
   MRCLASS = {14H20 (14H50 32S55 57R45)},
  MRNUMBER = {2107253},
MRREVIEWER = {C\'{i}cero Carvalho},
       DOI = {10.1017/CBO9780511617560},
       URL = {https://doi.org/10.1017/CBO9780511617560},
}

@article {CDGlondon,
    AUTHOR = {Campillo, A. and Delgado, F. and Gusein-Zade, S. M.},
     TITLE = {On generators of the semigroup of a plane curve singularity},
   JOURNAL = {J. London Math. Soc. (2)},
  FJOURNAL = {Journal of the London Mathematical Society. Second Series},
    VOLUME = {60},
      YEAR = {1999},
    NUMBER = {2},
     PAGES = {420--430},
      ISSN = {0024-6107},
   MRCLASS = {14B05},
  MRNUMBER = {1724869},
MRREVIEWER = {Mutsuo Oka},
       DOI = {10.1112/S0024610799007917},
       URL = {https://doi.org/10.1112/S0024610799007917},
}

@article {CDG99a,
    AUTHOR = {Guseu{i}n-Zade, S. M. and Delgado, F. and Kampillo, A.},
     TITLE = {On the monodromy of a plane curve singularity and the
              {P}oincar\'{e} series of its ring of functions},
   JOURNAL = {Funktsional. Anal. i Prilozhen.},
  FJOURNAL = {Funktsional\textquotesingle nyu{i} Analiz i ego Prilozheniya},
    VOLUME = {33},
      YEAR = {1999},
    NUMBER = {1},
     PAGES = {66--68},
      ISSN = {0374-1990},
   MRCLASS = {32S55 (14H20 32S40)},
  MRNUMBER = {1711890},
MRREVIEWER = {J. S. Joel},
       DOI = {10.1007/BF02465144},
       URL = {https://doi.org/10.1007/BF02465144},
}

@article {CDG99b,
    AUTHOR = {Guseu{i}n-Zade, S. M. and Delgado, F. and Kampillo, A.},
     TITLE = {The {A}lexander polynomial of a plane curve singularity, and
              the ring of functions on the curve},
   JOURNAL = {Uspekhi Mat. Nauk},
  FJOURNAL = {Uspekhi Matematicheskikh Nauk},
    VOLUME = {54},
      YEAR = {1999},
    NUMBER = {3(327)},
     PAGES = {157--158},
      ISSN = {0042-1316},
   MRCLASS = {32S55 (32S40)},
  MRNUMBER = {1728649},
MRREVIEWER = {J. S. Joel},
       DOI = {10.1070/rm1999v054n03ABEH000160},
       URL = {https://doi.org/10.1070/rm1999v054n03ABEH000160},
}

@article {CDG00,
    AUTHOR = {Guseu{i}n-Zade, S. M. and Delgado, F. and Kampillo, A.},
     TITLE = {Integration with respect to the {E}uler characteristic over a
              function space, and the {A}lexander polynomial of a plane
              curve singularity},
   JOURNAL = {Uspekhi Mat. Nauk},
  FJOURNAL = {Uspekhi Matematicheskikh Nauk},
    VOLUME = {55},
      YEAR = {2000},
    NUMBER = {6(336)},
     PAGES = {127--128},
      ISSN = {0042-1316},
   MRCLASS = {32S55 (32S40 57M25 58K65)},
  MRNUMBER = {1840363},
MRREVIEWER = {Aleksandr G. Aleksandrov},
       DOI = {10.1070/rm2000v055n06ABEH000338},
       URL = {https://doi.org/10.1070/rm2000v055n06ABEH000338},
}

@article {CDG02,
    AUTHOR = {Guseu{i}n-Zade, S. M. and Delgado, F. and Kampillo, A.},
     TITLE = {Integrals with respect to the {E}uler characteristic over
              spaces of functions, and the {A}lexander polynomial},
   JOURNAL = {Tr. Mat. Inst. Steklova},
  FJOURNAL = {Trudy Matematicheskogo Instituta Imeni V. A. Steklova},
    VOLUME = {238},
      YEAR = {2002},
    NUMBER = {Monodromiya v Zadachakh Algebr. Geom. i Differ. Uravn.},
     PAGES = {144--157},
      ISSN = {0371-9685},
   MRCLASS = {32S55 (14H20)},
  MRNUMBER = {1969310},
MRREVIEWER = {Jan Stevens},
}

@article {CDG03a,
    AUTHOR = {Campillo, A. and Delgado, F. and Gusein-Zade, S. M.},
     TITLE = {The {A}lexander polynomial of a plane curve singularity and
              integrals with respect to the {E}uler characteristic},
   JOURNAL = {Internat. J. Math.},
  FJOURNAL = {International Journal of Mathematics},
    VOLUME = {14},
      YEAR = {2003},
    NUMBER = {1},
     PAGES = {47--54},
      ISSN = {0129-167X},
   MRCLASS = {14H20 (32S55)},
  MRNUMBER = {1955509},
MRREVIEWER = {Thiruvalloor E. Venkata Balaji},
       DOI = {10.1142/S0129167X03001703},
       URL = {https://doi.org/10.1142/S0129167X03001703},
}

@article {CDG04,
    AUTHOR = {Campillo, A. and Delgado, F. and Gusein-Zade, S. M.},
     TITLE = {Poincar\'{e} series of a rational surface singularity},
   JOURNAL = {Invent. Math.},
  FJOURNAL = {Inventiones Mathematicae},
    VOLUME = {155},
      YEAR = {2004},
    NUMBER = {1},
     PAGES = {41--53},
      ISSN = {0020-9910},
   MRCLASS = {14J17 (32S25)},
  MRNUMBER = {2025300},
MRREVIEWER = {De-Qi Zhang},
       DOI = {10.1007/s00222-003-0310-y},
       URL = {https://doi.org/10.1007/s00222-003-0310-y},
}

@article {CDG05,
    AUTHOR = {Campillo, A. and Delgado, F. and Gusein-Zade, S. M.},
     TITLE = {Poincar\'{e} series of curves on rational surface singularities},
   JOURNAL = {Comment. Math. Helv.},
  FJOURNAL = {Commentarii Mathematici Helvetici. A Journal of the Swiss
              Mathematical Society},
    VOLUME = {80},
      YEAR = {2005},
    NUMBER = {1},
     PAGES = {95--102},
      ISSN = {0010-2571},
   MRCLASS = {14H20 (14J17 32S25)},
  MRNUMBER = {2130568},
MRREVIEWER = {Gerhard Pfister},
       DOI = {10.4171/cmh/6},
       URL = {https://doi.org/10.4171/cmh/6},
}

@article {CDG07,
    AUTHOR = {Campillo, A. and Delgado, F. and Gusein-Zade, S. M.},
     TITLE = {Multi-index filtrations and generalized {P}oincar\'{e} series},
   JOURNAL = {Monatsh. Math.},
  FJOURNAL = {Monatshefte f\"{u}r Mathematik},
    VOLUME = {150},
      YEAR = {2007},
    NUMBER = {3},
     PAGES = {193--209},
      ISSN = {0026-9255},
   MRCLASS = {32S50 (14H20)},
  MRNUMBER = {2308548},
MRREVIEWER = {Alejandro Melle-Hern\'{a}ndez},
       DOI = {10.1007/s00605-006-0437-1},
       URL = {https://doi.org/10.1007/s00605-006-0437-1},
}

@incollection {Neu03,
    AUTHOR = {Neumann, Walter},
     TITLE = {Topology of Hypersurface Singularities},
 BOOKTITLE = {Erich Kahler - Mathematische Werke, Math-
ematical works},
    SERIES = {Contemp. Math.},
    VOLUME = {778},
     PAGES = {727--736},
 PUBLISHER = {ABerlin, New York: Walter de
Gruyter Verlag},
      YEAR = {[2003] \copyright 2003},
}

@article {Kahler,
    AUTHOR = {K\"{a}hler, Erich},
     TITLE = {\"{U}ber die {V}erzweigung einer algebraischen {F}unktion zweier
              {V}er\"{a}nderlichen in der {U}mgebung einer singul\"{a}ren {S}telle},
   JOURNAL = {Math. Z.},
  FJOURNAL = {Mathematische Zeitschrift},
    VOLUME = {30},
      YEAR = {1929},
    NUMBER = {1},
     PAGES = {188--204},
      ISSN = {0025-5874},
   MRCLASS = {DML},
  MRNUMBER = {1545053},
       DOI = {10.1007/BF01187762},
       URL = {https://doi.org/10.1007/BF01187762},
}

@article {Zar29,
    AUTHOR = {Zariski, Oscar},
     TITLE = {On the {P}roblem of {E}xistence of {A}lgebraic {F}unctions of
              {T}wo {V}ariables {P}ossessing a {G}iven {B}ranch {C}urve},
   JOURNAL = {Amer. J. Math.},
  FJOURNAL = {American Journal of Mathematics},
    VOLUME = {51},
      YEAR = {1929},
    NUMBER = {2},
     PAGES = {305--328},
      ISSN = {0002-9327},
   MRCLASS = {DML},
  MRNUMBER = {1506719},
       DOI = {10.2307/2370712},
       URL = {https://doi.org/10.2307/2370712},
}

@book {rosbook,
    AUTHOR = {Rosales, J. C. and Garc\'{i}a-S\'{a}nchez, P. A.},
     TITLE = {Numerical semigroups},
    SERIES = {Developments in Mathematics},
    VOLUME = {20},
 PUBLISHER = {Springer, New York},
      YEAR = {2009},
     PAGES = {x+181},
      ISBN = {978-1-4419-0159-0},
   MRCLASS = {20M14},
  MRNUMBER = {2549780},
       DOI = {10.1007/978-1-4419-0160-6},
       URL = {https://doi.org/10.1007/978-1-4419-0160-6},
}

@book {Epple,
    AUTHOR = {Epple, Moritz},
     TITLE = {Die {E}ntstehung der {K}notentheorie},
      NOTE = {Kontexte und Konstruktionen einer modernen mathematischen
              Theorie. [Contexts and constructions of a modern mathematical
              theory]},
 PUBLISHER = {Friedr. Vieweg \& Sohn, Braunschweig},
      YEAR = {1999},
     PAGES = {xvi+449},
      ISBN = {3-528-06787-X},
   MRCLASS = {01-02 (01A50 01A55 01A60 57-03 57M25)},
  MRNUMBER = {1716305},
MRREVIEWER = {G. Burde},
       DOI = {10.1007/978-3-322-80295-8},
       URL = {https://doi.org/10.1007/978-3-322-80295-8},
}

@article {CDGduke,
    AUTHOR = {Campillo, A. and Delgado, F. and Gusein-Zade, S. M.},
     TITLE = {The {A}lexander polynomial of a plane curve singularity via the ring of functions on it},
   JOURNAL = {Duke Math. J.},
  FJOURNAL = {Duke Mathematical Journal},
    VOLUME = {117},
      YEAR = {2003},
    NUMBER = {1},
     PAGES = {125--156},
      ISSN = {0012-7094},
   MRCLASS = {14H20 (32S50)},
  MRNUMBER = {1962784},
MRREVIEWER = {Daniel Matei},
       DOI = {10.1215/S0012-7094-03-11712-3},
       URL = {https://doi.org/10.1215/S0012-7094-03-11712-3},
}

@article {Delgmanuscripta1,
    AUTHOR = {Delgado de la Mata, F\'{e}lix},
     TITLE = {The semigroup of values of a curve singularity with several
              branches},
   JOURNAL = {Manuscripta Math.},
  FJOURNAL = {Manuscripta Mathematica},
    VOLUME = {59},
      YEAR = {1987},
    NUMBER = {3},
     PAGES = {347--374},
      ISSN = {0025-2611},
   MRCLASS = {14H20},
  MRNUMBER = {909850},
MRREVIEWER = {Joan Elias},
       DOI = {10.1007/BF01174799},
       URL = {https://doi.org/10.1007/BF01174799},
}

@article {Delgmanuscripta2,
    AUTHOR = {Delgado de la Mata, F\'{e}lix},
     TITLE = {Gorenstein curves and symmetry of the semigroup of values},
   JOURNAL = {Manuscripta Math.},
  FJOURNAL = {Manuscripta Mathematica},
    VOLUME = {61},
      YEAR = {1988},
    NUMBER = {3},
     PAGES = {285--296},
      ISSN = {0025-2611},
   MRCLASS = {14H20 (13H10 13J10)},
  MRNUMBER = {949819},
MRREVIEWER = {J\"{u}rgen Herzog},
       DOI = {10.1007/BF01258440},
       URL = {https://doi.org/10.1007/BF01258440},
}

@article {CDGextended,
    AUTHOR = {Guseu{i}n-Zade, S. M. and Del\textquotesingle gado, F. and Kampil\textquotesingle o, A.},
     TITLE = {Extended semigroup of a plane curve singularity},
   JOURNAL = {Tr. Mat. Inst. Steklova},
  FJOURNAL = {Trudy Matematicheskogo Instituta Imeni V. A. Steklova},
    VOLUME = {221},
      YEAR = {1998},
    NUMBER = {},
     PAGES = {149--167},
      ISSN = {0371-9685},
   MRCLASS = {32S25 (14E15 14H20 32S40)},
  MRNUMBER = {1683692},
MRREVIEWER = {Aleksandr G. Aleksandrov},
}

@book {EN,
    AUTHOR = {Eisenbud, David and Neumann, Walter},
     TITLE = {Three-dimensional link theory and invariants of plane curve
              singularities},
    SERIES = {Annals of Mathematics Studies},
    VOLUME = {110},
 PUBLISHER = {Princeton University Press, Princeton, NJ},
      YEAR = {1985},
     PAGES = {vii+173},
      ISBN = {0-691-08380-0; 0-691-08381-9},
   MRCLASS = {57M25 (14H20 57S15)},
  MRNUMBER = {817982},
MRREVIEWER = {Christopher W. Stark},
}

@article {SW,
    AUTHOR = {Sumners, D. W. and Woods, J. M.},
     TITLE = {The monodromy of reducible plane curves},
   JOURNAL = {Invent. Math.},
  FJOURNAL = {Inventiones Mathematicae},
    VOLUME = {40},
      YEAR = {1977},
    NUMBER = {2},
     PAGES = {107--141},
      ISSN = {0020-9910},
   MRCLASS = {57C45 (14B05)},
  MRNUMBER = {458440},
MRREVIEWER = {Richard Randell},
       DOI = {10.1007/BF01390342},
       URL = {https://doi.org/10.1007/BF01390342},
}

@article {Zar32,
    AUTHOR = {Zariski, Oscar},
     TITLE = {On the {T}opology of {A}lgebroid {S}ingularities},
   JOURNAL = {Amer. J. Math.},
  FJOURNAL = {American Journal of Mathematics},
    VOLUME = {54},
      YEAR = {1932},
    NUMBER = {3},
     PAGES = {453--465},
      ISSN = {0002-9327},
   MRCLASS = {DML},
  MRNUMBER = {1507926},
       DOI = {10.2307/2370887},
       URL = {https://doi.org/10.2307/2370887},
}

@incollection {Delgadoari,
    AUTHOR = {Delgado de la Mata, F\'{e}lix},
     TITLE = {An arithmetical factorization for the critical point set of
              some map germs from {${\bf C}^2$} to {${\bf C}^2$}},
 BOOKTITLE = {Singularities ({L}ille, 1991)},
    SERIES = {London Math. Soc. Lecture Note Ser.},
    VOLUME = {201},
     PAGES = {61--100},
 PUBLISHER = {Cambridge Univ. Press, Cambridge},
      YEAR = {1994},
   MRCLASS = {32S05 (14H20)},
  MRNUMBER = {1295072},
MRREVIEWER = {Arkadiusz Pl oski},
}

@article {Brauner28,
    AUTHOR = {Brauner, Karl},
     TITLE = {Das {V}erhalten der {F}unktionen in der {U}mgebung ihrer
              {V}erzweigungsstellen},
   JOURNAL = {Abh. Math. Sem. Univ. Hamburg},
  FJOURNAL = {Abhandlungen aus dem Mathematischen Seminar der Universit\"{a}t
              Hamburg},
    VOLUME = {6},
      YEAR = {1928},
    NUMBER = {1},
     PAGES = {1--55},
      ISSN = {0025-5858},
   MRCLASS = {DML},
  MRNUMBER = {3069487},
       DOI = {10.1007/BF02940600},
       URL = {https://doi.org/10.1007/BF02940600},
}

@article {Burau33,
    AUTHOR = {Burau, Werner},
     TITLE = {Kennzeichnung der {S}chlauchknoten},
   JOURNAL = {Abh. Math. Sem. Univ. Hamburg},
  FJOURNAL = {Abhandlungen aus dem Mathematischen Seminar der Universit\"{a}t
              Hamburg},
    VOLUME = {9},
      YEAR = {1933},
    NUMBER = {1},
     PAGES = {125--133},
      ISSN = {0025-5858},
   MRCLASS = {DML},
  MRNUMBER = {3069587},
       DOI = {10.1007/BF02940635},
       URL = {https://doi.org/10.1007/BF02940635},
}

@article {Bayer,
    AUTHOR = {Bayer, Valmecir},
     TITLE = {Semigroup of two irreducible algebroid plane curves},
   JOURNAL = {Manuscripta Math.},
  FJOURNAL = {Manuscripta Mathematica},
    VOLUME = {49},
      YEAR = {1985},
    NUMBER = {3},
     PAGES = {207--241},
      ISSN = {0025-2611},
   MRCLASS = {14H20 (14B05)},
  MRNUMBER = {777126},
MRREVIEWER = {Sherwood Washburn},
       DOI = {10.1007/BF01215247},
       URL = {https://doi.org/10.1007/BF01215247},
}

@article {Garcia,
    AUTHOR = {Garc\'{i}a, Arnaldo},
     TITLE = {Semigroups associated to singular points of plane curves},
   JOURNAL = {J. Reine Angew. Math.},
  FJOURNAL = {Journal f\"{u}r die Reine und Angewandte Mathematik. [Crelle's
              Journal]},
    VOLUME = {336},
      YEAR = {1982},
     PAGES = {165--184},
      ISSN = {0075-4102},
   MRCLASS = {14H20},
  MRNUMBER = {671326},
MRREVIEWER = {Author's review},
       DOI = {10.1515/crll.1982.336.165},
       URL = {https://doi.org/10.1515/crll.1982.336.165},
}

@incollection {Waldi,
    AUTHOR = {Waldi, Rolf},
     TITLE = {{W}ertehalbgruppe und {S}ingularit\"{a}t einer ebenen algebraischen {K}urve},
    SERIES = {Dissertation},
     PAGES = {45--58, 103},
 PUBLISHER = {Universit\"{a}t Regensburg},
      YEAR = {1972},
}

@article {CDKmanuscr,
    AUTHOR = {Campillo, A. and Delgado, F. and Kiyek, K.},
     TITLE = {Gorenstein property and symmetry for one-dimen\-sional local
              {C}ohen-{M}acaulay rings},
   JOURNAL = {Manuscripta Math.},
  FJOURNAL = {Manuscripta Mathematica},
    VOLUME = {83},
      YEAR = {1994},
    NUMBER = {3-4},
     PAGES = {405--423},
      ISSN = {0025-2611},
   MRCLASS = {13H10},
  MRNUMBER = {1277539},
MRREVIEWER = {Susan Williamson},
       DOI = {10.1007/BF02567623},
       URL = {https://doi.org/10.1007/BF02567623},
}

@article {Torres,
    AUTHOR = {Torres, Guillermo},
     TITLE = {On the {A}lexander polynomial},
   JOURNAL = {Ann. of Math. (2)},
  FJOURNAL = {Annals of Mathematics. Second Series},
    VOLUME = {57},
      YEAR = {1953},
     PAGES = {57--89},
      ISSN = {0003-486X},
   MRCLASS = {56.0X},
  MRNUMBER = {52104},
MRREVIEWER = {R. H. Fox},
       DOI = {10.2307/1969726},
       URL = {https://doi.org/10.2307/1969726},
}

@article {Alexander,
    AUTHOR = {Alexander, J. W.},
     TITLE = {Topological invariants of knots and links},
   JOURNAL = {Trans. Amer. Math. Soc.},
  FJOURNAL = {Transactions of the American Mathematical Society},
    VOLUME = {30},
      YEAR = {1928},
    NUMBER = {2},
     PAGES = {275--306},
      ISSN = {0002-9947},
   MRCLASS = {57M25},
  MRNUMBER = {1501429},
       DOI = {10.2307/1989123},
       URL = {https://doi.org/10.2307/1989123},
}

@article {Burau34,
    AUTHOR = {Burau, Werner},
     TITLE = {Kennzeichnung der schlauchverkettungen},
   JOURNAL = {Abh. Math. Sem. Univ. Hamburg},
  FJOURNAL = {Abhandlungen aus dem Mathematischen Seminar der Universit\"{a}t
              Hamburg},
    VOLUME = {10},
      YEAR = {1934},
    NUMBER = {1},
     PAGES = {285--297},
      ISSN = {0025-5858},
   MRCLASS = {DML},
  MRNUMBER = {3069631},
       DOI = {10.1007/BF02940680},
       URL = {https://doi.org/10.1007/BF02940680},
}

@article {MFjpaa,
    AUTHOR = {Moyano-Fern\'{a}ndez, Julio Jos\'{e}},
     TITLE = {Poincar\'{e} series for curve singularities and its behaviour
              under projections},
   JOURNAL = {J. Pure Appl. Algebra},
  FJOURNAL = {Journal of Pure and Applied Algebra},
    VOLUME = {219},
      YEAR = {2015},
    NUMBER = {6},
     PAGES = {2449--2462},
      ISSN = {0022-4049},
   MRCLASS = {14H20 (11G20 32S10)},
  MRNUMBER = {3299740},
MRREVIEWER = {Caterina Cumino},
       DOI = {10.1016/j.jpaa.2014.09.009},
       URL = {https://doi.org/10.1016/j.jpaa.2014.09.009},
}

@article {Delormegluing,
    AUTHOR = {Delorme, Charles},
     TITLE = {Sous-mono\"{i}des d'intersection compl\`ete de {$N.$}},
   JOURNAL = {Ann. Sci. \'{E}cole Norm. Sup. (4)},
  FJOURNAL = {Annales Scientifiques de l'\'{E}cole Normale Sup\'{e}rieure. Quatri\`eme
              S\'{e}rie},
    VOLUME = {9},
      YEAR = {1976},
    NUMBER = {1},
     PAGES = {145--154},
      ISSN = {0012-9593},
   MRCLASS = {14M10},
  MRNUMBER = {407038},
MRREVIEWER = {Henry C. Pinkham},
       URL = {http://www.numdam.org/item?id=ASENS_1976_4_9_1_145_0},
}

@incollection {Webertop,
    AUTHOR = {Weber, Claude},
     TITLE = {On the topology of singularities},
 BOOKTITLE = {Singularities {II}},
    SERIES = {Contemp. Math.},
    VOLUME = {475},
     PAGES = {217--251},
 PUBLISHER = {Amer. Math. Soc., Providence, RI},
      YEAR = {2008},
   MRCLASS = {14-03 (14B05 14H20 32S55 57M25)},
  MRNUMBER = {2454369},
MRREVIEWER = {Tomohiro Okuma},
       DOI = {10.1090/conm/475/09285},
       URL = {https://doi.org/10.1090/conm/475/09285},
}

@inproceedings {LeCarousels,
    AUTHOR = {Tr\'{a}ng, L\^{e} D\~{u}ng},
     TITLE = {Plane curve singularities and carousels},
 BOOKTITLE = {Proceedings of the {I}nternational {C}onference in {H}onor of
              {F}r\'{e}d\'{e}ric {P}ham ({N}ice, 2002)},
   JOURNAL = {Ann. Inst. Fourier (Grenoble)},
  FJOURNAL = {Universit\'{e} de Grenoble. Annales de l'Institut Fourier},
    VOLUME = {53},
      YEAR = {2003},
    NUMBER = {4},
     PAGES = {1117--1139},
      ISSN = {0373-0956},
   MRCLASS = {14H20 (32S55 57M25)},
  MRNUMBER = {2033511},
MRREVIEWER = {Lee Rudolph},
       URL = {http://aif.cedram.org/item?id=AIF_2003__53_4_1117_0},
}

@article {Shinohara1,
    AUTHOR = {Shinohara, Y. and Sumners, D. W.},
     TITLE = {Homology invariants of cyclic coverings with application to
              links},
   JOURNAL = {Trans. Amer. Math. Soc.},
  FJOURNAL = {Transactions of the American Mathematical Society},
    VOLUME = {163},
      YEAR = {1972},
     PAGES = {101--121},
      ISSN = {0002-9947},
   MRCLASS = {55.20},
  MRNUMBER = {284999},
MRREVIEWER = {J. M. McPherson},
       DOI = {10.2307/1995710},
       URL = {https://doi.org/10.2307/1995710},
}

@article {Shinohara2,
    AUTHOR = {Shinohara, Yaichi},
     TITLE = {Higher dimensional knots in tubes},
   JOURNAL = {Trans. Amer. Math. Soc.},
  FJOURNAL = {Transactions of the American Mathematical Society},
    VOLUME = {161},
      YEAR = {1971},
     PAGES = {35--49},
      ISSN = {0002-9947},
   MRCLASS = {57.20 (55.00)},
  MRNUMBER = {287559},
MRREVIEWER = {D. W. L. Sumners},
       DOI = {10.2307/1995926},
       URL = {https://doi.org/10.2307/1995926},
}

@article {ACampo1,
    AUTHOR = {A'Campo, Norbert},
     TITLE = {Sur la monodromie des singularit\'{e}s isol\'{e}es d'hypersurfaces
              complexes},
   JOURNAL = {Invent. Math.},
  FJOURNAL = {Inventiones Mathematicae},
    VOLUME = {20},
      YEAR = {1973},
     PAGES = {147--169},
      ISSN = {0020-9910},
   MRCLASS = {32C40 (14D05)},
  MRNUMBER = {338436},
MRREVIEWER = {A. H. Wallace},
       DOI = {10.1007/BF01404063},
       URL = {https://doi.org/10.1007/BF01404063},
}

@article {ACampo2,
    AUTHOR = {A'Campo, Norbert},
     TITLE = {La fonction z\^{e}ta d'une monodromie},
   JOURNAL = {Comment. Math. Helv.},
  FJOURNAL = {Commentarii Mathematici Helvetici},
    VOLUME = {50},
      YEAR = {1975},
     PAGES = {233--248},
      ISSN = {0010-2571},
   MRCLASS = {14B05 (14D05 32C40)},
  MRNUMBER = {371889},
MRREVIEWER = {Richard Randell},
       DOI = {10.1007/BF02565748},
       URL = {https://doi.org/10.1007/BF02565748},
}

@article {CDG15,
    AUTHOR = {Campillo, A. and Delgado, F. and Gusein-Zade, S. M.},
     TITLE = {An equivariant {P}oincar\'{e} series of filtrations and monodromy
              zeta functions},
   JOURNAL = {Rev. Mat. Complut.},
  FJOURNAL = {Revista Matem\'{a}tica Complutense},
    VOLUME = {28},
      YEAR = {2015},
    NUMBER = {2},
     PAGES = {449--467},
      ISSN = {1139-1138},
   MRCLASS = {14B05 (13A18)},
  MRNUMBER = {3344086},
MRREVIEWER = {Stanisl aw Tadeusz Janeczko},
       DOI = {10.1007/s13163-014-0160-8},
       URL = {https://doi.org/10.1007/s13163-014-0160-8},
}

@article {Yamamoto,
    AUTHOR = {Yamamoto, Makoto},
     TITLE = {Classification of isolated algebraic singularities by their
              {A}lexander polynomials},
   JOURNAL = {Topology},
  FJOURNAL = {Topology. An International Journal of Mathematics},
    VOLUME = {23},
      YEAR = {1984},
    NUMBER = {3},
     PAGES = {277--287},
      ISSN = {0040-9383},
   MRCLASS = {57M25 (14B05 32B30 57Q45)},
  MRNUMBER = {770564},
MRREVIEWER = {N. V. Ivanov},
       DOI = {10.1016/0040-9383(84)90011-9},
       URL = {https://doi.org/10.1016/0040-9383(84)90011-9},
}

@incollection {Mams,
    AUTHOR = {Moyano-Fern\'{a}ndez, Julio Jos\'{e}},
     TITLE = {Generalized {P}oincar\'{e} series for plane curve singularities},
 BOOKTITLE = {{$p$}-adic analysis, arithmetic and singularities},
    SERIES = {Contemp. Math.},
    VOLUME = {778},
     PAGES = {25--69},
 PUBLISHER = {Amer. Math. Soc., [Providence], RI},
      YEAR = {[2022] \copyright 2022},
   MRCLASS = {14H20 (32S99)},
  MRNUMBER = {4419242},
       DOI = {10.1090/conm/778/15654},
       URL = {https://doi.org/10.1090/conm/778/15654},
}

@article {Seifert50,
    AUTHOR = {Seifert, H.},
     TITLE = {On the homology invariants of knots},
   JOURNAL = {Quart. J. Math. Oxford Ser. (2)},
  FJOURNAL = {The Quarterly Journal of Mathematics. Oxford. Second Series},
    VOLUME = {1},
      YEAR = {1950},
     PAGES = {23--32},
      ISSN = {0033-5606},
   MRCLASS = {56.0X},
  MRNUMBER = {35436},
MRREVIEWER = {R. H. Fox},
       DOI = {10.1093/qmath/1.1.23},
       URL = {https://doi.org/10.1093/qmath/1.1.23},
}

@article {waldhausenI,
    AUTHOR = {Waldhausen, Friedhelm},
     TITLE = {Eine {K}lasse von {$3$}-dimensionalen {M}annigfaltigkeiten.
              {I}},
   JOURNAL = {Invent. Math.},
  FJOURNAL = {Inventiones Mathematicae},
    VOLUME = {3},
      YEAR = {1967},
     PAGES = {308--333},
      ISSN = {0020-9910},
   MRCLASS = {57.31},
  MRNUMBER = {235576},
MRREVIEWER = {D. B. A. Epstein},
       DOI = {10.1007/BF01402956},
       URL = {https://doi.org/10.1007/BF01402956},
}

@article {waldhausenII,
    AUTHOR = {Waldhausen, Friedhelm},
     TITLE = {Eine {K}lasse von {$3$}-dimensionalen {M}annigfaltigkeiten.
              {II}},
   JOURNAL = {Invent. Math.},
  FJOURNAL = {Inventiones Mathematicae},
    VOLUME = {4},
      YEAR = {1967},
     PAGES = {87--117},
      ISSN = {0020-9910},
   MRCLASS = {57.31},
  MRNUMBER = {235576},
MRREVIEWER = {D. B. A. Epstein},
       DOI = {10.1007/BF01425244},
       URL = {https://doi.org/10.1007/BF01425244},
}

@article {Thurston1,
    AUTHOR = {Thurston, William P.},
     TITLE = {Three-dimensional manifolds, {K}leinian groups and hyperbolic
              geometry},
   JOURNAL = {Bull. Amer. Math. Soc. (N.S.)},
  FJOURNAL = {American Mathematical Society. Bulletin. New Series},
    VOLUME = {6},
      YEAR = {1982},
    NUMBER = {3},
     PAGES = {357--381},
      ISSN = {0273-0979},
   MRCLASS = {57N10 (20H15 30F40 57M35 57M40 57S17)},
  MRNUMBER = {648524},
MRREVIEWER = {Klaus Johannson},
       DOI = {10.1090/S0273-0979-1982-15003-0},
       URL = {https://doi.org/10.1090/S0273-0979-1982-15003-0},
}

@book {Thurston2,
    AUTHOR = {Thurston, William P.},
     TITLE = {Three-dimensional geometry and topology. {V}ol. 1},
    SERIES = {Princeton Mathematical Series},
    VOLUME = {35},
      NOTE = {Edited by Silvio Levy},
 PUBLISHER = {Princeton University Press, Princeton, NJ},
      YEAR = {1997},
     PAGES = {x+311},
      ISBN = {0-691-08304-5},
   MRCLASS = {57M50 (53A35 57M25 57M60 57N10)},
  MRNUMBER = {1435975},
MRREVIEWER = {Athanase Papadopoulos},
}

@book {Johannson,
    AUTHOR = {Johannson, Klaus},
     TITLE = {Homotopy equivalences of {$3$}-manifolds with boundaries},
    SERIES = {Lecture Notes in Mathematics},
    VOLUME = {761},
 PUBLISHER = {Springer, Berlin},
      YEAR = {1979},
     PAGES = {ii+303},
      ISBN = {3-540-09714-7},
   MRCLASS = {57N10},
  MRNUMBER = {551744},
MRREVIEWER = {John Hempel},
}

@article {JacoShalen,
    AUTHOR = {Jaco, William H. and Shalen, Peter B.},
     TITLE = {Seifert fibered spaces in {$3$}-manifolds},
   JOURNAL = {Mem. Amer. Math. Soc.},
  FJOURNAL = {Memoirs of the American Mathematical Society},
    VOLUME = {21},
      YEAR = {1979},
    NUMBER = {220},
     PAGES = {viii+192},
      ISSN = {0065-9266},
   MRCLASS = {57N10},
  MRNUMBER = {539411},
MRREVIEWER = {Hugh M. Hilden},
       DOI = {10.1090/memo/0220},
       URL = {https://doi.org/10.1090/memo/0220},
}

@article {Thurston3,
    AUTHOR = {Thurston, William P.},
     TITLE = {Hyperbolic structures on {$3$}-manifolds. {I}. {D}eformation
              of acylindrical manifolds},
   JOURNAL = {Ann. of Math. (2)},
  FJOURNAL = {Annals of Mathematics. Second Series},
    VOLUME = {124},
      YEAR = {1986},
    NUMBER = {2},
     PAGES = {203--246},
      ISSN = {0003-486X},
   MRCLASS = {57N10},
  MRNUMBER = {855294},
MRREVIEWER = {G. Peter Scott},
       DOI = {10.2307/1971277},
       URL = {https://doi.org/10.2307/1971277},
}

@article {apery,
    AUTHOR = {Ap\'{e}ry, Roger},
     TITLE = {Sur les branches superlin\'{e}aires des courbes alg\'{e}briques},
   JOURNAL = {C. R. Acad. Sci. Paris},
  FJOURNAL = {Comptes Rendus Hebdomadaires des S\'{e}ances de l'Acad\'{e}mie des
              Sciences},
    VOLUME = {222},
      YEAR = {1946},
     PAGES = {1198--1200},
      ISSN = {0001-4036},
   MRCLASS = {14.0X},
  MRNUMBER = {17942},
MRREVIEWER = {J. G. Semple},
}

@book {BurdeKnots,
	AUTHOR = {Burde, Gerhard and Zieschang, Heiner},
	TITLE = {Knots},
	SERIES = {De Gruyter Studies in Mathematics},
	VOLUME = {5},
	EDITION = {Second},
	PUBLISHER = {Walter de Gruyter \& Co., Berlin},
	YEAR = {2003},
	PAGES = {xii+559},
	ISBN = {3-11-017005-1},
	MRCLASS = {57M25 (57-02)},
	MRNUMBER = {1959408},
}
\end{document}